\documentclass[10pt]{amsart}

\usepackage{color,graphicx,amssymb}
\usepackage{amsthm}
\usepackage{dsfont}
\usepackage{tikz-cd}

\theoremstyle{definition}

\newtheorem{theorem}{Theorem}[section]
\newtheorem{definition}[theorem]{Definition}
\newtheorem{lemma}[theorem]{Lemma}
\newtheorem{corollary}[theorem]{Corollary}
\newtheorem{proposition}[theorem]{Proposition}
\newtheorem{claim}[theorem]{Claim}
\newtheorem{example}[theorem]{Example}

\newtheorem{remark}[theorem]{Remark}
\newtheorem{problem}[theorem]{Problem}
\newtheorem{acknowledgments}[theorem]{Acknowledgments}

\DeclareMathOperator{\Con}{Con}

\DeclareMathOperator{\Aut}{Aut}
\DeclareMathOperator{\Hom}{Hom}
\DeclareMathOperator{\Mod}{Mod}

\DeclareMathOperator{\id}{id}

\DeclareMathOperator{\im}{im}
\DeclareMathOperator{\ar}{ar}

\DeclareMathOperator{\Id}{Id}

\begin{document}

\title{Extensions of Multilinear Module Expansions}
\author{Alexander Wires}
\address{School of Mathematical Sciences, University of Electronic Science and Technology of China, Chengdu,611731, Sichuan, PRC}
\email{awires81@uestc.edu.cn}

\begin{abstract}
We consider the deconstruction/reconstruction of extensions in varieties of algebras which are modules expanded by multilinear operators. The parametrization of extensions determined by abelian ideals with unary actions agrees with the previous development of extensions realizing affine datum in arbitrary varieties of universal algebras. We establish a Wells-type theorem which, for a fixed affine ideal, characterizes those ideal-preserving derivations of a group-trivial extension as a Lie algebra extension of the compatible pairs of derivations of the datum algebras by the cohomological derivations of the datum. For these varieties, we establish a low-dimensional Hochschild-Serre exact sequence associated to an arbitrary extension equipped with an additional affine action.  
\end{abstract}

\maketitle


\section{Introduction}
\vspace{0.3cm}

In this manuscript, we examine the parameters characterizing arbitrary extensions in varieties of $R$-modules expanded by multilinear operations. These algebras can be seen as a special case of Higgin's \cite{higgins} groups with multiple operators formalisms but only slightly broader than Kurosh's formalism of $\Omega$-algebras since we consider modules rather than vector spaces. We describe a low-dimensional cohomology which is replete with the expected machinery of 2-cocycles, 2-coboundaries, derivations and stabilizing automorphisms in order for second-cohomology to classify extensions; in addition, we define a notion of principal derivation connected to special stabilizing automorphisms which determines a first cohomology which is not always defined for cohomologies of nonassociative or general structures. Those extensions realizing affine datum as defined in Wires \cite{wiresI} can be characterized by datum with abelian ideals and unary actions only. The 2-cocycle identity for a subvariety is formalized by a satisfaction relation associated to a multisorted signature formed by the action terms and factor sets; in this way, the variety membership (or equational theory) of extensions is incorporated into cohomology as a parameter. This affords a Galois connection between the lattice of subvarieties and second-cohomology. The notion of a derivation (as opposed to cohomological derivation of datum) makes sense for algebras in these varieties and so we prove a Wells-type theorem characterizing kernel-preserving derivations of a group-trivial extension realizing affine datum by a Lie algebra extension of derivations. A similar result for automorphisms is established in Wires \cite[Thm 1.4]{wiresIII}. Since we are able to define a first cohomology group for any variety of multilinear module expansions, we establish a low-dimension Hochschild-Serre (or inflation/deflation) exact sequence associated to a general (nonabelian) extension equipped with an additional affine action.

The topic and literature on the cohomology or parametrization of extensions in various varieties of modules expanded by bilinear operations is surely too vast to give adequate justice here so those references cited reflect only the author's limited knowledge - mostly to more recent papers. The work parametrizing nonabelian extensions in the present manuscript recovers in a uniform manner the cohomological classification of extensions previously developed for rings (Everett \cite{rings}), associative algebras (Hochschild \cite{hoch}, Agore and Militaru \cite{agore}), Lie algebras (Inassaridze, Khmaladze and Ladra \cite{lie}), Liebniz algebras (Casas, Khmaladze and Ladra \cite{leibniz}), dendriform and bilinear Rota-Baxter algebras (Das and Rathee \cite{das}), Lie-Yamaguti algebras ( Yamaguti \cite{yamaguti}) or conformal algebras (Bakalov, Kac and Voronov\cite{conform1}, Hou and Zhao \cite{conform2}, Smith \cite{smith1}) to give just a few well-studied examples. Here we are able to generalize to any equational class of $R$-modules expanded by any set of multilinear operations. Our Theorem~\ref{thm:HSexact} extends to general extensions in our particular varieties the Hochschild-Serre exact sequence established in Wires \cite{wiresII} for central extensions with an idempotent element in varieties with a difference term. Two applications for the Hochschild-Serre exact sequence (as outlined for groups in Karpilovsky \cite{karpil}) is in the development of the Schur multiplier and its characterization by relative commutators in free presentation, and the connection with perfect algebras (which for groups can be read in Milnor \cite{milnor}). For varieties of multilinear expansions of $R$-modules, the analogous results are already recovered by specialization of the more general results in \cite{wiresII}; for examples of Schur-type results in recent work on some varieties of interest which would fall under the general umbrella of \cite{wiresII}, we may look at Batten \cite{batten1} for Lie algebras, Elyse \cite{elyse1} and Mainellis \cite{mainellis1} for Leibniz algebras and Mainellis \cite{mainellis2,mainellis3} for diassociative algebras. The inspiration for Theorem~\ref{thm:derivations} on kernel-preserving derivations in arbitrary varieties of multilinear module expansions can be found in Hou and Zhao \cite{conform2} where a similar result is established for abelian extensions of associative conformal algebras.

The principal motivation for the present manuscript derives from the work in \cite{wiresI,wiresII} for extensions realizing affine datum in arbitrary varieties of universal algebras. We would like to consider the parametrization of general or ``nonabelian'' extensions of universal algebras and develop proper interpretations for any possible higher cohomologies with general and affine datum. Varieties of $R$-modules expanded by multilinear operations form a particularly well-behaved class of Mal'cev varieties in which the generators of the commutator relations can be given a manageable form; therefore, these classes of algebras are a good place to start. The abelian group reduct and multilinearity of the higher arity operations appears to simplify structure and calculations but still allows for nonabelian extensions. These varieties may serve as a test case or set of examples for continuing the work from \cite{wiresI,wiresII}. In light of Theorem~\ref{thm:HSexact} and \cite{wiresII}, we think there is a good argument to be made that a fair amount of the compiled theory and results presented in Karpilovsky's book \cite{karpil} can be successfully developed for these varieties.

We do not consider higher cohomologies at all in this manuscript. Since we consider nonabelian extensions in arbitrary varieties (so any equational axiomatization), we find it difficult to fit the present development into the chain complex framework of modern cohomology inspired by algebraic topology, even in the affine datum case. In our view, the 2-cocycle condition for groups which is determined by a coboundary map in a chain complex, in the case of general varieties is explicitly tied to the equational theory of extensions; that is, 2-cocycles are interpretations of certain multisorted signatures satisfying some equations in that mixed signature. Our personal intuition is that it is be possible to give a more explicit and concrete interpretation for higher cohomologies analogous to that given by Holt \cite{holt} for groups. Such an explicit description would involve equations among different levels of action terms which may by easier to directly modify in order to describe the ``correct'' multisorted structures for nonabelian cohomologies. In any case, all this deserves its own more focused future effort. We leave these questions for Section~\ref{section:6}.


\section{Preliminaries}\label{section:2}
\vspace{0.3cm}

Let $\mathcal V$ be a variety of $R$-modules for a fixed ring $R$ expanded by multilinear operations indexed by the set $F$; formally, the signature $\tau = \{+,-,0, r : r \in R \} \mathrel{\cup} F$ is the union of the signature $\{+, -, 0, r : r \in R \}$ of an $R$-module with $F$. Note the ring element $r \in R$ in the signature is interpreted in a module as the unary operation which is scalar multiplication by $r$. An algebra $M \in \mathcal V$ can be written in a compact but slightly incorrect notation as $M = \left\langle _{R}M,F^{M} \right\rangle$ where $_{R} M$ is an $R$-module and $F^{M} = \{f^{M}: f \in F \}$ is the part of the signature which names the multilinear operations; that is, for $f \in F$ with arity $\ar f=n$, we have for each coordinate $1 \leq i \leq n$
\[
f(x_{1},\ldots,r \cdot y + z,\ldots,x_{n}) = r \cdot f(x_{1},\ldots,y,\ldots,x_{n}) + f(x_{1},\ldots,z,\ldots,x_{n})
\]
for $r \in R$, $x_{i},y,z \in M$. If we let $_{R} \mathcal M_{F}$ denote the largest variety of $R$-modules expanded by multilinear operation named by $F$, then $\mathcal V \, \leq \,  _{R}\mathcal M_{F}$. Let $\Id \mathcal V$ denote the set of identities of the variety $\mathcal V$.

Congruences $\alpha \in \Con M$ are in bijective correspondence with \emph{ideals} $I \triangleleft M$ which are submodules of $M$ that are \emph{absorbing} for each of the multilinear operations; that is, for each $f \in F$ with $n=\ar f$, $f(x_{1},\ldots,x_{n}) \in I$ whenever some $x_{i} \in I$. For a congruence $\alpha \in \Con M$, the corresponding ideal $I_{\alpha}$ is the $\alpha$-class which contains $0 \in M$. For an ideal $I \triangleleft M$, the corresponding congruence is given by $\alpha_{I} = \{(a,b) \in M^{2}: a-b \in I \} \in \Con M$. For two ideals $I,J \triangleleft M$, the commutator $[I,J]$ is the smallest ideal containing all elements of the form $f(\vec{a})$ where $f$ is a multilinear operation and $\vec{a} \in J^{\ar f}$ with at least one $a_{i} \in I$ or $\vec{a} \in I^{\ar f}$ with at least one $a_{i} \in J$.

Given a function $f: X \rightarrow A$ and $\vec{x} \in X^{n}$ we write $f(\vec{x}):= (f(x_{1}),\ldots,f(x_{n}))$ for the tuple which is the evaluation of the function at each coordinate. We will have to write formulas which sum the values of higher arity functions on substitutions over different subsets of coordinates between fixed tuples. It will be convenient to introduce notational conventions for this. For $0 < n \in \mathds{N}$, we write $[n] = \{1,\ldots,n\}$ for the initial segment of positive integers and $[n]^{\ast} = \{ s \subseteq \{ 1,\ldots,n \}: 0 < |s| < n \}$ for the non-empty proper subsets of $[n]$. Given two tuples $\vec{q},\vec{p} \in Q^{n}$ and $s \in [n]^{\ast}$, we form the tuple $\vec{q}_{s} \left[ \vec{p} \right] \in Q^{n}$
\[
\vec{q}_{s} [\vec{p}] (i) =  
  \begin{cases}
   \vec{p}(i) & ,  i \in s \\
   \vec{q}(i) & , i \not\in s;
  \end{cases}
\]
that is, $\vec{q}_{s} [\vec{p}]$ is formed by substituting the values of the tuple $\vec{p}$ into the tuple $\vec{q}$ at those coordinates specified by $s$. Then for a function $f: Q^{n} \rightarrow M$, we can write the evaluation of the function on the constructed tuple in two ways as $f(\vec{q})_{s}[\vec{p}] = f \left( \vec{q}_{s} [\vec{p}] \right)$. Both ways will be convenient if we have to make additional substitutions of the coordinates over overlapping subsets of coordinates. For $\vec{x} \in Q^{n}$ and $s \in [n]^{\ast}$, $\vec{x}|_{s} \in Q^{s}$ is the tuple $\vec{x}$ restricted to the coordinates $s$.

For an abelian group $A$, suppose we have an operation $f: A^{n} \rightarrow A$ and a map $l: Q \rightarrow A$. Given $s \in [n]^{\ast}$, this defines a function $a(f,s) : Q^n \times A^{n} \rightarrow A$ by the rule $a(f,s)(\vec{q},\vec{a}) := f(l(\vec{q}))_{s}[\vec{a}]$; certainly, the function is independent of the coordinates of $\vec{q}$ in $s$ and independent of the coordinates of $\vec{a}$ outside of $s$. Then by fixing the tuples $\vec{q} \in Q^{n}, \vec{a} \in A^{n}$, we can sum
\begin{align}\label{eqn:1}
\sum_{s \in [n]^{\ast}} a(f,s)(\vec{q},\vec{a})
\end{align}
the values of the function on substitutions over all nonempty proper subsets of the coordinates. Of course, what really matters are the value of the functions over the possibly dependent coordinates so that we could naturally consider the function $a'(f,s): Q^{[n] - s} \times A^{s} \rightarrow A$ where $a'(f,s)(\vec{q}|_{[n]-s},\vec{a}|_{s}) = a(f,s)(\vec{q},\vec{a})$. Since we will later need to consider additional substitutions in our fomulas, we have found the form in Eq.(\ref{eqn:1}) more convenient. An exception is sometimes made for unary subsets $s=\{i\} \in [n]^{\ast}$; in this case, we may sometimes write $a(f,i)(q_{1},\ldots,q_{i-1},a_{i},q_{i+1},\ldots,q_{n})$ rather than $a(f,i)(\vec{q},\vec{a})$ when the choice of $\vec{q} \in Q^{n},\vec{a} \in A^{n}$ is understood.

In writing $a(f,s)$ we decided to take the viewpoint of substituting $\vec{a}$ into the $s$-coordinates of $f(l(\vec{q}))$ for the following reason: if $f$ is multilinear over a module structure on $A$, then for each $\vec{q} \in Q^{n}$, $a(f,s)(\vec{q},-)$ is multilinear over the restriction to the $s$-coordinates over $A$.

The commutator for the algebras in this manuscript is a specialization of the term-condition commutator applicable for universal algebras. A concise exposition in the case of congruence modular varieties can be found in McKenzie and Snow \cite{mckenziesnow}. We do not require an extensive knowledge of the term-condition commutator but the present manuscript is connected to the theory of extensions realizing affine datum expounded in \cite{wiresI,wiresII,wiresIII} - the closest point of contact occurs in Theorem~\ref{thm:affinecohom} and Theorem~\ref{thm:centcohom}. Let $A$ be an algebra and $\alpha \beta, \delta \in \Con A$. We write $C(\alpha, \beta ; \delta)$ and say $\alpha$ \emph{centralizes} $\beta$ \emph{modulo} $\delta$ if for all terms $t(\vec{x},\vec{y})$ and all tuples $\vec{a} \mathrel{\alpha} \vec{b}$, $\vec{c} \mathrel{\beta} \vec{d}$ we have the implication 
\begin{align*}
t(\vec{a},\vec{c}) \mathrel{\delta} t(\vec{a},\vec{d}) \rightarrow t(\vec{b},\vec{c}) \mathrel{\delta} t(\vec{b},\vec{d})
\end{align*}
The \emph{commutator} (\emph{term-condition commutator}) is defined as the intersection $[\alpha,\beta] := \bigcap \{\delta : C(\alpha, \beta ; \delta ) \}$. The short argument for the commutator in rings in \cite{mckenziesnow} can be easily adapted to show $[\alpha_{I}, \alpha_{J}] = \alpha_{[I,J]}$ for the algebras in the present manuscript; thus, the two notions of abelianess agree.


\section{Machinery}\label{section:3}
\vspace{0.3cm}

In this section, we develop the machinery around nonabelian second-cohomology . We begin with the category of extensions which can be realized by three-term exact sequences since congruences are ideal-determined.

\begin{definition}\label{def:extension}
Let $\mathcal V$ be a variety of $R$-modules expanded by multilinear operations. A pair of algebras $(Q,I)$ with $Q,I \in \mathcal V$ is called \emph{datum} in $\mathcal V$. An algebra $M$ is an \emph{extension} of datum $(Q,I)$ if there is an ideal $K \triangleleft M$ such that $K \approx I$ and $M/K \approx Q$.
\end{definition}

Let $\mathrm{EXT} \, \mathcal V$ be the category whose objects are three term exact sequences 
\begin{align}\label{eqn:exactseq}
0 \longrightarrow I \stackrel{i}{\longrightarrow} M \stackrel{\pi}{\longrightarrow} Q \longrightarrow 0
\end{align}
of algebras in $\mathcal V$ and the morphisms are triples $(\alpha,\lambda,\beta)$ of homomorphisms which yield commuting squares 
\[
\begin{tikzcd}
	0 & {I} & {M} & {Q} & 0 \\
	0 & {J} & {N} & {P} & 0
	\arrow[from=1-1, to=1-2]
	\arrow["{i}", from=1-2, to=1-3]
	\arrow["\alpha", from=1-2, to=2-2]
	\arrow["\pi", from=1-3, to=1-4]
	\arrow["\lambda", no head, from=1-3, to=2-3]
	\arrow[from=1-4, to=1-5]
	\arrow["\beta", from=1-4, to=2-4]
	\arrow[from=2-1, to=2-2]
	\arrow["{j}", from=2-2, to=2-3]
	\arrow["\rho", from=2-3, to=2-4]
	\arrow[from=2-4, to=2-5]
\end{tikzcd}
\]
between exact sequences. Composition of morphisms is affected coordinate-wise.

\begin{lemma}
Two extensions are isomorphic in the category $\mathrm{EXT} \, \mathcal V$ if and only if there is a morphism $(\alpha,\lambda,\beta)$ between the two exacts sequences where each component is an isomorphism. 
\end{lemma}

We can define different equivalence relations on the objects by restricting the type of morphisms which appear in the commuting squares. By composing with isomorphisms, any extension of datum $(Q,I)$ according to Definition~\ref{def:extension} yields an exact sequence in the form of Eq~\eqref{eqn:exactseq}. Our interest will be on reconstructing extension which are concretely built from the algebras $I$ and $Q$.

\begin{definition}\label{def:equivexactseq}
Two extensions $1 \longrightarrow I \stackrel{i}{\longrightarrow} A \stackrel{\pi}{\longrightarrow} Q \longrightarrow 1$ and $1 \longrightarrow I \stackrel{i}{\longrightarrow} M \stackrel{\pi'}{\longrightarrow} Q \longrightarrow 1$ are \emph{equivalent} if there is an isomorphism $(\id_{I},\lambda,\id_{Q})$ in $\mathrm{EXT} \, \mathcal V$ between them.
\end{definition}

The above definition defines an equivalence relation which refines isomorphism types on the objects of the category. Let $\mathrm{EXT}(Q,I)$ denote the set of all equivalence classes denoted by $[M]$ of extensions in the form of Eq~\eqref{eqn:exactseq}. We now examine how extensions of fixed datum $(Q,I)$ can be deconstructed in terms of actions terms and factor sets.

\begin{definition}
Let $\mathcal V$ be a variety of $R$-modules expanded by multilinear operations and $Q,I \in \mathcal V$. An \emph{action} $Q \ast I$ is a sequence of operations $ \ast = \{a(f,s): f \in F, s \in [\ar f]^{\ast} \}$ where $a(f,s):Q^{\ar f} \times I^{\ar f} \rightarrow I$.  
\end{definition}

\begin{definition}
Let $\tau$ be a signature for a variety $\mathcal V$ of $R$-modules expanded by multilinear operations $F$ and $(Q,I)$ be datum in $\mathcal V$. A 2-\emph{cocycle} for the datum $(Q,I)$ is a sequence $T = \{ T_{+}, T_{r}, T_{f}, a(f,s) : r \in R, f \in F, s \in [\ar f]^{\ast} \}$ where  
\begin{enumerate}
	
	\item[(T1)] $T_{+} : Q^{2} \rightarrow I$ such that $T_{+}(x,0)=T_{+}(0,x)=0$;
	
	\item[(T2)] $T_{r} : Q \rightarrow I$ for each $r \in R$ such that $T_{r}(0)=0$;
	
	\item[(T3)] $T_{f} : Q^{\ar f} \rightarrow I$ for $f \in F$ such that $T_{f}(\vec{x})=0$ whenever some $x_{i}=0$;
	
	\item[(T4)] $\{ a(f,s): f \in F, s \in [\ar f]^{\ast} \}$ is an action $Q \ast I$ such that for each $\vec{q} \in Q^{\ar f}$, $a(f,s)(\vec{q},\cdot)|_{s} \in \Hom_{R}(\otimes^{s} \, _{R}I ,\,  _{R}I)$ and $a(f,s)(\vec{q},\cdot) \equiv 0$ for any $\vec{q} \in Q^{\ar f}$ with $\vec{q}(i) = 0$ for some $i \not\in s$.

\end{enumerate}
\end{definition}

The symbols $\{ a(f,s): f \in F, s \in [\ar f]^{\ast} \}$ in a 2-cocycle will be called the \emph{action terms} of the 2-cocycle and the symbols $T_{+}, T_{r}$ and $T_{f}$ the \emph{factor-sets}. A 2-cocycle $T = \{ T_{+}, T_{r}, T_{f}, a(f,s) : r \in R, f \in F, s \in [\ar f]^{\ast} \}$ is a multisorted signature in two diffferent sorts; for example, one sort for the domains of the symbols $T_{+}$, $T_{r}$ and $T_{f}$ and another sort for their codomains, but the symbols $a(f,s)$ use both sorts for the domain. Note that the datum $(Q,I)$ consists of two actual algebras so we have interpretations of the signature $\tau^{I}$ and $\tau^{Q}$ giving the module and multilinear operations for the algebras $I$ and $Q$. A 2-cocycle $T$ for datum $(Q,I)$ is then an interpretation of the multisorted signature $T$ in the multisorted structure $\left\langle I \cup Q, \tau^{I}, \tau^{Q}, T \right\rangle$ which satisfies the properties specified in (T1) - (T4). The structure $\left\langle I \cup Q, \tau^{I}, \tau^{Q}, T \right\rangle$ has $\tau_{I} \cup \tau_{Q} \cup T$ as an appropriate multisorted signature where we have written $\tau_{I}$ for the operation symbols for the sort which have the algebra $I$ as our intended interpretation. In case where the full signature $\tau$ is not stated explicitly but we have named the set of multilinear operations by $F$, we may also informally write the intended structure as $\left\langle _{R} I \cup \, _{R} Q, F^{I}, F^{Q}, T \right\rangle$ where we are using the informal notation $_{R} I \mathrel{\cup} \, _{R} Q$ to denote the different $R$-module structures over the two sorts of $I$ and $Q$.

\begin{definition}
Let $\mathcal V$ be a variety of $R$-modules expanded by multilinear operations $\mathcal F$. Let be datum in $\mathcal V$ and $T$ a 2-cocycle for the datum. The algebra $I \rtimes_{T} Q$ is defined with universe $I \times Q$ and operations
\begin{enumerate}

	\item $\left\langle \, a,x \, \right\rangle + \left\langle \, b,y \, \right\rangle := \left\langle \, a+b + T_{+}(x,y), x+y \, \right\rangle$;
	
	\item $r \cdot \left\langle \, a,x \, \right\rangle := \left\langle \, r \cdot a + T_{r}(x), r \cdot x \, \right\rangle$ for $r \in R$;
	
	\item $F_f \left(\left\langle \, a_1, x_1 \, \right\rangle, \ldots, \left\langle \, a_n, x_n \, \right\rangle \right) := \left\langle \, f(\vec{a}) + \sum_{s \in [n]^{\ast}} a(f,s)(\vec{x},\vec{a}) + T_{f}(\vec{x}), f(\vec{x}) \, \right\rangle$ for $f \in F$ with $n = \ar f$.
	
\end{enumerate}
\end{definition}

Note there is an embedding $i: I \rightarrow I \rtimes_{T} Q$ given by $i(a):= \left\langle a,0 \right\rangle$ and the second-projection $p_{2}: I \rtimes_{T} Q \rightarrow Q$ is an extension over $Q$ with $i(I) = I_{\ker p_2}$. We would like to place a condition on the 2-cocycle $T$ so that $I \rtimes_{T} Q \in \mathcal V$ whenever $(Q,I)$ is datum in $\mathcal V$. This is done by writing a representation of the terms in $I \rtimes_{T} Q$.

\begin{lemma}\label{lem:semident}
Let $\mathcal V$ be a variety of $R$-modules expanded by multilinear operations $F$. Let $T$ be a 2-cocycle for the datum $(Q,I)$. For each term $t(\vec{x})$ in the signature $\tau$ of $\mathcal V$, there are terms $t^{\ast,T}, t^{\partial,T}$ in the multisorted signature $\tau_{I} \cup \tau_{Q} \cup T$ such that $I \rtimes_{T} Q \in \mathcal V$ if and only if $Q,I \in \mathcal V$ and $\left\langle I \cup Q, \tau^{I}, \tau^{Q}, T \right\rangle \vDash t^{\ast,T} + t^{\partial,T} = s^{\ast,T} + s^{\partial,T}$ for all $t=s \in \mathrm{Id} \mathcal V$.
\end{lemma}
\begin{proof}
By the recursive generation of terms, it can be shown that for each term $t(x_{1},\ldots,x_{n})$ in the signature $\tau$ of $\mathcal V$ there are terms $t^{\ast,T}$ and $t^{\partial,T}$ in the multisorted signature $\tau_{I} \cup \tau_{Q} \cup T$ such that 
\begin{align}\label{eqn:55}
F_{t}(\vec{a},\vec{q}) = \left\langle \, t^{I}(\vec{a}) + t^{\ast,T}(\vec{a},\vec{q}) + t^{\partial,T}(\vec{a},\vec{q}), t^{Q}(\vec{q}) \, \right\rangle
\end{align}
computed in the algebra $I \rtimes_{T} Q$ where we have written $\left( \vec{a}, \vec{q} \right) = \left( \left\langle \, a_{1}, q_{1} \, \right\rangle, \ldots, \left\langle \, a_{n}, q_{n} \, \right\rangle  \right)$. The term $t^{\ast,T}$ is recursively constructed from only the module operations and the actions terms and $t^{\partial,T}$ is recursively constructed from the module operations and the operations in $T$ and must contain a factor-set in its composition tree. Informally, we use the module operations and multilinearity to expand $F_{t}(\vec{a},\vec{q})$ into a sum of various composed operations: all the summands that utilize a factor-set are summed together to form $t^{\partial,T}(\vec{a},\vec{q})$, all the parts that utlize action terms but no factor-sets are summed together to form $t^{\ast,T}(\vec{a},\vec{q})$, and the remaining parts must be just the interpretation of the term $t$ in the algebra $I$.

Assume $I \rtimes_{T} Q \in \mathcal V$. Since $Q$ is a quotient determined by ideal $i(I)$, we have $I,Q \in \mathcal V$. Then for $t=s \in \mathrm{Id} \mathcal V$ we have $t^{I} = s^{I}$ and $F_{t} = F_{s}$. Then Eqn (\ref{eqn:55}) implies $t^{\ast,T} + t^{\partial,T} = s^{\ast,T} + s^{\partial,T}$. Conversely, $I,Q \in \mathcal V$ implies $t^{I}=s^{I}$ and $t^{Q}=s^{Q}$ for $t=s \in \mathrm{Id} \mathcal V$. Then substituting this and $\left\langle I \cup Q, \tau^{I}, \tau^{Q}, T \right\rangle \vDash t^{\ast,T} + t^{\partial,T} = s^{\ast,T} + s^{\partial,T}$ into Eqn (\ref{eqn:55}) yields $F_{t}=F_{s}$; thus, $I \rtimes_{T} Q \in \mathcal V$.
\end{proof}

\begin{definition}\label{def:3}
Let $\tau$ be the signature of a variety $\mathcal V$ of $R$-modules expanded by multilinear operations and $(Q,I)$ datum in the same signature as $\mathcal V$. A 2-cocycle $T$ for the datum $(Q,I)$ is $\mathcal V$-\emph{compatible} if 
\[
\left\langle I \cup Q, \tau^{I}, \tau^{Q}, T \right\rangle \vDash t^{\ast,T} + t^{\partial,T} = s^{\ast,T} + s^{\partial,T} \quad \text{for all}  \quad t=s \in \mathrm{Id} \mathcal V.
\]
\end{definition}

Given an equation $t=s$ appropriate for a variety of multilinear expansions of $R$-modules, we refer to the equation $t^{\ast,T} + t^{\partial,T} = s^{\ast,T} + s^{\partial,T}$ in Definition~\ref{def:3} as the \emph{general 2-cocycle identity} for $T$ of the equation $t=s$. We may refer to the equation $t^{\ast,T} = s^{\ast,T}$ as the \emph{action identity} for compatibility with $t=s$ and the equation $t^{\partial,T} = s^{\partial,T}$ as the \emph{strict 2-cocycle identity} for $T$ of the equation $t=s$. With the help of Lemma~\ref{lem:semident}, we shall later see that for extensions realizing affine datum that includes a fixed $\mathcal V$-compatible action as part of the datum, then it is the strict 2-cocycle identity that determines membership of an extension in a variety.

Given an action $Q \ast I$, we can form a 2-cocycle $T^{\ast}$ for the datum $(Q,I)$ where the action terms of $T^{\ast}$ are the same as the operations in $\ast$ and the factor-sets are exactly zero.

\begin{definition}
Let $\mathcal V$ be a variety of $R$-modules expanded by multilinear operations and $(Q,I)$ datum in the same signature as $\mathcal V$. We say an action $Q \ast I$ is $\mathcal V$-\emph{compatible} if the 2-cocycle $T^{\ast}$ is $\mathcal V$-compatible.
\end{definition}

Since the factor-sets of $T^{\ast}$ are all zero, we must have $t^{\partial,T^{\ast}}=0$ for all terms $t$; therefore, an action  $Q \ast I$ is $\mathcal V$-compatible if and only if $t^{\ast,T^{\ast}} = s^{\ast,T^{\ast}}$ is true for all $t=s \in \mathrm{Id} \mathcal V$ if and only if $I \rtimes_{T^{\ast}} Q \in \mathcal V$. We will also allow ourselves the notational convenience to write $I \rtimes_{\ast} Q = I \rtimes_{T^{\ast}} Q$.

\begin{example}\label{ex:multiequations}
We illustrate by writing the 2-cocycle identities for $f \in F$ to be multilinear. Assume $n=\ar f$ and $r\in R$. The equations for linearity in the i-th coordinate are
\begin{align}\label{eqn:3}
f(x_{1},\ldots,rx_{i},\ldots,x_{n}) = r \cdot f(x_{1},\ldots,x_{n})
\end{align}
and
\begin{align}\label{eqn:4}
f(x_{1},\ldots,x_{i}+y_{i},\ldots,x_{n}) = f(x_{1},\ldots,x_{i},\ldots,x_{n}) + f(x_{1},\ldots,y_{i},\ldots,x_{n}).
\end{align}
To compute the corresponding general 2-cocycle identities of $T$ for datum $(Q,I)$ in $\mathcal V$, we evaluate the corresponding left and right-hand terms in $I \rtimes_{T} Q$. Since $I$ satisfies the identities in $\mathcal V$, we set the first-coordinates equal and delete the identity in the operations of $I$. For Eqn (\ref{eqn:3}), calculating the operations
\begin{align}
F_{f} \left(\left\langle a_{1},x_{1} \right\rangle, \ldots, r \cdot \left\langle a_{i},x_{i} \right\rangle  , \ldots, \left\langle a_{n},x_{n} \right\rangle \right)
\end{align}  
and 
\begin{align}
r \cdot F_{f} \left(\left\langle a_{1},x_{1} \right\rangle, \ldots, \left\langle a_{i},x_{i} \right\rangle  , \ldots, \left\langle a_{n},x_{n} \right\rangle \right)
\end{align}  
yields the 2-cocycle identity
\begin{align*}
f^{I}\left(a_{1},\ldots,T_{r}(x_{i}),\ldots,a_{n} \right) &+ \sum_{s \in [n]^{\ast}} a(f,s)\left( (x_{1},\ldots, r \cdot x_{i},\ldots,x_{n}), (a_{1},\ldots,r \cdot a_{i} + T_{r}(x_{i}),\ldots,a_{n}) \right) \\
+ \ T_{f}(x_{1},\ldots, r \cdot x_{i},\ldots,x_{n}) \ &= \ \sum_{s \in [n]^{\ast}} r \cdot a(f,s)(\vec{x},\vec{a}) + r \cdot T_{f}(\vec{x}) + T_{r}(f^{Q}(\vec{x}))
\end{align*}
for respecting the $r$-action in the i-th coordinate; of course, by cancellation and multilinearity of the action terms this is equivalent to the identity
\begin{align*}
f^{I} &\left(a_{1},\ldots,T_{r}(x_{i}),\ldots,a_{n} \right) + \sum_{i \not\in s \in [n]^{\ast}} a(f,s)\left( (x_{1},\ldots, r \cdot x_{i},\ldots,x_{n}), \vec{a} \right) \\ 
&+ \sum_{i \in s \in [n]^{\ast}} a(f,s)\left( \vec{x}, (a_{1},\ldots,T_{r}(x_{i}),\ldots,a_{n}) \right)
+ T_{f}(x_{1},\ldots, r \cdot x_{i},\ldots,x_{n}) \\
&= \sum_{i \not\in s \in [n]^{\ast}} r \cdot a(f,s)(\vec{x},\vec{a}) + r \cdot T_{f}(\vec{x}) + T_{r}(f^{Q}(\vec{x})).
\end{align*}
In the case $f$ is binary, the identity for the $1^{\mathrm{st}}$-coordinate simplifies to
\begin{align}
\begin{split}
f^{I}(T_{r}(x_1),x_{2}) &+ a(f,2)(r \cdot x_{1}, a_{2}) + a(f,1)(T_{r}(x_{1}),x_{2}) + T_{f}(r \cdot x_{1},x_{2}) \\
&= r \cdot a(f,2)(x_{1},a_{2}) + r \cdot T_{r}(x_{1},x_{2}) + T_{r}(f^{Q}(x_{1},x_{2})).
\end{split}
\end{align}

For Eqn (\ref{eqn:4}), calculating for the operations
\begin{align}
F_{f} \left(\left\langle a_{1},x_{1} \right\rangle, \ldots, \left\langle a_{i},x_{i} \right\rangle + \left\langle b_{i},y_{i} \right\rangle, \ldots, \left\langle a_{n},x_{n} \right\rangle \right)
\end{align}  
and 
\begin{align}
F_{f} \left(\left\langle a_{1},x_{1} \right\rangle, \ldots, \left\langle a_{i},x_{i} \right\rangle  , \ldots, \left\langle a_{n},x_{n} \right\rangle \right) + F_{f} \left(\left\langle a_{1},x_{1} \right\rangle, \ldots, \left\langle b_{i},y_{i} \right\rangle, \ldots, \left\langle a_{n},x_{n} \right\rangle \right)
\end{align}
yields the 2-cocycle identity 
\begin{align*}
f^{I}(a_{1},\ldots,T_{+}(x_{i},y_{i}),\ldots,a_{n}) &+ \sum_{i \not\in s \in [n]^{\ast}} a(f,s)\left( (x_{1},\ldots,x_{i}+y_{i},\ldots,x_{n}), \vec{a} \right) \\
&+ \sum_{i \in s \in [n]^{\ast}} a(f,s)\left( \vec{x}, (a_{1},\ldots,T_{+}(x_{i},y_{i}),\ldots,a_{n}) \right) + T_{f}(x_{1},\ldots,x_{i}+y_{i},\ldots,x_{n}) \\
&= \sum_{i \not\in s \in [n]^{\ast}} a(f,s)\left( \vec{x}, \vec{a} \right) + a(f,s)\left( (x_{1},\ldots,y_{i},\ldots,x_{n}), \vec{a} \right) \\
&+ T_{f}(\vec{x}) + T_{f}(x_{1},\ldots,y_{i},\ldots,x_{n}) + T_{+}(f^{Q}(\vec{x}),f^{Q}(x_{1},\ldots,y_{i},\ldots,x_{n}))
\end{align*}
for additivity of $f \in F$ in the i-th coordinate.

In the case $f$ is binary, the identity for additivity in the $1^{\mathrm{st}}$-coordinate simplifies to
\begin{align*}
f^{I}(T_{+}(x_{1},y_{1}),a_{2}) &+ a(f,1)(T_{+}(x_{1},y_{1}),x_{2}) + a(f,2)(x_{1} + y_{1}, a_{2}) + T_{f}(x_{1}+y_{1},x_{2}) \\
&= a(f,2)(x_1,a_{2}) + a(f,2)(y_{1},a_{2}) + T_{f}(x_{1},x_{2}) + T_{f}(y_{1},x_{2}) + T_{+}(f^{Q}(x_{1},x_{2}),f^{Q}(y_{1},x_{2})).
\end{align*}
\end{example}

In the case of a single binary operation $F = \{ f \}$, it may be more convenient to use the notations $a \circ y := a(f,1)(a,y)$ and $x \ast b := a(f,2)(x,b)$ for the associated action terms.

\begin{example}\label{ex:RotaBaxter}
A \emph{general Rota-Baxter algebra of weight} $\lambda$ is of the form $\left\langle _{R} M, \cdot, P \right\rangle$ where $\cdot$ is a bilinear operation over the $R$-module $M_{R}$ and $P$ is a linear transformation satisfying the identity
\begin{align}\label{eqn:6}
P(x) \cdot P(y) = P(P(x) \cdot y) + P(x \cdot P(y)) + \lambda P(x \cdot y).
\end{align}
If $Q,I$ are Rota-Baxter algebras over the same ring $R$, then a 2-cocycle for the datum $(Q,I)$ is of the form $T=\{T_{+}, T_{r}, T_{\cdot}, \ast, \circ : r \in R \}$ where we are using our convention to write actions associated to bilinear operations. By evaluating
\begin{align*}
F_{P}\left( \left\langle a, x \right\rangle \right) \cdot F_{P}\left( \left\langle b, y \right\rangle \right) =  F_{P} \left( F_{P}\left( \left\langle a, x \right\rangle\right) \cdot \left\langle b, y \right\rangle \right) + F_{P} \left( \left\langle a, x \right\rangle \cdot F_{P}\left( \left\langle b, y \right\rangle\right) \right) + \lambda F_{P} \left( \left\langle a, x \right\rangle \cdot \left\langle b, y \right\rangle  \right)
\end{align*}
and deleting identity Eq (\ref{eqn:6}) calculated in $I$-operations from the first coordinate, we arrive at the general 2-cocycle identity for $T$
\begin{align*}
P(x) \ast P(b) &+ P(a) \circ P(y) + P(a) \cdot T_{P}(y) + T_{P}(x) \cdot P(b) + T_{P}(x) \cdot T_{P}(y) + P(x) \ast T_{P}(y) + T_{P}(x) \ast P(y) \\
&+ T_{\cdot}(P(x),P(y)) \ = \ P(P(x) \ast b) + P(P(a) \circ y) + P(x \ast P(b)) + P(a \circ P(y)) \\
&+ \lambda P(a \circ y) + \lambda P(x \ast b) + P(T_{P}(x) \cdot b) + P(T_{P}(x) \circ y) + P(T_{\cdot}(P(x),y)) + T_{P}(P(x) \cdot y) \\
&+ P(a \cdot T_{P}(y)) + P(x \ast T_{P}(y)) + P(T_{\cdot}(x,P(y))) + T_{P}(x \cdot P(y)) + \lambda P(T_{\cdot}(x,y)) + \lambda T_{P}(x \cdot y) \\
&+ T_{\lambda}(P(x \cdot y)) + T_{+}(P(P(x) \cdot y),P(x \cdot P(y))) + T_{+}( P(P(x) \cdot y), P(x \cdot P(y)) ) + \lambda P(x \cdot y)).
\end{align*}

If we write $s(x,y) := P(x) \cdot P(y)$ and $t(x,y):= P(P(x) \cdot y) + P(x \cdot P(y)) + \lambda \cdot P(x \cdot y)$, then we see that 
\begin{align*}
s^{\ast,T}(a,b,x,y) &= P(x) \ast P(b) + P(a) \circ P(y) \\
t^{\ast,T}(a,b,x,y) &= P(P(x) \ast b) + P(P(a) \circ y) + P(x \ast P(b)) + P(a \circ P(y)) + \lambda P(a \circ y) + \lambda P(x \ast b)
\end{align*}
and so the action identity for compatibility with the Rota-Baxter identity is $s^{\ast,T}(a,b,x,y) = t^{\ast,T}(a,b,x,y)$; that is, 
\begin{align*}
 P(x) \ast P(b) + P(a) \circ P(y) &= P(P(x) \ast b) + P(P(a) \circ y) + P(x \ast P(b)) + P(a \circ P(y)) \\
&+ \ \lambda P(a \circ y) + \lambda P(x \ast b).
\end{align*}
\end{example}

\begin{example}\label{ex:Leibniz}
The identity for a bilinear operation $[-,-]: M^{2} \rightarrow M$ to be \emph{left-Leibniz} is 
\[
[x,[y,z]] = [[x,y],z] + [y,[x,z]].
\]
Then enforcing equality of
\[
F \left( \left\langle a,x \right\rangle, F \left( \left\langle b,y \right\rangle, \left\langle c,z \right\rangle \right) \right) = F \left( F \left( \left\langle a,x \right\rangle, \left\langle b,y \right\rangle \right) , \left\langle c,z \right\rangle \right) + F \left( \left\langle b,y \right\rangle, F \left( \left\langle a,x \right\rangle, \left\langle c,z \right\rangle \right) \right)
\]
and calculating the first coordinate we see that the action identity for compatibility with the left-Leibniz identity is
\begin{align*}
[a, b \circ z ] + [a, y \ast z] &+ a \circ [y,z] + x \ast [b,c] + x \ast (b \circ z) + x \ast (y \ast c) \\
&= [x \ast b,c] + [a \circ y,c] + [x,y] \ast c + [a,b] \circ z +(x \ast b) \circ z + (a \circ y) \circ z + [b,a \circ z] \\
&+ \ [b,x \ast c] + b \circ [x,z] + y \ast [a,c] + y \ast (a \circ z) + y \ast (x \ast c) 
\end{align*}
and the strict 2-cocycle identity for $T$ of a left-Leibniz operation is
\begin{align*}
[a,T(x,z)] + x \ast T(y,z) + T(x,[y,z]) &= [T(x,y),c] + T(x,y) \circ z + T([x,y],z) + [b,T(y,z)] + y \ast T(x,z) \\
& \quad + \ T(y,[x,z]) + T([[x,y],z],[y,[x,z]]).
\end{align*}
\end{example}

\begin{definition}
Let $\mathcal V$ be a variety of $R$-modules expanded by multilinear operations $F$. The set of $\mathcal V$-compatible 2-cocycles of datum $(Q,I)$ is denoted by $Z^{2}_{\mathcal V}(Q,I)$.
\end{definition}

\begin{definition}
Let $\mathcal V$ be a variety of $R$-modules expanded by multilinear operations $F$. Let $(Q,I)$ be datum in the signature of $\mathcal V$ and $T$ a 2-cocycle for the datum. An algebra $M$ \emph{realizes} the 2-cocycle $T$ if there is an extension $\pi: M \rightarrow Q$ with an isomorphism $i: I \rightarrow \ker \pi $ and a lifting $l: Q \rightarrow \ker \pi$ such that
\begin{enumerate}
	
	\item[(R1)] $i \circ T_{+}(x,y) = l(x) + l(y) - l(x+y)$;
	
	\item[(R2)] $i \circ T_{r}(x,y) = r \cdot l(x) - l(r\cdot x)$;
	
	\item[(R3)] $i \circ T_{f}(\vec{x}) = f^{M}(l(\vec{x})) - l(f^{Q}(\vec{x}))$;
	
	\item[(R4)] $i \circ a(f,s)(\vec{x},\vec{a}) = f^{M}(l(\vec{x}))_{s}[\vec{a}]$ for each $f \in F$.

\end{enumerate}
\end{definition}

\begin{lemma}
Let $\mathcal V$ be a variety of $R$-modules expanded by multilinear operations and $(Q,I)$ datum in the signature of $\mathcal V$. If $T$ is a 2-cocycle of the datum, then $I \rtimes_{T} Q$ realizes the 2-cocycle $T$.
\end{lemma}
\begin{proof}
Fix the lifting $l: Q \rightarrow I \times Q$ defined by $l(x) = \left\langle 0, x \right\rangle$. We calculate for the operations. For the module operations, we have
\begin{align*}
l(x) + l(y) - l(x+y) &= \left\langle \, 0, x \, \right\rangle + \left\langle \, 0, y \, \right\rangle - \left\langle \, 0, x + y \, \right\rangle \\
&= \left\langle \, T_{+}(x,y), x + y \, \right\rangle + \left\langle \, - T_{+}(x+y,-x-y), -x - y \, \right\rangle \\
&= \left\langle \, T_{+}(x,y), 0 \, \right\rangle
\end{align*}
and 
\begin{align*}
r \cdot l(x) - l(r\cdot x) = r \cdot \left\langle \, 0, x \, \right\rangle - \left\langle \, 0, r \cdot x \, \right\rangle &= \left\langle \, T_{r}(x), r \cdot x \, \right\rangle + \left\langle \, - T_{+}(r \cdot x, - r \cdot x ), - r \cdot x \, \right\rangle \\
&= \left\langle \, T_{r}(x), 0 \, \right\rangle.
\end{align*}
For $f \in F$ with $n = \ar f$, we have
\begin{align*}
F_{f} \left( l(\vec{x}) \right) - l(f^{Q}(\vec{x})) &= \left\langle \, f^{I}(\vec{0}) + \sum_{s \in  [n]^{\ast}} a(f,s)(\vec{x},\vec{0}) + T_{f}(\vec{x}) , f^{Q}(\vec{x}) \, \right\rangle - \left\langle 0, f^{Q}(\vec{x}) \right\rangle \\
&= \left\langle \, T_{f}(\vec{x}) , f^{Q}(\vec{x}) \, \right\rangle + \left\langle -T_{+}(f^{Q}(\vec{x}),-f^{Q}(\vec{x})), -f^{Q}(\vec{x}) \right\rangle \\
&= \left\langle T_{f}(\vec{x}) , 0 \right\rangle
\end{align*}
and for $\emptyset \neq s \subseteq [n]$, 
\begin{align*}
F_{f}\left( l(\vec{x}) \right)_{s} [i(\vec{a})]  = \left\langle \, f^{I}(\vec{0})_{s}[\vec{a}] + \sum_{r \in [n]^{s}} a(f,r)\left( \vec{x}_{s}[\vec{0}], \vec{0}_{s}[\vec{a}] \right) + T_{f}(\vec{x})_{s}[\vec{0}] ,  \, f^{Q}(\vec{x})_{s}[\vec{0}] \right\rangle  &= \left\langle \, a(f,s)(\vec{x},\vec{a}) , 0 \, \right\rangle 
\end{align*}
where we have used the (T4) property of the action terms to cancel in the sum those terms for $r \neq s$.
\end{proof}

\begin{theorem}\label{thm:multirep}
Let $\mathcal V$ be a variety of $R$-modules expanded by multilinear operations. Let $(Q,I)$ be datum in $\mathcal V$ and $\pi: M \rightarrow Q$ an extension with $\ker \pi = I$. Then $M \in \mathcal V$ if and only if $\pi : M \rightarrow Q$ realizes a $\mathcal V$-compatible 2-cocycle $T$ of the datum such that $M \approx I \rtimes_{T} Q$.
\end{theorem}
\begin{proof}
Let $M \in \mathcal V$ and $\pi: M \rightarrow Q$ a surjective homomorphism with $\ker \pi = I$. Choose a lifting $l : Q \rightarrow M$ of $\pi$ and define a 2-cocycle $T$ by the rules
\begin{enumerate}

	\item $T_{+}(x,y) := l(x) + l(y) - l(x+y)$;
	
	\item $T_{r}(x) := r \cdot l(x) - l(r\cdot x)$;
	
	\item $T_{f}(\vec{x}) := f^{M}(l(\vec{x})) - l(f^{Q}(\vec{x}))$;
	
	\item $i \circ a(f,s)(\vec{x},\vec{a}) = f^{M}(l(\vec{x}))_{s}[\vec{a}]$ for each $f \in F$.

\end{enumerate}
It is easy to see that properties (T1)-(T4) hold so that $T$ is a 2-cocycle for datum $(Q,I)$. Define $\psi: M \rightarrow I \rtimes_{T} Q$ by $\psi(a):= \left\langle a - l \circ \pi(a),\pi(a)  \right\rangle$. Since $l$ is a lifting of $\pi$, every $a \in M$ is uniquely represented as $a = \left( a - l \circ \pi(a) \right) + l \circ \pi(a)$ with $a - l \circ \pi(a) \in I$; consequently, $\psi$ is bijective. We verify $\psi$ is a homomorphism. Take $f \in F$ with $n = \ar f$. For $\vec{m} \in M^{n}$, write $x_{i}=\pi(m_{i})$ and $a_{i}= m_{i} - l (x_{i})$ so that we have $m_{i} = a_{i} + l(x_{i})$ with $a_{i} \in I, x_{i} \in Q$; thus, $\psi(m_{i}) = \left\langle a_{i}, x_{i} \right\rangle$. Then using multilinearity of $f$ we have
\begin{align*}
\psi \left(f^{M}(\vec{m}) \right) &= \left\langle f^{M}(\vec{m}) - l \circ \pi (f^{M}(\vec{m})) , \pi \circ f^{M}(\vec{m}) \right\rangle \\
&= \left\langle  f^{M}(a_{1} + l(x_{1}),\ldots,a_{n} + l(x_{n})) - l(f^{Q}(\vec{x})) , f^{Q}(\vec{x}) \right\rangle \\ 
&= \left\langle f^{M}(\vec{a}) + f^{M}(l(\vec{x})) + \sum_{s \in [n]^{\ast}} f^{M}(l(\vec{x}))_{s}[\vec{a}] - l(f^{Q}(\vec{x})) , f^{Q}(\vec{x}) \right\rangle \\
&= \left\langle f^{I}(\vec{a}) + \sum_{s \in [n]^{\ast}} a(f,s)(\vec{x},\vec{a}) + T_{f}(\vec{x}) , f^{Q}(\vec{x}) \right\rangle \\
&= F_{f}\left( \left\langle a_{1}, x_{1} \right\rangle, \ldots, \left\langle a_{n}, x_{n} \right\rangle \right) = F_{f} \left( \psi(\vec{m}) \right).
\end{align*}
We also see that
\begin{align*}
\psi( r \cdot m_{1} + m_{2} ) &= \left\langle  r \cdot a_{1} + a_{2} + r\cdot l(x_{1}) + l(x_{2}) - l(r \cdot x_{1} + x_{2}) , r \cdot x_{1} + x_{2} \right\rangle \\
&= \left\langle  r \cdot a_{1} + a_{2} + r\cdot l(x_{1}) - l(r \cdot x_{1}) + l(r \cdot x_{1}) + l(x_{2}) - l(r \cdot x_{1} + x_{2})  , r \cdot x_{1} + x_{2} \right\rangle \\
&= \left\langle  r \cdot a_{1} + T_{r}(x_{1}) + a_{2} + T_{+}(r \cdot x_{1},x_{2}) , r \cdot x_{1} + x_{2} \right\rangle \\
&= \left\langle  r \cdot a_{1} + T_{r}(x_{1}), r \cdot x_{1} \right\rangle + \left\langle a_{2}, x_{2} \right\rangle\\
&= r \cdot \left\langle a_{1}, x_{1} \right\rangle + \left\langle a_{2}, x_{2} \right\rangle = r \cdot \psi(m_{1}) + \psi(m_{2});
\end{align*}
altogether, $M \approx I \rtimes_{T} Q$. Since $M \in \mathcal V$, Lemma~\ref{lem:semident} implies $T$ is $\mathcal V$-compatible.

Conversely, assume $T$ is a $\mathcal V$-compatible 2-cocycle of the datum $(Q,I)$. Then since we assumed $Q,I \in \mathcal V$, Lemma~\ref{lem:semident} implies $ M \approx I \rtimes_{T} Q \in \mathcal V$.
\end{proof}

Lemma~\ref{lem:semident} and Theorem~\ref{thm:multirep} guarantee that the algebras $I \rtimes_{T} Q$ for $\mathcal V$-compatible 2-cocycles $T$ for datum $(Q,I)$ provide up to isomorphism all extensions of $Q$ by $I$; however, the identification in Theorem~\ref{thm:multirep} depends on the choice of a lifting. Choosing different liftings lead to a combinatorial notion of equivalence of extensions via the 2-cocycles.

\begin{definition}\label{def:30}
Let $\mathcal V$ be a variety of $R$-modules expanded by multilinear operations $F$ and $(Q,I)$ datum in $\mathcal V$. A 2-\emph{coboundary} for the datum the $(Q,I)$ is a sequence $G = \{G_{+}, G_{r}, G_{f}, g(f,s): r \in R, f \in F, s \in [\ar f]^{\ast} \}$ such that there exists an action $Q \ast I$ and a function $h : Q \rightarrow I$ with $h(0)=0$ which satisfies
\begin{enumerate}

	\item[(B1)] $G_{+}(x,y) = h(x) + h(y) - h(x+y)$;
	
	\item[(B2)] $G_{r}(x) = r \cdot h(x) - h(r \cdot x)$;
	
	\item[(B3)] $G_{f}(\vec{x}) = \sum_{s \in [\ar f]^{\ast}} (-1)^{1+|s|} a(f,s)(\vec{x},h(\vec{x})) + (-1)^{1+\ar f}f^{I}(h(\vec{x})) - h(f^{Q}(\vec{x}))$;
	
	\item[(B4)] $g(f,s)(\vec{x},\vec{a}) = \sum_{s \subsetneq r \subseteq [\ar f]^{\ast}} (-1)^{1 + |r| - |s|} a(f,r)(\vec{x},h(\vec{x}))_{s}[ \vec{a} ] + (-1)^{1 + |\ar f| - |s|}f^{I}\left( h(\vec{x}) \right)_{s}[ \vec{a} ]$.

\end{enumerate}
\end{definition}

The map $h:Q \rightarrow I$ in Definition~\ref{def:30} is said to \emph{determine} or \emph{witness} the 2-coboundary. The \emph{null} 2-coboundary has all entries zero.

\begin{definition}
Let $\mathcal V$ be a variety of $R$-modules expanded by multilinear operations $F$. The set of 2-coboundaries of the datum $(Q,I)$ is denoted by $B^{2}(Q,I)$.
\end{definition}

Note the notation for the set of 2-coboundaries has no index for the variety. We shall show in the comments after Theorem~\ref{thm:10} that this is because $B^{2}(Q,I) \subseteq Z^{2}_{\mathcal V}(Q,I)$ for any variety $\mathcal V$ containing the datum.

\begin{definition}
The 2-cocycles $T$ and $T'$ for datum $(Q,I)$ are \emph{equivalent}, written $T \sim T'$, if $T - T' \in B^{2}(Q,I)$.
\end{definition}

This is indeed an equivalence relation on 2-cocycles but it will be easier to see this by relating it to an equivalence on extensions given in Definition~\ref{def:equivexactseq}. Let us first observe that the combinatorial definition of a 2-coboundary is derived by choosing different liftings for an extension.

\begin{lemma}\label{lem:2coboundmulti}
Let $\mathcal V$ be a variety of $R$-modules expanded by multilinear operations $F$ and $\pi : M \rightarrow Q$ an extension with $\ker \pi = I$. If $l,l': Q \rightarrow M$ are liftings of $\pi$ which respectively define 2-cocycles $T$ and $T'$ by $(R1) - (R4)$, then $T \sim T'$. 
\end{lemma}
\begin{proof}
Define $h = l - l': Q \rightarrow I$. We will show $h$ determines the 2-coboundary $T-T'$. We see that 
\begin{align*}
T_{+}(x,y) = l(x) + l(y) - l(x+y) &= l'(x) + h(x) + l'(y) + h(y) - \left( l'(x+y) + h(x+y) \right) \\
&= h(x) + h(y) - h(x+y) + T'_{+}(x,y)
\end{align*}
and
\[
T_{r}(x) = r \cdot l(x) - l(r \cdot x) = r \cdot (l'(x) + h(x)) - \left( l'(r \cdot x) + h(r \cdot x) \right) = r \cdot h(x) - h(r \cdot x) + T'_{r}(x)
\]
which covers (B1) and (B2) in Definition~\ref{def:30}. Now fix $f \in F$ with $n=\ar f$. By realization, we see that
\begin{align*}
T_{f}(\vec{x}) &= f^{M} \left( l'(\vec{x}) \right) + h(\vec{x})) - l(f^{Q}(\vec{x})) \\ 
&= \sum_{r \in [n]^{\ast}} f^{M}\left( l'(\vec{x}) \right)_{r} [h(\vec{x})] + f^{I}(h(\vec{x})) - h(f^{Q}(\vec{x})) + f^{M} \left( l'(\vec{x}) \right) - l'(f^{Q}(\vec{x})) \\ 
&= \sum_{r \in [n]^{\ast}} f^{M}\left( l(\vec{x}) - h(\vec{x}) \right)_{r} [h(\vec{x})] + f^{I}(h(\vec{x})) - h(f^{Q}(\vec{x})) + T'_{f}(\vec{x}) \\
&= \sum_{r \in [n]^{\ast}} \sum_{r \subseteq s \subseteq [n]} (-1)^{|s|-|r|} f^{M}\left( l(\vec{x}) \right)_{s}[h(\vec{x})] + f^{I}(h(\vec{x})) - h(f^{Q}(\vec{x})) + T'_{f}(\vec{x}) \\ 
&= \sum_{s \in [n]^{\ast}} \left( \sum_{\emptyset \neq r \subseteq s } (-1)^{|s| - |r|}  \right) f^{M} \left( l(\vec{x}) \right)_{s}[h(\vec{x})] + \left(1 + \sum_{r \in [n]^{\ast}} (-1)^{n - |r|} \right) f^{I} \left( h(\vec{x}) \right) - h(f^{Q}(\vec{x})) + T'_{f}(\vec{x}) \\ 
&= \sum_{s \in [n]^{\ast}} (-1)^{1+|s|}f^{M} \left( l(\vec{x}) \right)_{s}[h(\vec{x})] + (-1)^{1+n}f^{I} \left( h(\vec{x}) \right) - h(f^{Q}(\vec{x})) + T'_{f}(\vec{x}) \\
&= \sum_{s \in [n]^{\ast}} (-1)^{1+|s|} a(f,s)(\vec{x},h(\vec{x})) + (-1)^{1+n}f^{I}(h(\vec{x})) - h(f^{Q}(\vec{x})) + T'_{f}(\vec{x})
\end{align*}
which covers (B3). For an action term $a(f,s)$ with $s \in [n]^{\ast}$, we use realization and the previous calculation applied to the action term as a function of $Q^{n-s}$ to get
\begin{align*}
a(f,s)(\vec{x},\vec{a}) &= f(l(\vec{x}))_{s}[\vec{a}] \\
&= \sum_{t \in ([n]-s)^{\ast}} (-1)^{1+|t|} a(f,t)\left( \vec{x},h(\vec{x}) \right)_{s}[\vec{a}] + (-1)^{1+|n|-|s|} f^{I}\left( h(\vec{x})\right)_{s}[\vec{a}] + f(l'(\vec{x}))_{s}[\vec{a}] \\
&= \sum_{s \subseteq r \subseteq [n]^{\ast}} (-1)^{1 + |r| - |s|} a(f,r)(\vec{x},h(\vec{x}))_{s}[ \vec{a} ] + (-1)^{1 + |n| - |s|}f^{I}\left( h(\vec{x}) \right)_{s}[ \vec{a} ] + a'(f,s)(\vec{x},\vec{a})
\end{align*}
using $r = t \cup s$ in the last line.
\end{proof}

The following theorem shows that equivalence on extensions is the same as the combinatorial equivalence on the associated 2-cocycles.

\begin{theorem}\label{thm:10}
Let $\mathcal V$ be a variety of $R$-modules expanded by multilinear operations $F$ and $(Q,I)$ datum in $\mathcal V$. Two extensions of $Q$ with kernel $I$ in the variety $\mathcal V$ are equivalent if and only they realize equivalent $\mathcal V$-compatible 2-cocycles.   
\end{theorem}
\begin{proof}
Let $\pi: M \rightarrow Q$ and $\pi': M' \rightarrow Q$ be extensions of $Q$ by $I$. Assume $\pi$ and $\pi'$ are equivalent. Let $\gamma: M \rightarrow M'$ be an isomorphism such that $\gamma|_{I} = \id_{I}$ and $\pi' \circ \gamma = \pi$. Let $l:Q \rightarrow M$ be a lifting of $\pi$ and $T$ be the $\mathcal V$-compatible 2-cocycle defined by $l$ according to (R1)-(R4). Then $\gamma \circ l: Q \rightarrow M'$ is a lifting of $\pi'$. Let $S$ be the 2-cocycle defined by the lifting $\gamma \circ l$. If we apply $\gamma$ to (R1)-(R4) defining $T$, then $\gamma|_{I} = \id_{I}$ implies that $T=S$; thus, $T$ is a 2-cocycle realized by the extension $\pi'$. So if $l': Q \rightarrow M'$ is another lifting which defines a 2-cocycle $T'$ realized by $\pi'$, Lemma~\ref{lem:2coboundmulti} yields $T \sim T'$.

Conversely, assume $T$ is a 2-cocycle realized by $\pi$, $T'$ is a 2-cocycle realized by $\pi'$ and $T \sim T'$. Then there is a map $h: Q \rightarrow I$ which determines the 2-coboundary $T'-T$ by (B1)-(B4). According to Theorem~\ref{thm:multirep}, we have the isomorphisms $\phi_{1}: M \rightarrow I \rtimes_{T} Q$ and $\phi_{2} : M' \rightarrow I \rtimes_{T'} Q$. We then define $\gamma': I \rtimes_{T} Q \rightarrow I \rtimes_{T'} Q$ by $\gamma' \left\langle a,x \right\rangle : = \left\langle a - h(x), x \right\rangle$ which is clearly bijective. To see that $\gamma'$ is a homomorphism, we calculate for the module operations
\begin{align*}
\gamma'\left\langle a,x \right\rangle + \gamma'\left\langle b,y \right\rangle = \left\langle a + b - h(x) - h(y) + T'_{+}(x,y) , x + y \right\rangle &= \left\langle a + b + T_{+}(x,y) - h(x+y), x + y \right\rangle\\
&= \gamma' \left\langle a + b + T_{+}(x,y), x+y \right\rangle \\
&= \gamma' \left( \left\langle a,x \right\rangle + \left\langle b,y \right\rangle\right)  
\end{align*}
and 
\begin{align*}
r \cdot \gamma'\left\langle a,x \right\rangle = \left\langle r \cdot a - r \cdot h(x) + T'_{r}(x), r \cdot x \right\rangle &= \left\langle r \cdot a + T_{r}(x) - h(r \cdot x), r \cdot x\right\rangle \\
&= \gamma'\left\langle r \cdot a + T_{r}(x) ,r \cdot x \right\rangle  \\
&= \gamma'\left( r \cdot \left\langle a, x\right\rangle \right).
\end{align*}
Now take $f \in F$ with $n=\ar f$, $\vec{a} \in I^{n}$, $\vec{x} \in Q^{n}$ and calculate
\begin{align*}
F_{f} &\left( \gamma'\left\langle a_1,x_1 \right\rangle, \ldots, \gamma'\left\langle a_n,x_n \right\rangle \right) \\
&= F_{f} \left( \left\langle a_{1} - h(x_{1}),x_{1} \right\rangle,\ldots,\left\langle a_{n} - h(x_{n}),x_{n} \right\rangle  \right) \\
&= \left\langle f^{I} \left(\vec{a} - h(\vec{x}) \right) + \sum_{s \in [n]^{\ast}} a'(f,s)\left(\vec{x}, \vec{a} - h(\vec{x}) \right) + T'_{f}(\vec{x}) , f^{Q}(\vec{x}) \right\rangle  \\
&= \Bigg<  f^{I}(\vec{a}) + \sum_{ \emptyset \neq t \subseteq [n]} (-1)^{|t|} f^{I} \left( \vec{a} \right)_{t}[h(\vec{x})] + \sum_{s \in [n]^{\ast}} \sum_{r \subseteq s} (-1)^{|r|} a'(f,s)\left(\vec{x}, \vec{a} \right)_{r}[h(\vec{x})] + T'_{f}(\vec{x}),  f^{Q}(\vec{x})  \Bigg> \\ \\
&= \Bigg< f^{I}(\vec{a}) + \sum_{u \in [n]^{\ast}} (-1)^{n - |u|} f^{I} \left(h(\vec{x}) \right)_{u}[\vec{a}]  +  (-1)^{n} f^{I} \left( h(\vec{x}) \right)  \\
&\quad + \sum_{u \in [n]^{\ast}} \sum_{u \subseteq v \in [n]^{\ast}} (-1)^{|v| - |u|} a'(f,v)(\vec{x},h(\vec{x}))_{u}[\vec{a}]  + T'_{f}(\vec{x}) ,  f^{Q}(\vec{x})  \Bigg>  \\ 
&=  \left\langle   f^{I}(\vec{a}) + \sum_{s \in [n]^{\ast}}  a(f,s)(\vec{x},\vec{a}) + T_{f}(\vec{x}) - h \left( f^{Q}(\vec{x}) \right),  f^{Q}(\vec{x}) \right\rangle \\
&= \gamma' \left\langle f^{I} \left( \vec{a} \right) + \sum_{s \in [n]^{\ast}} a(f,s)(\vec{x},\vec{a}) + T_{f}(\vec{x}), f^{Q}(\vec{x}) \right\rangle \\
&= \gamma' \left( F_{f}\left( \left\langle a_{1},x_{1} \right\rangle, \ldots, \left\langle a_{n},x_{n} \right\rangle \right) \right).
\end{align*}
It is immediate that $\gamma'|_{I \times 0} = \id_{I \times 0}$ and $p_{2} \circ \gamma' = p_{2}$. The required isomorphism is then $\gamma = \phi_{2}^{-1} \circ \gamma' \circ \phi_{1}$.
\end{proof}

Let $G$ be a 2-coboundary for $(Q,I)$ witnessed by $h: Q \rightarrow I$. It is straightforward to verify that $G$ satisfies properties (T1) - (T4), and so it is a 2-cocycle of $(Q,I)$. Since $G \sim 0$, we see by Theorem~\ref{thm:10} that $I \rtimes_{G} Q$ and $I \times Q$ are equivalent extensions, and so isomorphic. If $I,Q \in \mathcal V$, then $I \rtimes_{G} Q \approx I \times Q \in \mathcal V$ and so by Theorem~\ref{thm:multirep}, $G$ is $\mathcal V$-compatible. This shows $B^{2}(Q,I) \subseteq Z^{2}_{\mathcal V}(Q,I)$ for any variety $\mathcal V$ containing the datum algebras. This also shows that if $T$ is $\mathcal V$-compatible and $T \sim T'$, then $T'$ is also $\mathcal V$-compatible; therefore, the equivalence on 2-cocycles respects the compatibility with equational theories. The set of 2-cocycles equivalent to $T$ is denoted by $[T]$.

\begin{definition}
Let $\mathcal V$ be a variety of $R$-modules expanded by multilinear operations $F$. The $2^{\mathrm{nd}}$-cohomology of the datum $(Q,I)$ in the variety $\mathcal V$ is the set $H^{2}_{\mathcal V}(Q,I)$ of $\mathcal V$-compatible 2-cocycles modulo equivalence. 
\end{definition}

\begin{proposition}\label{prop:1}
Let $\mathcal V$ be a variety of $R$-modules expanded by multilinear operations $F$ and $M \in \mathcal V$. The following are equivalent:
\begin{enumerate}

	\item there is a retraction $r: M \rightarrow M$;
	
	\item there is a surjective homomorphism $\pi: M \rightarrow Q$ and homomorphism $l: Q \rightarrow M$ such that $\pi \circ l = \id_{Q}$; 
	
	\item there is an ideal $I \triangleleft M$ and subalgebra $Q \leq M$ such that $I \cap Q = 0$ and $M = I + Q$;
	
	\item $M \approx I \rtimes_{T^{\ast}} Q$ for a $\mathcal V$-compatible action $Q \ast I$.

\end{enumerate}
\end{proposition}
\begin{proof}
The equivalence of (1)-(3) is standard. We show $(2)$ and (4) are equivalent. Suppose $\phi: M \rightarrow I \rtimes_{T^{\ast}} Q$ witnesses an isomorphism. Define $l:Q \rightarrow I \times Q$ by $l(x):= \left\langle 0, x \right\rangle$; clearly, $l$ is a right-inverse map for the second-projection $p_{2}: I \rtimes_{T^{\ast}} Q \rightarrow Q$ which is always a homomorphism. Since $T^{\ast}$ is a 2-cocycle for $(Q,I)$ which satisfies properties (T1) -(T4) and all the non-action terms are zero, we see that $l$ is a homomorphism.

Conversely, assume $(2)$ holds. Then we define the 2-cocycle $T$ with the homomorphism $l: Q \rightarrow M$ in the same manner as Theorem~\ref{thm:multirep}. Since $l$ is a homomorphism we see that $T_{+} \equiv T_{r} \equiv T_{f} \equiv 0$ for all $r \in R$ and $f \in F$; thus $T=T^{\ast}$ for the action terms defined by $l$ and $M \approx I \rtimes_{T^{\ast}} Q$.
\end{proof}

An algebra $M$ satisfying condition (2) in Proposition~\ref{prop:1} is called a \emph{semidirect product} of $Q$ by $I$. It is then tempting to consider a general extension of datum $(Q,I)$ as a semidirect product translated by the non-action terms of the 2-cocycle.

At this point, Theorem~\ref{thm:multirep}, Lemma~\ref{lem:2coboundmulti} and Theorem~\ref{thm:10} establish that equivalence classes of extensions of fixed datum $(Q,I)$ are in bijective correspondence with equivalence classes of compatible 2-cocycles appropriate to $(Q,I)$. If we take the additional step to lay out the category of compatible 2-cocyles associated to $\mathcal V$, then this correspondence is just a restriction of an equivalence between categories. The category $\mathcal H^{2} \mathcal V$ has \emph{objects} the compatible 2-cocycles appropriate to datum in $\mathcal V$ and morphisms given in the following manner: if $T$ is a 2-cocycle for $(Q,I)$ and $T'$ is a 2-cocycle for $(P,J)$, then a \emph{morphism} is a triple $(\alpha,h,\beta) : T \rightarrow T'$ where $\alpha \in \Hom(I,J)$, $\beta \in \Hom(Q,P)$ and $h \in \mathrm{Map} (Q,J)$ with $h(0)=0$ such that 
\begin{enumerate}
			
	\item[(E1)] $\alpha \circ T_{+}(x,y) = T'_{+}(\beta(x),\beta(y)) + h(x) + h(y) - h(x+y)$,
			
	\item[(E2)] $\alpha \circ T_{r}(x) = T'_{r}(\beta(x)) + r \cdot h(x) - h(r \cdot x)$,
	
	\item[(E3)] $\alpha \circ a^{1}(f,s)(\vec{x},\vec{a}) = a^{2}(f,s)(\alpha(\vec{x}),\alpha(\vec{a})) + \sum_{s \subsetneq r \subseteq [\ar f]^{\ast}} a^{2}(f,r)(\beta(\vec{x}),h(\vec{x}))_{s}[\alpha(\vec{a})] + f^{K}(h(\vec{x}))_{s}[\alpha(\vec{a})]$.
			
\end{enumerate}
The intended interpretation for the map $h$ is the difference between two sections as in Lemma~\ref{lem:2coboundmulti}. If we have a morphism
\[
\begin{tikzcd}
	0 & {I} & {M} & {Q} & 0 \\
	0 & {J} & {N} & {P} & 0
	\arrow[from=1-1, to=1-2]
	\arrow["{i_{1}}", from=1-2, to=1-3]
	\arrow["\alpha_{1}", from=1-2, to=2-2]
	\arrow["\pi_{1}", from=1-3, to=1-4]
	\arrow["\lambda_{1}", no head, from=1-3, to=2-3]
	\arrow[from=1-4, to=1-5]
	\arrow["\beta_{1}", from=1-4, to=2-4]
	\arrow[from=2-1, to=2-2]
	\arrow["{i_{2}}", from=2-2, to=2-3]
	\arrow["\pi_{2}", from=2-3, to=2-4]
	\arrow[from=2-4, to=2-5]
\end{tikzcd}
\]
in $\mathrm{Ext} \mathcal V$, then let $l_{1} : Q \rightarrow M$ and $l_{2} : P \rightarrow N$ be sections which determine 2-cocycles $T^{1}$ and $T^{2}$, respectively. Then $h_{1}: Q \rightarrow J$ is the unique map such that
\begin{align}
i_{2} \circ h_{1}(x) + l_{2} \circ \beta_{1}(x) = \lambda_{1} \circ l_{1}(x).
\end{align}
If we have another morphism
\[
\begin{tikzcd}
	0 & {J} & {N} & {P} & 0 \\
	0 & {H} & {E} & {A} & 0
	\arrow[from=1-1, to=1-2]
	\arrow["{i_{2}}", from=1-2, to=1-3]
	\arrow["\alpha_{2}", from=1-2, to=2-2]
	\arrow["\pi_{2}", from=1-3, to=1-4]
	\arrow["\lambda_{2}", no head, from=1-3, to=2-3]
	\arrow[from=1-4, to=1-5]
	\arrow["\beta_{2}", from=1-4, to=2-4]
	\arrow[from=2-1, to=2-2]
	\arrow["{i_{3}}", from=2-2, to=2-3]
	\arrow["\pi_{3}", from=2-3, to=2-4]
	\arrow[from=2-4, to=2-5]
\end{tikzcd}
\]
with section $l_{3}:A \rightarrow E$ which defines 2-cocycle $T^{3}$, then there is $h_{2}:P \rightarrow H$ such that
\begin{align}
i_{3} \circ h_{2}(x) + l_{3} \circ \beta_{2}(x) = \lambda_{2} \circ l_{2}(x)
\end{align}
which then yields
\begin{align}
\lambda_{2} \circ \lambda_{1} \circ l_{1}(x) &= \lambda_{2} \circ i_{2} \circ h_{1}(x) + \lambda_{2} \circ l_{2} \circ \beta_{1}(x) \\
&= i_{3} \circ \alpha_{2} \circ h_{1}(x) + i_{3} \circ h_{2} \circ \beta_{1}(x) + l_{3} \circ \beta_{2} \circ \beta_{1}(x) \\
&= i_{3} \left( \alpha_{2} \circ h_{1}(x) + h_{2} \circ \beta_{1}(x) \right) + l_{3} \circ \beta_{2} \circ \beta_{1}(x).
\end{align}
With this interpretation, the composition of morphisms should be given by $(\alpha_{2},h_{2},\beta_{2}) \circ (\alpha_{1},h_{1},\beta_{1}) = (\alpha_{2} \circ \alpha_{1} , k, \beta_{2} \circ \beta_{1})$ where $k(x):= \alpha_{2} \circ h_{1}(x) + h_{2} \circ \beta_{1}(x).$

\begin{theorem}\label{thm:2cohomnonabel}
Let $\mathcal V$ be a variety of $R$-modules expanded by multilinear operations $F$ and $(Q,I)$ datum in $\mathcal V$. The set of equivalence classes of extensions of $Q$ by $I$ in the variety $\mathcal V$ is in bijective correspondence with $H^{2}_{\mathcal V}(Q,I)$. The equivalence class of a semidirect product corresponds to an equivalence class of a 2-cocycle with only action terms
\end{theorem}

We observe how the $2^{\mathrm{nd}}$-cohomologies are related through the parameter of the variety (or equational theory). For fixed datum $(Q,I)$, the set of $2^{\mathrm{nd}}$-cohomologies $\{  H^{2}_{\mathcal V}(Q,I) : \mathcal V \leq  _{R} \mathcal M_{F} \}$ forms an ordered set under inclusion denoted by $\mathcal H^{2}(Q,I)$.

\begin{proposition}
Let $(Q,I)$ be datum in the signature appropriate for multilinear $R$-module expansions. The class of varieties in the same signature which contain the datum forms a complete lattice $\mathcal L(Q,I)$. The maps $\Psi(\mathcal V) := H^{2}_{\mathcal V}(Q,I)$ and $\Theta (E) := \mathcal V \left( \{ I \rtimes_{T} Q : [T] \in E \} \right)$ define a Galois connection $(\Theta,\Psi)$ between the ordered sets $\mathcal H^{2}(Q,I)$ and $\mathcal L(Q,I)$.
\end{proposition}
\begin{proof}
It is easy to see that $\Psi$ is monotone on varieties and so yields the inequality $\Psi ( \mathcal U_{1} \wedge \mathcal U_{1} ) \leq \Psi (\mathcal U_{1}) \wedge \Psi(\mathcal U_{2})$. Now take $[T] \in H^{2}_{\mathcal U_{1}}(Q,I) \cap H^{2}_{\mathcal U_{2}}(Q,I)$; thus, $T$ is both $\mathcal U_{1}$ and $\mathcal U_{2}$-compatible. Then from Theorem~\ref{thm:multirep} and the class operator we see that $I \rtimes_{T} Q \in \Mod \left( \Id \mathcal U_{1} \cup \Id \mathcal U_{2} \right) = \mathcal U_{1} \wedge \mathcal U_{2}$. This implies $[T] \in H^{2}_{\Psi (\mathcal U_{1}) \wedge \Psi(\mathcal U_{2})}(Q,I)$ and so $\Psi$ preserves the meet.

From the definitions of the maps, it is not difficult to see that $\Theta \circ \Psi \geq \id_{\mathcal L(Q,I)}$ and so $\Theta$ is the lower-adjoint for $\Psi$.
\end{proof}

We can restrict the notion of equivalent extensions to the automorphisms of a fixed extension. An automorphism $\gamma \in \Aut M$ \emph{stabilizes} the extension $\pi: M \rightarrow Q$ where $\ker \pi = I$ if $\gamma|_{I} = \id_{I}$ and $\pi \circ \gamma = \pi$. The set of stabilizing automorphisms of the extension $\pi$ is denoted by $\mathrm{Stab}(\pi)$ and is a subgroup of $\Aut M$. Since the second-projection is the canonical map for extensions $I \rtimes_{T} Q$, we may also write $\mathrm{Stab}( I \rtimes_{T} Q )$ to denote its group of stabilizing automorphisms.

As we saw in the definition of 2-coboundaries, we may sometimes need to invoke an action as part of the definitions which leads to situations where an action is part of the notion of datum. This will become more apparent when we discuss extensions with abelian kernels and the set of extensions realizing particular fixed action terms. For algebras $Q$ and $I$ with an available action $Q \ast I$, the triple $(Q,I,\ast)$ may also be referred to as datum.

\begin{definition}
Let $\mathcal V$ be a variety of $R$-modules expanded by multilinear operations $F$, $Q, I \in \mathcal V$ and $Q \ast I$ an action. A 1-cocycle or \emph{derivation} of the datum $(Q,I,\ast)$ is a map $h: Q \rightarrow I$ which determines the null 2-coboundary. 
\end{definition}

Let $\mathrm{Der}(Q,I,\ast)$ denote the set of derivations of the datum $(Q,I,\ast)$. It is possible to see directly from (B1)-(B4) that the set of derivations is an abelian group under the addition induced by $I$ in the codomain; however, this can more easily seen by making the following connection between derivations and stabilizing automorphisms.

\begin{theorem}\label{thm:stabderiv}
Let $\mathcal V$ be a variety of $R$-modules expanded by multilinear operations $F$ and $(Q,I)$ datum in $\mathcal V$. The set of 1-cocycles is an abelian group and for any extension $\pi : M \rightarrow Q$ with $\ker \pi = I$, $\mathrm{Stab}(\pi) \approx \mathrm{Der}(Q,I,\ast)$ where the $Q \ast I$ is the associated action of the extension; in particular, $\mathrm{Stab}(\pi)$ is an abelian group.
\end{theorem}
\begin{proof}
Note stabilizing automorphisms are just the restriction to a fixed extension of the isomorphisms which witness equivalence and derivations determine the 2-coboundaries which witness reflexivity $T \sim T$. Then the proofs of Theorem~\ref{thm:10} and Lemma~\ref{lem:2coboundmulti} show that $\gamma: I \rtimes_{T} Q \rightarrow I \rtimes_{T} Q$ is a stabilizing automorphism if and only if it is of the form $\gamma(a,x) = \left\langle a - d(x), x \right\rangle$ for a derivation $d: Q \rightarrow I$.

We define a map $\Psi: \mathrm{Stab}(\pi) \approx \mathrm{Der}(Q,I,\ast)$ by $\Psi(\gamma) := d_{\gamma}$ where $\gamma(a,x) = \left\langle a - d_{\gamma}(x), x \right\rangle$. It follows that $\Psi$ is bijective by directly verifying $d_{\gamma_{d}} = d$ and $\gamma_{d_{\gamma}} = \gamma$ from the definitions. For two stabilizing automorphisms $\gamma$ and $\sigma$, the evaluation $\gamma \circ \sigma \left\langle a,x \right\rangle = \gamma \left\langle a - d_{\sigma}(x), x \right\rangle = \left\langle a - d_{\sigma}(x) - d_{\gamma}(x), x \right\rangle = \sigma \circ \gamma \left\langle a,x \right\rangle$ implies $\Psi(\gamma \circ \sigma) = \Psi(\gamma) + \Psi(\sigma)=\Psi(\sigma \circ \gamma)$ by bijectivity of $\Psi$; therefore, $\Psi$ is an isomorphism.
\end{proof}

We essentially follow \cite{wiresI} in the definition of principal derivations and $1^{\mathrm{st}}$-cohomology but with a modification to accommodate for the fact that congruences are determine by ideals. The motivation come from the fact that under the analogous isomorphism of Theorem~\ref{thm:stabderiv} for groups, principal derivations correspond to the stabilizing automorphisms which are inner automorphisms of the semidirect product of the datum.

Given an ideal $I \triangleleft M$, polynomials $p$ and $q$ are $I$-twins if there is a term $t(x,\vec{y})$ and $\vec{c},\vec{d} \in I^{k}$ such that $p(x)=t(x,\vec{c})$ and $q(x)=t(x,\vec{d})$. The set of $I$-twins of the identity is denoted by $\mathrm{Tw}_{I} \, M$; that is, $p \in \mathrm{Tw}_{I} \, M$ if there is a term $t(x,\vec{y})$ and $\vec{c},\vec{d} \in I^{k}$ such that $p(x)=t(x,\vec{c})$ and $x=t(x,\vec{d})$. Note $\mathrm{Tw}_{I} \, M$ is closed under composition. If $p,q \in \mathrm{Tw}_{I} \, M$, then there are terms $t(x, \vec{y}), s(x,\vec{z})$ and tuples $\vec{c},\vec{d} \in I^{k}$, $\vec{a}, \vec{b}\in I^{m}$ such that $p(x) = t(x,\vec{c}), x=t(x,\vec{d})$ and $q(x)=s(x,\vec{a}), x = s(x,\vec{b})$. Then the term $r(x,\vec{y},\vec{z}) := t(s(x,\vec{z}),\vec{y})$ with the tuples $(\vec{c},\vec{a})$ and $(\vec{d},\vec{b})$ shows $p \circ q \in \mathrm{Tw}_{I} \, M$. We then consider the subset $\mathrm{Tw}_{I,F} \, M = \{ p \in \mathrm{Tw}_{I} \, M : \exists a \in M, p(a)=a \}$ of $I$-twins of the identity which have a fixed point. The \emph{principal stabilizing automorphisms} of an extension $\pi: M \rightarrow Q$ realizing datum $(Q,I)$ are defined as
\[
\mathrm{PStab}(\pi) := \mathrm{Tw}_{I,F} \, M \cap \mathrm{Stab}(\pi).
\]
In general, $\mathrm{Tw}_{I,F} \, M$ is not closed under composition, but $\mathrm{Tw}_{I,F} \, M \cap \mathrm{Stab}(\pi)$ will be because for the algebras under consideration stabilizing automorphisms restrict to the identity on the ideal of the extension.

\begin{lemma}
Let $\mathcal V$ be a variety of $R$-modules expanded by multilinear operations, $Q,I \in \mathcal V$ and $Q \ast I$ an action. Then $\mathrm{PStab}(I \rtimes_{T^{\ast}} Q ) \leq \mathrm{Stab}(I \rtimes_{T^{\ast}} Q )$.
\end{lemma}
\begin{proof}
We need only show closure under inverses. We make the identification of $I$ with $I \times \{ 0 \}$. Let $\gamma \in \mathrm{PStab}(I \rtimes_{T^{\ast}} Q )$ witnessed by the term $t(x,\vec{y})$ and tuples $\vec{c},\vec{d}$; that is, 
\[
 \left\langle a - d_{\gamma}(x), x \right\rangle = p \left( a, x \right) = F_{t} \left( \left\langle a, x \right\rangle, \vec{c} \right) \quad \quad \text{and} \quad \quad \left\langle a, x \right\rangle = F_{t} \left( \left\langle a, x \right\rangle, \vec{d} \right).
\]
Consider the term $r(x,\vec{y},\vec{z}) := t(x,\vec{z}) - t(x,\vec{y}) + x$. Since in the semidirect product, the non-action terms in the 2-cocycle are all trivial, we calculate
\begin{align*}
F_{r} \left( \left\langle a, x \right\rangle, \vec{c}, \vec{d} \right) &= \left\langle a + d(x), x \right\rangle = p^{-1}(a,x)\\
F_{r} \left( \left\langle a, x \right\rangle, \vec{d}, \vec{d} \right) &= \left\langle a, x \right\rangle 
\end{align*} 
so that $p^{-1} \in \mathrm{PStab}(I \rtimes_{T^{\ast}} Q )$.
\end{proof}

The set of \emph{principal derivations} of the datum $(Q,I,\ast)$ is then defined as
\[
\mathrm{PDer}(Q,I,\ast) := \{ \, d(x): \gamma(a,x) = \left\langle a - d(x), x \right\rangle, \gamma \in \mathrm{PStab}(I \rtimes_{T^{\ast}} Q ) \}.
\]
The principal derivations are those derivations which under the isomorphism of Theorem~\ref{thm:stabderiv} correspond to the principal stabilizing automorphisms of the semidirect product $I \rtimes_{T^{\ast}} Q$.

\begin{definition}
Let $\mathcal V$ be a variety of $R$-modules expanded by multilinear operations, $Q,I \in \mathcal V$ and $Q \ast I$ an action. The $1^{\mathrm{st}}$-cohomology of the datum $(Q,I,\ast)$ is 
\[
H^{1}(Q,I,\ast) := \mathrm{Der}(Q,I,\ast)/\mathrm{PDer}(Q,I,\ast)
\] 
and the $1^{\mathrm{st}}$-cohomology of $(Q,I)$ is $H^{1}(Q,I) := \bigcup_{\ast} H^{1}(Q,I,\ast)$ where the union is over possible actions.
\end{definition}

Since the notion of derivation depends on the choice of action, we can restrict to a particular variety $\mathcal V$ by taking the union of $1^{\mathrm{st}}$-cohomology only over those $\mathcal V$-compatible actions
\[
H^{1}_{\mathcal V}(Q,I) = \bigcup \{ H^{1}(Q,I,\ast) : \ast \text{ is } \mathcal V\text{-compatible} \}.
\]

We say a 2-cocycle $T$ for datum $(Q,I)$ is \emph{linear} if the action terms are unary in $I$. Let us first observe that extensions with kernels which are abelian congruences are characterized by abelian ideals with linear 2-cocycles. This will then give a characterization of solvable algebras in the varieties of interest.

\begin{theorem}\label{thm:multiabel}
Let $\mathcal V$ be a variety of $R$-modules expanded by multilinear operations. A cohomology class $[T] \in H^{2}_{\mathcal V}(Q,I)$ represents an extension $\pi: M \rightarrow Q$ in which $I = \ker \pi \triangleleft M$ determines an abelian congruence if and only if $I$ is an abelian algebra and $T$ is linear. 
\end{theorem}
\begin{proof}
The 2-cocycle $T$ determines the extensions $M = I \rtimes_{T} Q \rightarrow Q$. We identify $I$ with its copy in $I \rtimes_{T} Q$. Let $\alpha$ be the congruence determined by $I$. First, assume $\alpha$ is an abelian congruence. By realization, the actions terms $a(f,s)$ are given by the multilinear operation symbols $a(f,s)(\vec{x},\vec{a}) = f(l(\vec{x}))_{s}[\vec{a}]$ for a fixed lifting $l:Q \rightarrow M$. Take multilinear $f \in F$ with $n \geq 2$. For any $\vec{z} \in M^{n-2}$, $a,b \in I$ we have $f(a,0,\vec{z})=0=f(0,0,\vec{z}) \Longrightarrow f(a,b,\vec{z}) = f(0,b,\vec{z}) = 0$ since $\alpha$ is abelian; therefore, $a(f,s) = 0$ for $s=\{1,2\}$. A similar argument shows $a(f,s) = 0$ for all $f \in F$ with $\ar f \geq 2$ and $|s|>1$; thus, $a(f,i)$ are the only possible nontrivial action terms. The same argument shows the mutlilinear operations $f \in F$ on $I$ are all trivial. This implies the ideal $I \triangleleft M$ carries only a module structure which is abelian in the TC-commutator \cite{mckenziesnow}.

Now, assume $I$ is abelian and $T$ is linear. We directly show the congruence $\alpha$ is abelian in the TC-commutator. Take a term $t$ and tuples $(\vec{a},\vec{x}), (\vec{b},\vec{x}) \in (I \times Q)^{k}$, $(\vec{c},\vec{z}), (\vec{d},\vec{z}) \in (I \times Q)^{m}$. Note we are using the convention $(\vec{a},\vec{x}) = \big( \left\langle a_{1},x_{1} \right\rangle, \ldots, \left\langle a_{k},x_{k} \right\rangle \big)$ to write the tuples. Assume 
\begin{align}\label{eqn:51}
F_{t} \big( (\vec{a},\vec{x}),(\vec{c},\vec{z}) \big) = F_{t} \big( (\vec{a},\vec{x}),(\vec{d},\vec{z}) \big).
\end{align}

The following claim on the representation of terms in $I \rtimes_{T} Q$ can be established by induction on their generation. The useful facts are that the action terms are unary and so linear in $I$, and the multilinear operations in $I$ are trivial.

\begin{claim}
Let $(Q,I,\ast)$ be affine datum in a variety $\mathcal V$ of modules expanded by multilinear operations. If $t(\vec{x})$ is a term in the signature and $T$ a $\mathcal V$-compatible 2-cocycle of the datum, then there exists operations $s$ and $s^{\ast}$ where $s$ and $s^{\ast}$ are iterations of action terms such that the interpretation of the term $t$ in the algebra $I \rtimes_{T} Q$ is given by
\begin{align}\label{eqn:52}
F_{t} \big( (\vec{a},\vec{x}),(\vec{c},\vec{z}) \big)  = \left\langle \, s((\vec{x},\vec{z}),\vec{a}) + s^{\ast}((\vec{x},\vec{z}), \vec{c}) + t^{\partial,T}(\vec{x},\vec{z}), t^{Q}(\vec{x},\vec{z}) \,\right\rangle
\end{align}
for $(\vec{a},\vec{x}) \in (I \times Q)^{k}$, $(\vec{c},\vec{z}) \in (I \times Q)^{m}$.
\end{claim}

Now apply the representation in Eq.(\ref{eqn:52}) to both terms in Eq.(\ref{eqn:51}) to conclude $s^{\ast}((\vec{x},\vec{z}),\vec{c}) = s^{\ast}((\vec{x},\vec{z}),\vec{d})$. To this equality we can add the operations $s((\vec{x},\vec{z}),\vec{b})$ and $t^{\partial,T}(\vec{x},\vec{z})$ to both sides and again use Eq.(\ref{eqn:52}) to conclude $F_{t} \big( (\vec{b},\vec{x}),(\vec{c},\vec{z}) \big) = F_{t} \big( (\vec{b},\vec{x}),(\vec{d},\vec{z}) \big)$. This shows $\alpha$ satisfies the term-condition and so is an abelian congruence.
\end{proof}

For an ideal $I \triangleleft M$, the derived series is defined by $[I]^{0} = I$ and $[I]^{n+1} = [[I]^{n},[I]^{n}]$. The ideal $I$ is \emph{n-step solvable} if $[I]^{n}=0$ and the algebra $M$ is \emph{n-step solvable} if $[M]^{n}=0$.

\begin{proposition}\label{prop:solvable}
Let $\mathcal V$ be a variety of $R$-modules expanded by multilinear operations. Then $M \in \mathcal V$ is a n-step solvable algebra if and only if $M$ can be represented as a right-associated product
\[
M \approx Q_{n} \rtimes_{T_{n-1}} Q_{n-1}\rtimes_{T_{n-2}} \cdots \rtimes_{T_{1}} Q_{1}
\]
where each $Q_{i} \in \mathcal V$ is abelian and $T_i$ is linear.  
\end{proposition}
\begin{proof}
Assume $M \in \mathcal V$ is n-step solvable. Set $Q_{k} = [M]^{k-1}/[M]^{k}$ for $1 \leq k \leq n$. By Theorem~\ref{thm:multiabel} and Theorem~\ref{thm:multirep}, we have $M/[M]^{k+1} \approx [M]^{k}/[M]^{k+1} \rtimes_{T_{k}} M/[M]^{k}$ where $[M]^{k}/[M]^{k+1}$ is abelian and each $T_k$ is linear. The right-associated product follows recursively.

Assume now we have the right-associated representation 
\[
M \approx Q_{n} \rtimes_{T_{n-1}} Q_{n-1}\rtimes_{T_{n-2}} \cdots \rtimes_{T_{1}} Q_{1}
\]
for abelian algebras $Q_i \in \mathcal V$ and compatible linear 2-cocycles $T_i$. Set $A_{i} := Q_{i} \rtimes_{T_{n-1}} Q_{n-1} \rtimes_{T_{n-2}} \cdots \rtimes_{T_{1}} Q_{1}$. By Theorem~\ref{thm:multiabel}, the second-projection $p_{2}: A_{i} = Q_{i} \rtimes_{T_{i-1}} A_{i-1} \rightarrow A_{i-1}$ realizes an extension in which $Q_{i} = \ker p_{2}$ determines an abelian congruence. We can apply the homomorphism property of the commutator to see that
\[
[A_{2},A_{2}] \vee Q_{2}/Q_{2} = [A_{2} \vee Q_{2}/Q_{2}, A_{2} \vee Q_{2}/Q_{2} ] = [Q_{1},Q_{1}] = 0
\]  
since $Q_{1}$ is abelian; thus, $[A_{2},A_{2}] \leq Q_{2}$ and so $[A_{2}]^{2} \leq [Q_{2},Q_{2}]=0$. Assume we have that $A_{k}$ is k-step solvable. Then 
\begin{align*}
[A_{k+1},A_{k+1}]^{k} \vee Q_{k+1}/Q_{k+1} &= \left[ [A_{k+1}]^{k-1} \vee Q_{k+1}/Q_{k+1}, [A_{k+1}]^{k-1} \vee Q_{k+1}/Q_{k+1} \right] \\
&= \left[ [A_{k+1}]^{k-2} \vee Q_{k+1}/Q_{k+1}  \right]^{2} \\
&= \vdots \\
&= [A_{k+1} \vee Q_{k+1}/Q_{k+1}]^{k} = \left[ A_{k} \right]^{k} = 0 
\end{align*}
which implies $[A_{k+1},A_{k+1}]^{k} \leq Q^{k+1}$ and so $[A_{k+1}]^{k+1} \leq [Q_{k+1},Q_{k+1}] = 0$; therefore, $A_{k+1}$ is k+1-step solvable. The proof is completed by induction since $A_{n} \approx M$.
\end{proof}

We can loosely paraphrase Proposition~\ref{prop:solvable} as stating that solvable algebras are iterated linear translations of semidirect products of abelian algebras; analogously, Proposition~\ref{prop:nilpotent} will state that nilpotent algebras are iterated linear translations of direct products of abelian algebras.

We say $(Q,I,\ast)$ is \emph{affine datum} if $I$ is an abelian algebra and if the terms in the action $Q \ast I$ are all unary in $I$. If $T$ is a linear 2-cocycle for datum $(Q,I)$ in which $I$ is an abelian algebra and $T \sim T'$, then $T'$ has the same action terms as $T$; to see this, for any 2-coboundary which witnesses $T-T' \in B^{2}(Q,I)$, the operations in (B4) must be zero because the higher-arity action terms and multilinear operation in $I$ are both trivial. This means equivalence respects the class of extensions which realize a fixed affine datum. There is an addition on equivalence classes of 2-cocycles $[T] + [T']:=[T + T']$ by the addition induced by $I$ on the factor-sets; that is, 
\[
(T + T')_{+} := T_{+} + T_{+}' \quad \quad (T + T')_{r} := T_{r} + T_{r}' \quad \quad (T + T')_{f} := T_{f} + T_{f}'
\]
and the action terms of $T+T'$ are exactly the same as $T$ and $T'$. Since the action terms are unary and the multilinear operations in $I$ are trivial, it is easy to see from (B1)-(B3) that the addition on equivalence classes is a well-defined abelian group operation.

We say $(Q,I,\ast)$ is datum in $\mathcal V$ if $Q,I \in \mathcal V$ and $Q \ast I$ is a $\mathcal V$-compatible action. We say $T$ is a 2-cocycle for $(Q,I,\ast)$ if it is a 2-cocycle for $(Q,I)$ and the action-terms of $T$ are given by the action $Q \ast I$. The $2^{\mathrm{nd}}$-cohomology for the datum $(Q,I,\ast)$ in the variety $\mathcal V$ is the set $H^{2}_{\mathcal V}(Q,I,\ast)$ of equivalence classes of strictly $\mathcal V$-compatible 2-cocycles for the datum $(Q,I,\ast$).

\begin{theorem}\label{thm:affinecohom}
Let $\mathcal V$ be a variety of $R$-modules expanded by multilinear operations $F$ and $(Q,I,\ast)$ affine datum in $\mathcal V$. The set of equivalence classes of extensions which realize the datum $(Q,I,\ast)$ in the variety $\mathcal V$ is in bijective correspondence with the abelian group $H^{2}_{\mathcal V}(Q,I,\ast)$. The equivalence class of the semidirect product corresponds to the equivalence class of the trivial 2-cocycle.
\end{theorem}
\begin{proof}
Since we are given that the action $Q \ast I$ is $\mathcal V$-compatible, it follows from Lemma~\ref{lem:semident} and Theorem~\ref{thm:multirep} that for any 2-cocycle $T$ for $(Q,I,\ast)$, $I \rtimes_{T} Q \in \mathcal V$ if and only if $T$ is strictly $\mathcal V$-compatible. Then Theorem~\ref{thm:10} guarantees the set $H^{2}_{\mathcal V}(Q,I,\ast)$ parametrizes the equivalence classes of extensions in $\mathcal V$ which realize $(Q,I,\ast)$. Since the action terms are linear and the multilinear operations are trivial in $I$, it follows by the same derivation as in Lemma~\ref{lem:semident} that for any term $t$ in the signature of $\mathcal V$ we have $t^{\partial,T+T'} = t^{\partial,T} + t^{\partial,T'}$. Because the action terms of $T + T'$ are the same as $T$ and $T'$, it follows that $T + T'$ is again strictly $\mathcal V$-compatible whenever both $T$ and $T'$ are; altogether, the abelian group addition on 2-cocycles for affine datum respects equivalence classes and compatibility.
\end{proof}

\begin{remark}\label{remark:affineident}
In the case of affine datum $(Q,I,\ast)$, there is slight modification to the representation of terms from Lemma~\ref{lem:semident}. Let $M = I \rtimes_{T} Q$ realize the datum. In the evaluation $F_{t}(\vec{m})$ of the term $t(\vec{x})$ in the algebra $I \rtimes_{T} Q$, the operation $t^{\partial,T}$ gathers together all instances of the factor-sets of the 2-cocycle $T_{+}$, $T_{r}$ and $T_{f}$; for general datum, it involves evaluation of the action terms and all the operations of the algebra $I$ and so $t^{\partial,T}$ may depend on both the $I$ and $Q$ coordinates of $\vec{m}$. In the case of affine datum, the multilinear operations of $I$ are trivial and all the action terms are unary in $I$. Since the factor-sets of the 2-cocycle must appear in the operations of $t^{\partial,T}$, it cannot depend on the $I$-coordinates. Then for $m_i = \left\langle b_{i}, x_{i} \right\rangle \in I \times Q$ we can write
\[
F^{M}_{t}(\vec{m}) = \left\langle t^{\ast,T}(\vec{b},\vec{x}) + t^{\partial,T}(\vec{x}),  t^{Q}(\vec{x}) \right\rangle.
\]
\end{remark}

Take an extension $\pi: M \rightarrow Q$ realizing affine datum $(Q,I,\ast)$ and determined by the 2-cocycle $T$. Then Theorem~\ref{thm:10} yields the isomorphism $\psi: M \ni a \longmapsto \left\langle a - l \circ \pi (a), \pi(a) \right\rangle \in I \rtimes_{T} Q$ where in the algebra $I \rtimes_{T} Q$ the multilinear operations are computed by
\begin{align}\label{eqn:20}
F_{f}\left( \vec{a}, \vec{x} \right) = \left\langle \sum_{i \in [\ar f]} a(f,i)(\vec{x},\vec{a}) + T_{f}(\vec{x}), f^{Q}(\vec{x}) \right\rangle .
\end{align}
This agrees with the more general notion of extensions realizing affine datum developed for arbitrary varieties of universal algebras in \cite{wiresI}. Since varieties of multilinear module expansions form a special subclass of varieties with a difference term, the next two results follow directly by comparison of Eq~\eqref{eqn:20} with the characterization of central extensions and nilpotent algebras from \cite{wiresI}; however, we will argue separately. A 2-cocycle for datum $(Q,I)$ is \emph{action-trivial} if the action terms are all zero.

\begin{theorem}\label{thm:multicentral}
Let $\mathcal V$ be a variety of $R$-modules expanded by multilinear operations. A class $[T] \in H^{2}_{\mathcal V}(Q,I)$ represents an extension $\pi: M \rightarrow Q$ in which $I = \ker \pi \triangleleft M$ determines a central congruence if and only if $I$ is an abelian algebra and $T$ is action-trivial. 
\end{theorem}
\begin{proof}
Assume the congruence $\alpha_{I}$ determined by the ideal $I$ is central; in particular, $I$ is an abelian algebra and the action terms are unary in $I$ according to Theorem~\ref{thm:multiabel}. Take $f \in F$, $i \in [f]$ and $\vec{x} \in Q^{\ar f}$, $\vec{a} \in I^{\ar f}$. Since $\ar f \geq 2$, choose $i \neq j \in [ \ar f ]$ and $\vec{z} \in Q^{\ar f}$ such that $z_{j} = 0$. Then applying the term-condition we see that
\[
a(f,i)(\vec{z},\vec{0}) = 0 = a(f,i)(\vec{x},\vec{0}) \quad \Longrightarrow  \quad 0 = a(f,i)(\vec{z},\vec{a}) = a(f,i)(\vec{x},\vec{a})
\]
since $[I,M]=0$; thus, $a(f,i) \equiv 0$ and so $T$ is action-trivial.

Now assume $I$ is an abelian algebra and $T$ is action-trivial. Then for a term $t$, the interpretation in $I \rtimes_{T} Q$ is given by $F_{t}\left( \vec{a}, \vec{x} \right) = \left\langle t^{I}(\vec{a}) + t^{\partial,T}(\vec{a},\vec{x}), t^{Q}(\vec{x}) \right\rangle$ where no action-terms appear in $t^{\partial,T}$. We can follow the argument in Theorem~\ref{thm:multiabel} and verify the term-condition to show $\alpha_{I}$ is a central congruence.
\end{proof}

\begin{proposition}\label{prop:nilpotent}
Let $\mathcal V$ be a variety of $R$-modules expanded by multilinear operations. Then $M \in \mathcal V$ is a n-step nilpotent algebra if and only if $M$ can be represented as a right-associated product
\[
M \approx Q_{n} \rtimes_{T_{n-1}} Q_{n-1}\rtimes_{T_{n-2}} \cdots \rtimes_{T_{1}} Q_{1}
\]
where each $Q_{i} \in \mathcal V$ is abelian and $T_i$ is action-trivial. 
\end{proposition}
\begin{proof}
This is precisely the argument in Proposition~\ref{prop:solvable} where we use the the characterization of central extensions in Theorem~\ref{thm:multicentral} instead of Theorem~\ref{thm:multiabel}.
\end{proof}

We denote by $(Q,I,0)$ datum in which the action is trivial; that is, $a(f,s) \equiv 0$ for all $f \in F$, $s \in [\ar f]^{\ast}$.

\begin{theorem}\label{thm:centcohom}
Let $\mathcal V$ be a variety of $R$-modules expanded by multilinear operations $F$ and $(Q,I)$ datum in $\mathcal V$ with $I$ an abelian algebra. The set of equivalence classes of central extensions which realize the datum $(Q,I)$ in the variety $\mathcal V$ is in bijective correspondence with the abelian group $H^{2}_{\mathcal V}(Q,I,0)$. The equivalence class of the direct product $I \times Q$ corresponds to the equivalence class of the trivial 2-cocycle. 
\end{theorem}

In light of the last theorem, when $I$ is an abelian algebra we may refer to $(Q,I,0)$ as \emph{central datum}. The next lemma gives a description of $1^{\mathrm{st}}$-cohomology for central datum.

\begin{lemma}
Let $\mathcal V$ be a variety of $R$-modules expanded by multilinear operations $F$ and $Q,I \in \mathcal V$ with $I$ an abelian algebra. Then $H^{1}(Q,I,0) = \mathrm{Der}(Q,I,0) =  \{ h \in \Hom_{R}(Q,I): h([Q,Q]) = 0 \}$. 
\end{lemma}
\begin{proof} By Proposition~\ref{thm:multicentral}, the copy of $I$ in an extension realizing the datum determines a central congruence. Let $\gamma$ be a principal stabilizing automorphism of $I \times Q$. Then there is term $t(x,\vec{y})$ and $\vec{c},\vec{d} \in (I \times \{0\})^{k}$ such that $\gamma(x)=t(x,\vec{c})$ and $x=t(x,\vec{d})$. Since $\gamma$ is a stabilizing automorphism, it restricts to the identity on $I \times \{ 0 \}$ and so it has fixed-points, say $a \in I \times \{ 0 \}$. Then for any $x \in I \times Q$, the matrix
\[
\begin{bmatrix} a & \gamma(x) \\ a & x \end{bmatrix} = \begin{bmatrix} \gamma(a) & \gamma(x) \\ a & x \end{bmatrix} = \begin{bmatrix} t(a,\vec{c}) & t(x,\vec{c}) \\ t(a,\vec{d}) & t(x,\vec{d}) \end{bmatrix} \in M(\alpha_{I \times \{0 \}},1)
\]
implies $(\gamma(x),x) \in [1, \alpha_{I \times \{0\}}] = 0$; thus, $0 = \mathrm{PStab}(I \times Q) \approx \mathrm{PDer}(Q,I,0)$ and so $H^{1}(Q,I,0) = \mathrm{Der}(Q,I,0)$.

Since the action terms are all trivial, it follows from (B1) - (B3) that $h \in \mathrm{Der}(Q,I,0)$ if and only if $h$ is an $R$-module homomorphism from $Q$ to $I$ such that $h(f^{Q}(\vec{x})) = 0$ for all $f \in F$ and $\vec{x} \in Q^{\ar f}$; that is, $\mathrm{Der}(Q,I,0) = \{ h \in \Hom_{R}(Q,I): h([Q,Q]) = 0 \}$.
\end{proof}

It is possible to give a more concrete form to principal derivations.

\begin{remark}\label{rem:principal}
Let $(Q,I,\ast)$ be datum and $\gamma$ a nontrivial principal stabilizing automorphism. There is a term $t(x,\vec{y})$ and $\vec{c},\vec{d} \in I^{n}$ such that  
\begin{align*}
\left\langle a - d_{\gamma}(x), x \right\rangle = \gamma(a,x) &= F_{t} \left( \left\langle a,x \right\rangle, \left\langle c_{1},0 \right\rangle, \ldots, \left\langle c_{n},0 \right\rangle \right) = \left\langle t^{I}(a,\vec{c}) + t^{\ast,T}((a,\vec{c}),(x,\vec{0})) , t^{Q}(x,\vec{0})\right\rangle \\
\left\langle a, x \right\rangle &= F_{t} \left( \left\langle a,x \right\rangle, \left\langle d_{1},0 \right\rangle, \ldots, \left\langle d_{n},0 \right\rangle\right) = \left\langle t^{I}(a,\vec{d}) + t^{\ast,T}((a,\vec{d}),(x,\vec{0})) , t^{Q}(x,\vec{0})\right\rangle
\end{align*}
where we have used the representation from Lemma~\ref{lem:semident} in the semidirect product $I \rtimes_{T} Q$. Taking $x=0$ and using $d_{\gamma}(0)=0$, we conclude that $t(a,\vec{c}) = t(a,\vec{d})=a$. Then from the above we see that 
\[
- d_{\gamma}(x) = t^{\ast,T}((0,\vec{c}),(x,\vec{0})) \quad \quad \text{ and } \quad \quad 0 = t^{\ast,T}((0,\vec{d}),(x,\vec{0}))
\]
since the derivation $d_{\gamma}$ is independent of $a$. Note that $t^{\ast,T}$ is a sum of iterated action terms and so cannot depend on $0 \in Q$. In this way, we can just write $d_{\gamma}(x) = s^{\ast,T}(\vec{c},x)$ which is a sum of iterated action terms which depend only on $x \in Q$ and evaluated on the constants $\vec{c} \in I^{n}$. Then $\left\langle d_{\gamma}(x), 0 \right\rangle = F_{s} \left(\left\langle 0,x \right\rangle, \left\langle c_{1},0 \right\rangle ,\ldots, \left\langle c_{n}, \right\rangle \right)$.
\end{remark}


\section{Ideal-Preserving Derivations}\label{section:4}
\vspace{0.3cm}

For the varieties under consideration, the notion of a derivation (as different from a cohomological derivation) is meaningful. For affine datum, we prove a Wells-type theorem which characterizes the Lie algebra of derivations which preserve a fixed ideal; more formally, for $I \triangleleft M$ which realizes a group-trivial extension, the Lie algebra of derivations of $M$ which fix $I$ is a Lie algebra extension of compatible pairs of derivations of $I$ and $M/I$ by the cohomological derivations of the datum $\mathrm{Der}(Q,I,\ast)$. This is very similar to the theorem of Wells \cite{wells} for automorphisms fixing a normal subgroup and the analogous result for associative conformal algebras in \cite{conform2}.

\begin{theorem}\label{thm:derivations}
Let $\mathcal V$ be a variety of modules expanded by multilinear operations and $M \in \mathcal V$ Let $\pi: M \rightarrow Q$ be a group-trivial extension realizing affine datum $(Q,I,\ast)$. If $T$ is the 2-cocycle associated to the extension, then 
\begin{align}\label{eqn:derivext}
0 \longrightarrow \mathrm{Der}(Q,I,\ast) \longrightarrow \mathrm{Der}_{I} M \stackrel{\psi}{\longrightarrow} c(I,Q,\ast) \stackrel{W_{T}}{\longrightarrow} H^{2}_{_{R} \mathcal M_{F}}(Q,I)
\end{align}
is an exact sequence of Lie algebras.
\end{theorem}

We begin by developing the items in the statement of the theorem. Let $T$ be a 2-cocycle which determines the extension $\pi: M \rightarrow Q$. The extension is \emph{group-trivial} if $T \sim T'$ with $T_{+}'=0$; consequently, the abelian group reduct of a group-trivial extension factors as a direct product of the datum groups.

\begin{definition}
Let $\mathcal V$ be a variety of multilinear expansions of $R$-modules and $M \in \mathcal V$. A map $h: M \rightarrow M$ is a \emph{derivation} if it is an $R$-module homomorphism such that
\begin{align*}
h \left( f^{M}(x_{1},\ldots,x_{n}) \right) &= f^{M}(h(x_{1}),x_{2},\ldots,x_{n}) + f^{M}(x_{1},h(x_{2}),\ldots,x_{n}) + \cdots + f^{M}(x_{1},x_{2},\ldots, h(x_{n})) \\
&= \sum_{i=1}^{n} f^{M}(\vec{x})_{i}[h(\vec{x})] 
\end{align*}
for all multilinear operations $f \in F$ with $n= \ar f$.
\end{definition}

The set of derivations of $M$ is denoted by $\mathrm{Der} \, M$ and is an $R$-module under the module operations induced by $M$; consequently, it is a Lie algebra over $R$ by the standard bracket $[\alpha,\beta] := \alpha \circ \beta - \beta \circ \alpha$ through composition.

Given $(\alpha,\beta) \in \mathrm{Der} \, I \times \mathrm{Der} \, Q$, we define a map $\phi_{(\alpha,\beta)} : I \times Q \rightarrow I \times Q$ by the rule $\phi_{(\alpha,\beta)} \left\langle a,x \right\rangle:= \left\langle \alpha(a),\beta (x) \right\rangle$. Let $T$ be a linear group-trivial 2-cocycle. We can query what conditions must be true in order for $\phi_{(\alpha,\beta)}$ to be a derivation of the semidirect product $I \rtimes_{T} Q$.

\begin{lemma}\label{lem:compderiv}
Let $\mathcal V$ be a variety of multilinear expansions of $R$-modules which contains datum $(Q,I)$. Let $T$ be group-trivial linear 2-cocycle and $(\alpha,\beta) \in \mathrm{Der} \, I \times \mathrm{Der} \, Q$. Then $\phi_{(\alpha,\beta)} \in \mathrm{Der} \, \left( I \rtimes_{T} Q \right)$ if and only if
\begin{enumerate}

	\item[(C1)] $\alpha \circ T_{r}(x) = T_{r}(\beta(x))$ for all $r \in R$;

	\item[(C2)] for each $f \in F$ with $n= \ar f$ and $k \in [n]$, 
\begin{align*}
\alpha \circ a(f,k)(\vec{x},\vec{a}) = \sum_{i \in [n]} a(f,k)\left( (x_{1},\ldots,\beta(x_{i}),\ldots,x_{n}), (a_{1},\ldots,\alpha(a_{i}),\ldots,a_{n}) \right);
\end{align*}
	
	\item[(C3)] for each $f \in F$ with $n= \ar f$, $\alpha \circ T_{f}(\vec{x}) = \sum_{i \in [n]} T_{f}(x_{1},\ldots,\beta(x_{i}),\ldots,x_{n})$.

\end{enumerate}
\end{lemma}
\begin{proof}
Since $T$ is group-trivial and linear, the action terms are unary in $I$ and $T_{+} = 0$. It is then easy to see that $\phi_{(\alpha,\beta)}$ is an abelian group homomorphism. We first see how to derive conditions (C1)-(C3). In order for $\phi_{(\alpha,\beta)}$ to a be module homomorphism we must have equality between
\begin{align*}
\phi_{(\alpha,\beta)} \left( r \cdot \left\langle a,x\right\rangle \right) = \phi_{(\alpha,\beta)} \left( \left\langle ra + T_{r}(x), rx \right\rangle\right) = \left\langle \alpha(rx) + \alpha \circ T_{r}(x), \beta(rx) \right\rangle
\end{align*}
and
\begin{align*}
r \cdot \phi_{(\alpha,\beta)} \left(\left\langle a, x \right\rangle \right) = r \cdot \left\langle \alpha(a), \beta(x) \right\rangle = \left\langle r \alpha(x) + T_{r}(\beta(x)), r\beta(x) \right\rangle.
\end{align*}
Since $\alpha$ and $\beta$ are module homomorphisms, we must have $\alpha \circ T_{r}(x) = T_{r}(\beta(x))$ for all $r \in R$. Calculating for multilinear $f \in F$ with $n=\ar f$, 
\begin{align*}
\phi_{(\alpha,\beta)} \circ F_{f}\left( \left\langle a_{1}, x_{1} \right\rangle,\ldots, \left\langle a_{1}, x_{1} \right\rangle \right) &= \phi_{(\alpha,\beta)} \circ \left\langle f^{I}(\vec{a}) + \sum_{i \in [n]} a(f,i)(\vec{x},\vec{a}) + T_{f}(\vec{x}) , f^{Q}(\vec{x}) \right\rangle  \\
&= \left\langle \alpha \circ f^{I}(\vec{a}) + \sum_{i \in [n]} \alpha \circ  a(f,i)(\vec{x},\vec{a}) + \alpha \circ  T_{f}(\vec{x}), \beta \circ f^{Q}(\vec{x}) \right\rangle  
\end{align*}
and
\begin{align*}
\sum_{i \in [n]} &F_{f} \left( \left\langle a_{1},x_{1} \right\rangle,\ldots,\left\langle \alpha(a_{i}),\beta(x_{i}) \right\rangle,\ldots,\left\langle a_{n},x_{n} \right\rangle  \right) \\
&= \sum_{i \in [n]} \left\langle f^{I}(\vec{a})_{i}[\alpha(\vec{a})] + \sum_{j \in [n]} a(f,j) \left( (\vec{x})_{i} [\beta(\vec{x})], (\vec{a})_{i} [\alpha(\vec{a})] \right) + T_{f}(\vec{x})_{i} [\beta(\vec{x})], f^{Q}(\vec{x})_{i} [\beta(\vec{x})] \right\rangle \\
&= \left\langle \sum_{i \in [n]} f^{I}(\vec{a})_{i}[\alpha(\vec{a})] + \sum_{i \in [n]} \sum_{j \in [n]} a(f,j) \left( (\vec{x})_{i} [\beta(\vec{x})], (\vec{a})_{i} [\alpha(\vec{a})] \right) + \sum_{i \in [n]} T_{f}(\vec{x})_{i} [\beta(\vec{x})], \sum_{i \in [n]} f^{Q}(\vec{x})_{i} [\beta(\vec{x})] \right\rangle 
\end{align*}
using the fact that $T_{+}=0$. Let us assume $\phi_{(\alpha,\beta)}$ is a derivation which produces the equality
\[
\phi_{(\alpha,\beta)} \circ F_{f}\left( \left\langle a_{1}, x_{1} \right\rangle,\ldots, \left\langle a_{1}, x_{1} \right\rangle \right) = \sum_{i \in [n]} F_{f} \left( \left\langle a_{1},x_{1} \right\rangle,\ldots,\left\langle \alpha(a_{i}),\beta(x_{i}) \right\rangle,\ldots,\left\langle a_{n},x_{n} \right\rangle  \right).
\] 
Since $\alpha$ and $\beta$ are derivations, the above calculations yield
\begin{align}\label{eqn:48}
\sum_{i \in [n]} \alpha \circ  a(f,i)(\vec{x},\vec{a}) + \alpha \circ  T_{f}(\vec{x}) = \sum_{i \in [n]} \sum_{j \in [n]} a(f,j) \left( (\vec{x})_{i} [\beta(\vec{x})], (\vec{a})_{i} [\alpha(\vec{a})] \right) + \sum_{i \in [n]} T_{f}(\vec{x})_{i} [\beta(\vec{x})].
\end{align}
Suppose for each $k \in [n]$ we set $x_{k}=0$. Then $a(f,i)(\vec{x},\vec{a}) = 0$ for $i \neq k$ and $T_{f}(\vec{x}) = 0$. We also see that each $T_{f}(\vec{x})_{i} [\beta(\vec{x})] = 0$ since either $\beta(x_{k})=0$ or $x_{k}=0$ appears in the evaluated input of $T_{f}(\vec{x})_{i} [\beta(\vec{x})]$. Similarly, for $j \neq k$ either $x_{k}=0$ is an evaluated input in the action when $i \neq k$ or $\beta(x_{k})=0$ is an evaluated input in the action when $i=k$; thus, $a(f,j) \left( (\vec{x})_{i} [\beta(\vec{x})], (\vec{a})_{i} [\alpha(\vec{a})] \right) = 0$ for $j \neq k$. Then Eqn.(\ref{eqn:48}) implies 
\[
\alpha \circ a(f,k)(\vec{x},\vec{a}) = \sum_{i \in [n]} a(f,k) \left( (\vec{x})_{i} [\beta(\vec{x})], (\vec{a})_{i} [\alpha(\vec{a})] \right)
\]
which confirms condition (C2). Subtracting this from Eqn.(\ref{eqn:48}) yields 
\[
\alpha \circ T_{f}(\vec{x}) = \sum_{i \in [n]} T_{f}(\vec{x})_{i} [\beta(\vec{x})]
\] 
which confirms condition (C3).

From the above, it is straightforward to confirm that conditions (C1)-(C3) guarantee $\phi_{(\alpha,\beta)}$ is a derivation of $I \rtimes_{T} Q$.
\end{proof}

Since we are interested in reconstructing the derivations of an extension which realizes affine datum, note condition (C2) in Lemma~\ref{lem:compderiv} only depends on the datum since the action terms are incorporated into the datum.

\begin{definition}
Let $\mathcal V$ be a variety of multilinear expansions of $R$-modules containing affine datum $(Q,I,\ast)$. The set of \emph{compatible derivations} $c(I,Q,\ast)$ consists of all pairs $(\alpha,\beta) \in \mathrm{Der} \, I \times \mathrm{Der} \, Q$ which satisfy for each $f \in F$ with $n= \ar f$ and $k \in [n]$, 
\begin{align*}
\alpha \circ a(f,k)(\vec{x},\vec{a}) = \sum_{i \in [n]} a(f,k)\left( (x_{1},\ldots,\beta(x_{i}),\ldots,x_{n}), (a_{1},\ldots,\alpha(a_{i}),\ldots,a_{n}) \right);
\end{align*}
\end{definition}

Given $(\alpha,\beta), (\sigma,\kappa) \in c(I,Q,\ast)$, we can take $T$ to be trivial except for the action terms and observe that Lemma~\ref{lem:compderiv} shows $\phi_{(\alpha,\beta)}$ and $\phi_{(\sigma,\kappa)}$ are both derivations of $I \times_{T} Q$. Then the standard bracket $[\phi_{(\alpha,\beta)},\phi_{(\sigma,\kappa)}]$ will again be a derivation of $I \rtimes_{T} Q$ which must satisfy (C2) by Lemma~\ref{lem:compderiv}. This shows $c(I,Q,\ast)$ is a Lie subalgebra of $\mathrm{Der} \, I \times \mathrm{Der} \, Q$.

Fix an extension $\pi: M \rightarrow Q$ with $I=\ker \pi$. For any lifting $l:Q \rightarrow M$ of $\pi$, we have 
\begin{align}\label{eqn:47}
\pi \circ l = \id_{Q} \quad \quad \quad \text{and} \quad \quad \quad a - l \circ \pi(a) \in I \quad \quad (a \in M).
\end{align}
For any $\phi \in \mathrm{Der}_{I} \, M$, define $\phi_{l} := \pi \circ \phi \circ l$. If $l': Q \rightarrow M$ is another lifting for $\pi$, then $l(x) - l'(x) \in I$ for all $x \in Q$ which implies $\phi \circ l(x) - \phi \circ l'(x) \in I$ because $\phi \in \mathrm{Der}_{I} \, M$. It follows that $\phi_{l} = \phi_{l'}$ and so the value of $\phi_{l}$ is independent of the lifting. We can also see that $\phi_{l}$ is a derivation of $Q$. For a multilinear operation $f \in F$ with $n = \ar f$, if we substitute $f^{M}(l(\vec{x}))$ for $a$ in Eq.~\ref{eqn:47} and apply $\phi$, then $I \ni \phi \left( f^{M}(l(\vec{x})) \right) - \phi \left( l(f^{Q}(\vec{x})) \right)$. Then applying $\pi$ we see that
\begin{align*}
0 = \pi \circ \phi \left( f^{M}(l(\vec{x})) \right)  - \phi_{l}(f^{Q}(\vec{x})) &= \pi \left( \sum_{i \in [n]} f^{M}(l(\vec{x}))_{i} [ \phi(l(\vec{x})) ] \right) - \phi_{l}(f^{Q}(\vec{x})) \\
&= \sum_{i \in [n]} f^{Q}(\vec{x})_{i}[\phi_{l}(\vec{x})] - \phi_{l}(f^{Q}(\vec{x})).
\end{align*}
A similar argument shows $\phi_{l}$ is an $R$-module homomorphism; therefore, $\phi_{l} \in \mathrm{Der} \, Q$. If $\phi \in \mathrm{Der}_{I} \, M$, then clearly the restriction $\phi|_{I} \in \mathrm{Der} \, I$.

\begin{lemma}\label{lem:derivdecon}
Let $\mathcal V$ be a variety of multilinear expansions of $R$-modules and $M \in \mathcal V$. If $\pi: M \rightarrow Q$ is an extension realizing affine datum $(Q,I,\ast)$ and $l: Q \rightarrow M$ is a lifting of $\pi$, then $\psi(\phi) := (\phi|_{I},\phi_l)$ defines a Lie algebra homomorphism $\psi: \mathrm{Der}_{I} \, M \rightarrow c(I,Q,\ast)$.
\end{lemma}
\begin{proof}
We first show $(\phi|_{I},\phi_l)$ is a compatible pair for $\phi \in \mathrm{Der}_{I} \, M$. Since $(Q,I,\ast)$ is affine datum, the action terms $a(f,s)(\vec{x},\vec{a})=0$ for $|s|>1$. If we substitute $\phi \circ l(x_i)$ into a in Eq.(\ref{eqn:47}), we see that $I \ni b_{i} = \phi \circ l(x_i) - l \circ \pi \circ \phi \circ l(x_i) = \phi \circ l(x_i) - l \circ \phi_{l}(x_i)$ from some $b_i \in I$. Since the extension $M$ realizes the datum, we have for multilinear $f \in F$ with $n = \ar f$, $\vec{x} \in Q^{n}$, $\vec{a} \in I^{n}$ and $k \neq i$,
\begin{align*}
f^{M} &(l(x_{1}),\ldots,a_{k},\ldots,(x_{n}))_{i}[\phi \circ l(x_{i})] \\
&= f^{M}(l(x_{1}),\ldots,a_{k},\ldots,(x_{n}))_{i}[b_{i} + l \circ \phi_{l}(x_{i})] \\
&= f^{M}(l(x_{1}),\ldots,a_{k},\ldots,(x_{n}))_{i}[b_{i}] + f^{M}(l(x_{1}),\ldots,a_{k},\ldots,(x_{n}))_{i}[l \circ \phi_{l}(x_{i})] \\
&= f^{M}(l(x_{1}),\ldots,a_{k},\ldots,(x_{n}))_{i}[l \circ \phi_{l}(x_{i})] \\
&= a(f,k) \left( (x_{1},\ldots, \phi_{l}(x_{i}),\ldots,x_{n}), \vec{a} \right) \\
&= a(f,k) \left( (x_{1},\ldots, \phi_{l}(x_{i}),\ldots,x_{n}), (a_{1},\ldots,\phi(a_{i}),\ldots,a_{n}) \right).
\end{align*}
Then using that $\phi$ is a derivation we have
\begin{align*}
\phi \circ a(f,k)(\vec{x},\vec{a}) &= \phi \circ f^{M} \left( l(x_{1}),\ldots,a_{k},\ldots,l(x_{n}) \right)  \\
&= f^{M}\left( \phi \circ l(x_1),\ldots,a_{k},\ldots,x_{n} \right) + \cdots + f^{M}\left( x_1,\ldots,a_{k},\ldots,\phi \circ l(x_{n}) \right)  \\
&= f^{M}(l(x_{1}),\ldots,\phi(a_{k}),\ldots,l(x_{n})) + \sum_{i \neq k} f^{M} (l(x_{1}),\ldots,a_{k},\ldots,(x_{n}))_{i}[\phi \circ l(x_{i})]  \\ 
&= \sum_{i \in [n]} a(f,k)\left( (x_{1},\ldots,\phi_{l}(x_{i}),\ldots,x_{n}), (a_{1},\ldots,\phi(a_{i}),\ldots,a_{n}) \right)
\end{align*}
which shows $\psi(\phi) = (\phi|_{I},\phi_l) \in c(I,Q)$. 

Take $\phi,\sigma \in \mathrm{Der}_{I} \, M$. From Eq.(\ref{eqn:47}) we have $\sigma \circ l(x) - l \circ \pi (\sigma \circ l(x)) \in I$ and so $ \phi \circ \sigma \circ l(x) - \phi \circ l \circ \pi (\sigma \circ l(x)) \in I$ since $\phi(I) \subseteq I$. Then applying the canonical epimorphism $\pi$ we have $(\phi \circ \sigma)_{l}(x) = \phi_{l} \circ \sigma_{l}(x)$. Then by direct calculation
\begin{align*}
[\phi,\sigma]_{l}(x) = \pi \circ [\phi,\sigma] \circ l (x) = \pi \circ (\phi \circ \sigma) \circ l(x) - \pi \circ (\sigma \circ \phi) \circ l(x) &= \phi_{l} \circ \sigma_{l}(x) - \sigma_{l} \circ \phi_{l} (x) = [\phi_{l},\sigma_{l}](x).
\end{align*}
Clearly, restriction to $I$ respects the bracket. Since the bracket in $c(I,Q,\ast)$ is the restriction of the bracket in $\mathrm{Der} \, I \times \mathrm{Der} \, Q$, we have shown $\psi$ is a Lie algebra homomorphism.
\end{proof}

For each $(\alpha,\beta) \in c(I,Q,\ast)$ and function $f: Q^{n} \rightarrow I$, define $f^{(\alpha,\beta)}$ by
\begin{align*}
f^{(\alpha,\beta)}(\vec{x}):= \alpha \circ f(\vec{x}) - \sum_{i\in [n]} f(x_{1},\ldots,x_{i-1},\beta(x_{i}),x_{i+1},\ldots,x_{n}).
\end{align*}
If $T$ is a 2-cocycle for affine datum $(Q,I,\ast)$, then define $T^{(\alpha,\beta)} := \{T^{(\alpha,\beta)}_{+}, T^{(\alpha,\beta)}_{r},T^{(\alpha,\beta)}_{f} : r \in R, f \in F \}$. It may not be the case that $T^{(\alpha,\beta)}$ is $\mathcal V$-compatible when $T$ is, but we shall see that it is $_{R} \mathcal M_{F}$-compatible. Let $ H^{2}_{\mathcal V}(Q,I,\ast)^{\mathrm{gr}}$ denote the subset of group-trivial 2-cocycles and note it is a subgroup of $2^{\mathrm{nd}}$-cohomology for affine datum.

\begin{lemma}
Let $\mathcal V$ be a variety of multilinear expansions of $R$-modules and $(Q,I,\ast)$ affine datum in $\mathcal V$. Then $W((\alpha,\beta),[T]) := [T^{(\alpha,\beta)}]$ defines a Lie algebra representation
\[
W: c(I,Q,\ast) \times H^{2}_{\mathcal V}(Q,I,\ast)^{\mathrm{gr}} \rightarrow H^{2}_{_{R} \mathcal M_{F}}(Q,I,\ast)
\]
of $c(I,Q,\ast)$ on $H^{2}_{_{R} \mathcal M_{F}}(Q,I,\ast)$.
\end{lemma}
\begin{proof}
We must first show $W$ is well-defined on cohomology classes. This is done by showing the action on a group-trivial 2-coboundary is again a 2-coboundary. 

Suppose $G$ is a 2-coboundary for the datum; therefore, there is $h: Q \rightarrow I$ which satisfies 

\begin{enumerate}

	\item $0 = h(x) + h(y) - h(x + y)$;

	\item $G_{r}(x) = r \cdot h(x) - h(r \cdot x)$ for all $r \in R$;
	
	\item for all $f \in F$ with $n = \ar f$, $G_{f}(\vec{x}) = \sum_{i \in n} a(f,i)(\vec{x},h(\vec{x})) - h(f^{Q}(\vec{x}))$.

\end{enumerate} 
Take $(\alpha,\beta) \in c(I, Q, \ast)$. For the ring action we have
\begin{align*}
G_{r}(x)^{(\alpha,\beta)} = \alpha \circ G_{r}(x) - G_{r}(\beta(x)) &= \alpha \circ \left( r \cdot h(x) - h(r \cdot x) \right) - \left( r \cdot h(\beta(x)) - h(r \cdot \beta(x)) \right) \\
&= r \cdot \left( \alpha \circ h - h \circ \beta \right)(x) - \left( \alpha \circ h - h \circ \beta \right) ( r \cdot x) \\
&= r \cdot s(x) - s(r \cdot x)  
\end{align*}
where we set $s:= (\alpha \circ h - h \circ \beta) : Q \rightarrow I$. Note we explicitly used that $\alpha$ and $\beta$ are module homomorphisms and $h$ is a group homomorphism by (1). In the next calculation, we use the fact that the action terms are unary in $I$. For multilinear $f \in F$ with $n = \ar f$, we have 
\begin{align*}
&G_{f}^{(\alpha,\beta)}(\vec{x}) \\
&= \alpha \circ G_{f}(\vec{x}) - \sum_{i \in [n]} G_{f}(x_{1},\ldots,\beta(x_{i}),\ldots,x_{n}) \\
&=  \sum_{k \in [n]} \alpha \circ a(f,k) \left( \vec{x}, h(\vec{x}) \right) - \alpha \circ h(f^{Q}(\vec{x})) \\
&\quad - \sum_{i \in [n]} \Big( \sum_{k \in [n]} a(f,k)\left( (x_{1},\ldots,\beta(x_{i}),\ldots,x_{n}) , (h(x_{1}),\ldots, h \circ \beta(x_{i}),\ldots,h(x_{n})) \right) - h(f^{Q}(\vec{x}))_{i} [\beta(\vec{x})] \Big)   \\
&= \sum_{k \in [n]} \sum_{i \in [n]} a(f,k) \big( (x_{1},\ldots,\beta(x_{i}),\ldots,x_{n}), (h(x_{1}),\ldots,\alpha \circ h(x_{i}),\ldots,h(x_{n})) \big) - \alpha \circ h(f^{Q}(\vec{x}))   \\
&\quad  -  \sum_{i \in [n]}  \sum_{k \in [n]} a(f,k)\big( (x_{1},\ldots,\beta(x_{i}),\ldots,x_{n}) , (h(x_{1}),\ldots, h \circ \beta(x_{i}),\ldots,h(x_{n}))  \big) + \sum_{i \in [n]} h(f^{Q}(\vec{x}))_{i} [\beta(\vec{x})]  \\
&= \sum_{k \in [n]} a(f,k) \big( (x_{1},\ldots,\beta(x_{k}),\ldots,x_{n}) , (h(x_{1}),\ldots, \alpha \circ h(x_{k}) - h \circ \beta(x_{k}),\ldots,h(x_{n}))  \big) \\
&\quad - \alpha \circ h(f^{Q}(\vec{x})) + h \circ \beta \left( f^{Q}(\vec{x}) \right) \\
&= \sum_{k \in [n]} a(f,k)(\vec{x},s(\vec{x})) - s \left( f^{Q}(\vec{x}) \right).
\end{align*}
This shows $G^{\alpha,\beta}$ is a 2-coboundary witnessed by $s:Q \rightarrow I$.

Next, we show for fixed $[T] \in H^{2}_{\mathcal V}(Q,I,\ast)^{\mathrm{gr}}$ that $W(-,[T]): c(I,Q,\ast) \rightarrow H^{2}_{_{R} \mathcal M_{F}}(Q,I,\ast)$ respects the bracket. Since $[T]$ is group-trivial, then we can observe that $T$ is linear with respect to equivalence. For simplicity, we may assume $T_{+} = 0$. Then we may take an extension $M$ and lifting $l:Q \rightarrow M$ which is a group homomorphism and for this lifting the 2-cocycle is realized by definitions (R1)-(R4); in particular, for multilinear operation $f \in F$ we have $T_{f}(\vec{x}) = f^{M}(l(\vec{x})) - l(f^{Q}(\vec{x}))$. Then 
\begin{align*}
T_{f}(x_{1},\ldots,x_{i-1},q + p,x_{i+1},\ldots,x_{n}) &= f^{M}(l(\vec{x}))_{i}[l(q + p)] - l \big( f^{Q}(\vec{x})_{i}[q+p] \big) \\
&= f^{M}(l(\vec{x}))_{i}[l(q) + l(p)] - l \big( f^{Q}(\vec{x})_{i}[q] + f^{Q}(\vec{x})_{i}[p] \big) \\
&= f^{M}(l(\vec{x}))_{i}[l(q)] + f^{M}(l(\vec{x}))_{i}[l(p)] - l \big( f^{Q}(\vec{x})_{i}[q] \big) - l \big( f^{Q}(\vec{x})_{i}[p] \big) \\
&= T_{f}(x_{1},\ldots,x_{i-1},q,x_{i+1},\ldots,x_{n}) + T_{f}(x_{1},\ldots,x_{i-1},p,x_{i+1},\ldots,x_{n}).
\end{align*}
Now take $(\alpha,\beta),(\sigma,\kappa) \in c(I,Q,\ast)$ and note $[(\alpha,\beta),(\sigma,\kappa)] = ([\alpha,\sigma],[\beta,\kappa])$. Then  
\begin{align*}
(T^{(\sigma,\kappa)})^{(\alpha,\beta)}(\vec{x}) &- (T^{(\alpha,\beta)})^{(\sigma,\kappa)}(\vec{x}) \\
&= \Big( \sigma \circ T(\vec{x}) - \sum_{i \in [n]} T(\vec{x})_{i}[\kappa(\vec{x})] \Big)^{(\alpha,\beta)} - \Big( \alpha \circ T(\vec{x}) - \sum_{i \in [n]} T(\vec{x})_{i}[\beta(\vec{x})] \Big)^{(\sigma,\kappa)}\\
&= \alpha \circ \big( \sigma \circ T(\vec{x}) - \sum_{i \in [n]} T(\vec{x})_{i}[\kappa(\vec{x})]  \big) - \sum_{j \in [n]} \big( \sigma \circ T(\vec{x}) - \sum_{i \in [n]} T(\vec{x})_{i}[\kappa(\vec{x})]  \big)_{j}[\beta(\vec{x})] \\
&\quad - \sigma \circ \big( \alpha \circ T(\vec{x}) - \sum_{i \in [n]} T(\vec{x})_{i}[\beta(\vec{x})]  \big) + \sum_{j \in [n]} \big( \alpha \circ T(\vec{x}) - \sum_{i \in [n]} T(\vec{x})_{i}[\beta(\vec{x})]  \big)_{j}[\kappa(\vec{x})]   \\
&= [\alpha,\sigma] \circ T(\vec{x}) +  \sum_{j \in [n] } \sum_{i \in [n]} \big( T(\vec{x})_{i} [\kappa(\vec{x})] \big)_{j} [\beta(\vec{x})] - \sum_{j \in [n] } \sum_{i \in [n]} \big( T(\vec{x})_{i} [\beta(\vec{x})] \big)_{j} [\kappa(\vec{x})] \\
&= [\alpha,\sigma] \circ T(\vec{x}) + \sum_{i \in n} T(\vec{x})_{i} \left[ \kappa \circ \beta (\vec{x}) \right] - \sum_{i \in n} T(\vec{x})_{i} \left[ \beta \circ \kappa(\vec{x}) \right]  \\
&= [\alpha,\sigma] \circ T(\vec{x}) - \sum_{i \in n} T(\vec{x})_{i} \left[ [\beta,\kappa](\vec{x}) \right] = T^{([\alpha,\sigma],[\beta,\kappa])}(\vec{x}).
\end{align*}
It follows that $W \big( [(\alpha,\beta),(\sigma,\kappa)],[T] \big) = \big[ W \big( \alpha,\beta),[T] \big), W \big( (\alpha,\beta),[T] \big) \big]$.

Finally, we show $T^{(\alpha,\beta)}$ is $_{R} \mathcal M_{F}$-compatible. Since the datum $(Q,I,\ast)$ is affine, the equations in Example~\ref{ex:multiequations} to enforce $_{R} \mathcal M_{F}$-compatibility for a group trivial 2-cocycle $T$ of the datum simplifies to just stating that $T_{f}$ is a multilinear operation for each $f \in F$. It then follows that if we take $\mathcal V$-compatible group-trivial $[T] \in H^{2}_{\mathcal V}(Q,I,\ast)^{\mathrm{gr}}$ and since $\mathcal V \leq \, _{R} \mathcal M_{F}$, then $T_{f}$ is multilinear for each $f \in F$. Since both $\alpha$ and $\beta$ are module homomorphisms, the compositions $\alpha \circ T(\vec{a})$ and $T_{f}(\vec{x})_{i} [\beta(\vec{x})]$ are multilinear. Since the cohomology classes for affine datum forms an abelian group, $[T^{(\alpha,\beta)}] \in H^{2}_{_{R} \mathcal M_{F}}(Q,I,\ast)$. 
\end{proof}

\begin{proof}(of Theorem~\ref{thm:derivations})
According to Theorem~\ref{thm:multirep}, we may assume the extension is given by $M = I \rtimes_{T} Q$ and $\pi: I \rtimes_{T} Q \rightarrow Q$ is given by second-projection. Fix the lifting $l: Q \rightarrow I \rtimes_{T} Q$ defined by $l(x) = \left\langle 0, x \right\rangle$. We may identify $I$ with it's copy $I \times \{0\}$ in $I \rtimes_{T} Q$. Since the extension realizes a group-trivial 2-cocycle, we may assume $T_{+}=0$ and the lifting $l$ is a group homomorphism.

First, we show $\ker \psi \approx \mathrm{Der}(Q,I,\ast)$. Assume $\psi(\phi) = (\phi|_{I},\phi_{l}) = (0,0)$. Note $0 = \phi_{l}(x) = \pi \circ \phi \circ l(x)$ implies $\phi \circ l(x) \in I$ for all $x \in Q$. Then $\phi(a,0) = (0,0)$ implies $\phi(a,x) = \phi(a,0) + \phi(0,x) = \phi(0,x)$. We see that $F_{f}\left(l(\vec{x})\right) = \left\langle T_{f}(\vec{x}), f^{Q}(\vec{x}) \right\rangle = \left\langle T_{f}(\vec{x}), 0 \right\rangle + \left\langle 0, f^{Q}(\vec{x}) \right\rangle$. Then since $\phi$ is a derivation, we have
\begin{align*}
\phi \circ l (f^{Q}(\vec{x})) = \phi(0,f^{Q}(\vec{x})) = \phi \left( F_{f}\left(l(\vec{x})\right) - \left\langle T_{f}(\vec{x}), 0 \right\rangle \right) &= \phi \left( F_{f}\left(l(\vec{x})\right) \right) \\
&= \sum_{i \in [n]} F_{f} \left( l(\vec{x}) \right)_{i} [\phi(l(\vec{x}))] \\
&= \left\langle \, \sum_{i \in [n]} a(f,i)\left( \vec{x}, \phi \circ l(\vec{x}) \right) , 0 \, \right\rangle
\end{align*}
by realization; thus, $\phi \circ l \in \mathrm{Der} (Q,I,\ast)$. Define $\Delta: \ker \psi \rightarrow \mathrm{Der} (Q,I,\ast)$ by $\Delta(\phi):= \phi \circ l$. We show $\Delta$ is a Lie algebra isomorphism. It is clearly seen to respect the standard bracket operation.

For injectivity, assume $\phi \circ l = \Delta(\phi)=0$. Since $\phi \in \ker \psi$, we see that $\phi(a,x) = \phi(a,0) + \phi(0,x) = \phi \circ l(x) = 0$; thus, $\phi = 0$. For surjectivity, take $\sigma \in \mathrm{Der} (Q,I,\ast)$. Define $\gamma(a,x) := \left\langle \sigma(x), 0 \right\rangle$. Then $\gamma(a,0)= \left\langle 0,0 \right\rangle$ and $\pi \circ \gamma \circ l(x) = \pi(\sigma(x),0) = 0$ which implies $\gamma \in \ker \psi$. We also easily see that $\Delta(\gamma) = \sigma$ by the identification $I$ with $I \times \{0\}$. The last step is to show $\gamma$ is a derivation. It is easy to see that $\gamma$ is a module homomorphism. For $f \in F$ with $n = \ar f$, we calculate using that $\sigma \in \mathrm{Der} (Q,I,\ast)$, 
\begin{align*}
\gamma \circ F_{f} \left( \left\langle a_1,x_1 \right\rangle,\ldots, \left\langle a_n,x_n \right\rangle \right) = \gamma \left( \, \sum_{i \in [n]} a(f,i)(\vec{x},\vec{a}) + T_{f}(\vec{x}), f^{Q}(\vec{x}) \, \right)  &= \left\langle \, \sigma \circ f^{Q}(\vec{x}), 0 \, \right\rangle  \\
&= \left\langle \, \sum_{i \in [n]} a(f,i)(\vec{x},\sigma(\vec{x})), 0 \, \right\rangle  \\
&= \sum_{i \in [n]} F_{f} \left( \vec{x} \right)_{i} [\sigma(\vec{x})]
\end{align*}
by realization. We have finished showing $\ker \psi \approx \mathrm{Der}(Q,I,\ast)$.

Next we show $\im \psi \subseteq \ker W_{T}$. Take $\phi \in \mathrm{Der}_{I} M$. From Eq.(\ref{eqn:47}), we have $\phi \circ l(x) - l \circ \pi \circ \phi \circ l(x) \in I$ and so the map $h:Q \rightarrow I$ defined by $h(x):= \phi \circ l(x) - l \circ \phi_{l}(x)$ is the failure of the lifting $l$ to commute with the two derivations. Fix $f \in F$ with $n = \ar f$. Using the fact that both $\phi$ and $\phi_{l}$ are derivations and $l$ is group homomorphism we have
\begin{align*}
\phi \circ F_{f}\left( l(\vec{x}) \right) =  \sum_{i \in [n]} F_{f} \left( l(\vec{x}) \right)_{i} [\phi(l(\vec{x}))] &= \sum_{i \in [n]} F_{f} \left( l(\vec{x}) \right)_{i} [ (h + l \circ \phi_{l})(\vec{x}) ]  \\
&= \sum_{i \in [n]} F_{f} \left( l(\vec{x}) \right)_{i} [ h(\vec{x}) ]  + \sum_{i \in [n]} F_{f} \left( l(\vec{x}) \right)_{i} [ l \circ \phi_{l} (\vec{x}) ]  \\
&= \sum_{i \in [n]} a(f,i)\left( \vec{x}, h(\vec{x}) \right)  +  \sum_{i \in [n]} F_{f} \left( l(\vec{x}) \right)_{i} [ l \circ \phi_{l} (\vec{x}) ].
\end{align*}
We also can calculate
\begin{align*}
\phi \circ F_{f}\left( l(\vec{x}) \right) &= \phi \left( l \left( f^{Q}(\vec{x}) \right) + T_{f}(\vec{x}) \right)  \\
&= \phi \circ l \left( f^{Q}(\vec{x}) \right) + \phi \circ T_{f}(\vec{x})    \\ 
&= h(f^{Q}(\vec{x})) + l \circ \phi_{l}\left( f^{Q}(\vec{x}) \right) + \phi \circ T_{f}(\vec{x})  \\
&= h(f^{Q}(\vec{x})) + \sum_{i \in [n]} l \circ f^{Q}(\vec{x})_{i} [\phi_{l}(\vec{x})] + \phi \circ T_{f}(\vec{x}). 
\end{align*}
Rewriting the above two equations yields
\begin{align*}
\sum_{i \in [n]} a(f,i)\left( \vec{x}, h(\vec{x}) \right) - h(f^{Q}(\vec{x})) &=  \sum_{i \in [n]} l \circ f^{Q}(\vec{x})_{i} [\phi_{l}(\vec{x})] + \phi \circ T_{f}(\vec{x}) - \sum_{i \in [n]} F_{f} \left( l(\vec{x}) \right)_{i} [ l \circ \phi_{l} (\vec{x}) ] \\
&= \phi \circ T_{f}(\vec{x}) - \sum_{i \in [n]} \Big(  F_{f} \left( l(\vec{x}) \right)_{i} [ l \circ \phi_{l} (\vec{x}) ] - l \circ f^{Q}(\vec{x})_{i} [\phi_{l}(\vec{x})] \Big)  \\
&= \phi \circ T_{f}(\vec{x}) - \sum_{i \in [n]} T_{f}(\vec{x})_{i} [l \circ \phi_{l}(\vec{x})] \\
&= T^{(\phi|_{I},\phi_{l})}_{f}.
\end{align*}
This shows $W(\psi(\phi),[T]) = W((\phi|_{I},\phi_{l}), [T]) = 0$ witnessed by $h$.

We now complete the proof of exactness. Assume $(\alpha,\beta) \in c(I,Q,\ast)$ such that $[T^{(\alpha,\beta)}]=0$. There exists $h: Q \rightarrow I$ such that 
\begin{itemize}

	\item $0 = h(x) + h(y) - h(x+y)$;
	
	\item for $r \in R$, $\alpha \circ T_{r}(x) - T_{r}(\beta(x)) = r \cdot h(x) - h(r \cdot x)$;
	
	\item for $f \in F$ with $n = \ar f$, 
\[ 
\alpha \circ T_{f}(\vec{x}) - \sum_{i \in [n]} T_{f}(\vec{x})_{i} [\beta(\vec{x})]  = \sum_{i \in [n]} a(f,i)(\vec{x},h(\vec{x})) - h(f^{Q}(\vec{x})).
\]

\end{itemize}
Define $\phi: I \times Q \rightarrow I \times Q$ by $\phi (a,x):= \left\langle \alpha(a) + h(x), \beta(x) \right\rangle$. It is easy to see that $\phi (I) \subseteq I$ and $\phi|_{I} = \alpha$, and $\pi \circ \phi \circ l(x) = \pi \circ \phi(0,x)= \pi(h(x),\beta(x)) = \beta(x)$ and so $\phi_{l} = \beta$; thus, $\psi(\phi) = (\alpha,\beta)$. We show $\phi \in \mathrm{Der}_{I} M$. 

To see that $\phi$ is a module homomorphism, we calculate
\begin{align*}
\phi \big( r \cdot \left\langle a, x \right\rangle + \left\langle b, y \right\rangle \big) &= \phi \big( \left\langle \, r \cdot a + b + T_{r}(x), r \cdot x + y \, \right\rangle \big)  \\
&= \left\langle \, \alpha(r \cdot a) + \alpha(b) + \alpha \circ T_{r}(x) + h(r \cdot x + y), \beta(r \cdot x + y) \, \right\rangle  \\
&= \left\langle \, r \cdot \alpha(a) + \alpha(b) + \alpha \circ T_{r}(x) + h(r \cdot x) + h(y), r \cdot \beta(x) + \beta(y) \, \right\rangle  \\
&= \left\langle \, r \cdot \alpha(a) + \alpha(b) + T_{r}(\beta(x)) + r \cdot h(x) + h(y), r \cdot \beta(x) + \beta(y) \, \right\rangle  \\
&= \left\langle \,  r \cdot \alpha(a) +  T_{r}(\beta(x)) + r \cdot h(x) , r \cdot \beta(x) \, \right\rangle + \left\langle \, \alpha(b) + h(y) , \beta(y) \, \right\rangle \\
&= r \cdot \left\langle \alpha(a) + h(x) , \beta(x) \right\rangle + \phi(b,y) \\
&= r \cdot \phi\left(a,x \right) + \phi(b,y)
\end{align*}
Now for for $f \in F$ with $n = \ar f$, we have
\begin{align*}
\sum_{i \in [n]} &F_{f} \Big( \left\langle a_{1}, x_{1} \right\rangle, \ldots, \phi \big( \left\langle a_{i}, x_{i} \right\rangle \big) , \ldots, \left\langle a_{n}, x_{n} \right\rangle \Big) \\
&= \sum_{i \in [n]} F_{f} \Big( \left\langle a_{1}, x_{1} \right\rangle, \ldots, \left\langle \alpha(a_{i}) + h(x_{i}), \beta(x_{i}) \right\rangle , \ldots, \left\langle a_{n}, x_{n} \right\rangle \Big)   \\
&= \sum_{i \in [n]} \left\langle \ \sum_{k \in [n]} a(f,k) \big( \vec{x}_{i} [\beta(\vec{x})], (a_{1},\ldots,\alpha(a_{i}) + h(x_{i}),\ldots,a_{n}) \big) + T_{f}(\vec{x})_{i} [\beta(\vec{x})]  , f^{Q}(\vec{x})_{i} [\beta(\vec{x})] \ \right\rangle  \\
&= \Bigg< \ \sum_{i \in [n]} \sum_{k \in [n]} a(f,k) \big( \vec{x}_{i}  [\beta(\vec{x})], \vec{a}_{i}  [\alpha(\vec{a})] \big) + \sum_{i \in [n]} a(f,i) \big( \vec{x}_{i}  [\beta(\vec{x})],  \vec{a}_{i}  [ h(\vec{x}) ] \big)  \\
&\quad + \sum_{i \in [n]} T_{f}( \vec{x} )_{i} [\beta(\vec{x})] , \sum_{i \in [n]} f^{Q} (\vec{x})_{i} [\beta(\vec{x})] \ \Bigg> \\
&= \Bigg< \ \sum_{k \in [n]} \alpha \circ a(f,k)(\vec{x},\vec{a}) + \sum_{i \in [n]} a(f,i) \big( \vec{x}, h(\vec{x}) \big) + \sum_{i \in [n]} T_{f}( \vec{x} )_{i} [\beta(\vec{x})] , \beta \left( f^{Q}(\vec{x}) \right) \ \Bigg> \\
&= \Bigg< \ \sum_{k \in [n]} \alpha \circ a(f,k)(\vec{x},\vec{a}) + \alpha \circ T_{f}(\vec{x}) + h \left( f^{Q}(\vec{x}) \right)   , \beta \left( f^{Q}(\vec{x}) \right) \ \Bigg> \\ \\
&= \phi \left( \, \sum_{k \in [n]} a(f,k)(\vec{x}) + T_{f}(\vec{x}) , f^{Q}(\vec{x}) \, \right)  = \phi \circ F_{f} \left( \left\langle a_{1}, x_{1} \right\rangle, \ldots, \left\langle a_{n}, x_{n} \right\rangle \right)    
\end{align*}
Altogether, we have shown $\phi$ is a derivation and so $(\alpha,\beta) \in \im \psi$.
\end{proof}

For each extension class $[T] \in H^{2}_{\mathcal V}(Q,I,\ast)^{\mathrm{gr}}$, we may take a section of the map $\Psi$ in the Lie algebra extension in Eq~\eqref{eqn:derivext} which determines a Lie algebra 2-cocycle which realizes the ideal-preserving derivations of the extension determined by $T$ as a semidirect product.

\begin{corollary}
Let $\mathcal V$ be a variety of modules expanded by multilinear operations with affine datum $(Q,I,\ast)$. For each $[T] \in H^{2}_{\mathcal V}(Q,I,\ast)^{\mathrm{gr}}$, there are action terms $\hat{\ast}$ compatible with the variety of Lie algebras which realize the semidirect product
\[
\mathrm{Der}_{I} (I \rtimes_{T,\ast} Q) \approx \mathrm{Der}(Q,I,\ast) \rtimes_{\hat{\ast}} \ker W_{T}.
\]
\end{corollary}
\begin{proof}
For each $[T] \in H^{2}_{\mathcal V}(Q,I,\ast)^{\mathrm{gr}}$, Eq~\eqref{eqn:derivext} yields an extension of Lie algebras
\begin{align}
0 \longrightarrow \mathrm{Der}(Q,I,\ast) \longrightarrow \mathrm{Der}_{I} I \rtimes_{T} Q \stackrel{\psi}{\longrightarrow} \ker W_{T} \longrightarrow 0 .
\end{align}
According to the proof of Theorem~\ref{thm:derivations}, a section of $\psi$ can be constructed in the following manner: for each $(\sigma,\kappa) \in \ker W_{T}$ there is a map $h_{(\sigma,\kappa)}: Q \rightarrow I$ which witnesses that $W_{T}(\sigma,\kappa) = [T^{(\sigma,\kappa)}] = 0$; therefore, if we we define
\begin{align}\label{eqn:sectionderiv} 
l_{T}(\sigma,\kappa) (a,x):= \left\langle \, \sigma(a) + h_{(\sigma,\kappa)}(x) , \kappa(x) \, \right\rangle,
\end{align}
then we observed $l^{T}(\sigma,\kappa) \in \mathrm{Der}_{I} ( I \rtimes_{T} Q)$ and $\psi \circ l_{T}(\sigma,\kappa) = (\sigma,\kappa)$. The section $l_{T}$ determines a Lie algebra 2-cocycle $S$ with action terms $\hat{\ast}$ according to Theorem~\ref{thm:multirep}; thus, 
\[
\mathrm{Der}_{I} \, I \rtimes_{T} Q \approx \mathrm{Der}(Q,I,\ast) \rtimes_{S,\hat{\ast}} \ker W_{T}
\] 
is an extension realizing datum $(\ker W_{T}, \mathrm{Der}(Q,I,\ast) , \hat{\ast})$. We show that the function $h: \ker W_{T} \rightarrow \mathrm{Der}(Q,I,\ast)$ given by $h(\sigma,\kappa)(x) := h_{(\sigma,\kappa)}(x)$ witnesses that $[S]=0$ in cohomology. We first use Eq~\eqref{eqn:sectionderiv} to evaluate the factor sets of the 2-cocycle $S$. Since the extension is group-trivial we have 
\begin{align*} 
S^{T}_{+} \big( (\sigma,\kappa),(\gamma,\beta) \big)(a,x) &= \Big( l_{T}(\sigma,\kappa) + l_{T}(\gamma,\beta) - l_{T}(\sigma + \gamma,\kappa + \beta) \Big)(a,x) \\
&= \left\langle h_{(\sigma,\kappa)}(x) + h_{(\gamma,\beta)}(x) - h_{(\sigma + \gamma,\kappa + \beta)}(x) , 0 \right\rangle
\end{align*}
and
\begin{align*} 
S^{T}_{[-,-]} &\big( (\sigma,\kappa),(\gamma,\beta) \big)(a,x) \\
&= \Big( \big[ l_{T}(\sigma,\kappa), l_{T}(\gamma,\beta) \big] - l_{T} \big( [ \sigma , \gamma ] , [ \kappa , \beta ] \big) \Big)(a,x) \\
&= l_{T}(\sigma,\kappa) \circ l_{T}(\gamma,\beta)(a,x) - l_{T}(\gamma,\beta) \circ l_{T}(\sigma,\kappa)(a,x) - l_{T} \big( [ \sigma , \gamma ] , [ \kappa , \beta ] \big)(a,x)   \\
&= \left\langle \sigma \circ h_{(\gamma,\beta)}(x) + h_{(\sigma,\kappa)} \circ \beta(x) - \gamma \circ h_{(\sigma,\kappa)}(x) - h_{(\gamma,\beta)} \circ \kappa(x) - h_{([\sigma,\gamma],[\kappa,\beta])}(x) , 0 \right\rangle \\
&= \left\langle  h_{(\gamma,\beta)}^{(\sigma,\kappa)}(x) - h_{(\sigma,\kappa)}^{(\gamma,\beta)}(x) - h_{([\sigma,\gamma],[\kappa,\beta])}(x) , 0 \right\rangle .
\end{align*}
If we evaluate the action $a([-,-],1): \mathrm{Der}(Q,I,\ast) \times \ker W_{T} \rightarrow \mathrm{Der}(Q,I,\ast)$ induced by $l_{T}$, then 
\begin{align*}
a \big( [-,-],1 \big)(d,(\gamma,\beta))(a,x) &= \big[ d,  l_{T}(\gamma,\beta) \big] (a,x) \\
&= d \circ l_{T}(\gamma,\beta)(a,x) - l_{T}(\gamma,\beta) \circ d(a,x)  \\
&= d( \gamma(a) + h_{(\gamma,\beta)}(x), \beta(x) ) - l_{T}(\gamma,\beta)(d(x),0)  \\
&= \left\langle d \circ \beta (x) - \gamma \circ d(x) , 0 \right\rangle \\
&= \left\langle - d^{(\gamma,\beta)}(x)  , 0 \right\rangle ;  
\end{align*}
similarly, $a([-,-],2)((\sigma,\kappa),d)(a,x) = \left\langle d^{(\sigma,\kappa)}(x) , 0 \right\rangle$. From the previous calculations, we see that $[S] = 0$.
\end{proof}


\section{Low-dimensional Hochschild-Serre sequence}\label{section:5}
\vspace{0.3cm}

In this section, we establish a Hochschild-Serre exact sequence \cite{hochserre} for the first and second cohomology groups associated to affine datum and a general extension in a variety of modules expanded by multilinear operations. The inflation, restriction and transgression maps will be defined the same way as in the group case; however, there will be two main alterations in the development. For each of the cohomology groups, rather than just the cohomology associated to the quotient algebra, we must restrict to a certain null submodule of the given affine datum. The second alteration is the introduction of the $\square$-condition of an action which will restrict the domain of the transgression map on $1^{\mathrm{st}}$-cohomology. When restricted to the null submodule, the $\square$-condition is greatly simplified.

\begin{theorem}\label{thm:HSexact}
Let $\mathcal V$ be a variety of modules expanded by multilinear operations and $M \in \mathcal V$ which is an extension $0 \rightarrow I \rightarrow M \rightarrow Q \rightarrow 0$. Let $(M,A,\ast)$ be affine datum in $\mathcal V$. There is an exact sequence
\[
0 \longrightarrow H^{1}(Q,A^{I},\hat{\ast}) \stackrel{\sigma}{\longrightarrow} H^{1}(M,A^{I},\ast) \stackrel{r}{\longrightarrow} H^{1}(I,A^{I},\ast)^{\square} \stackrel{\delta_{T}}{\longrightarrow} H^{2}_{\mathcal V}(Q,A^{I},\hat{\ast}) \stackrel{\sigma}{\longrightarrow} H^{2}_{\mathcal V}(M,A^{I},\ast).
\]
\end{theorem}

We begin by developing the terms in the statement of the theorem. Fix an extension $0 \rightarrow I \rightarrow M \rightarrow Q \rightarrow 0$ and affine datum $(M,A,\ast)$ in $\mathcal V$. By Theorem~\ref{thm:multiabel}, the multilinear operations in $A$ are all trivial and the operations in the action $M \ast A$ are all unary in $A$ and given by the sequence $\{a(f,i): f \in F, i \in [\ar f] \}$. The extension $\pi: M \rightarrow Q$ induces an action $Q \star I$ with operations given by $\{ b(f,s): f \in F, s \in [\ar f]^{\ast} \}$. Let $T$ be the 2-cocycle determined by the extension so that $\pi: M \rightarrow Q$ is equivalent to $p_{2}: I \rtimes_{T} Q \rightarrow Q$. For simplicity, we make the identification $M=I \rtimes_{T} Q \stackrel{\pi}{\rightarrow} Q$.

Let $\mathcal S^{\ast}$ be the set of operations which are formed by composing the action terms of $M \ast A$ with each other; that is, $\mathcal S^{\ast}$ is the smallest set of operations such that each $a(f,i) \in \mathcal S^{\ast}$ for $f \in F$, $i \in [\ar f]$ and whenever $t(x,\vec{y}) \in \mathcal S^{\ast}$, we have $a(f,i)(z_{1},\ldots,z_{i-1},t(x,\vec{y}),z_{i+1},\ldots,z_{\ar f}) \in \mathcal S^{\ast}$ for $f \in F$, $i \in [\ar f]$. For the ideal $I \triangleleft M$, the \emph{null} submodule of $A$ determined by $I$ is 
\begin{align}\label{eqn:induceaction}
A^{I} = \{ a \in A : t(a,\vec{m})=0 \text{ for all } t \in \mathcal S^{\ast} \text{ with some } m_{j} \in I \}.
\end{align}
We see that $A^{I}$ is indeed a submodule since the action terms are unary in $A$.

Let us first observe for the 2-cocycle $T$ that for the group part we have $T_{+}(0,x) = T_{+}(x,0) = 0$ which implies we can write $\left\langle b, x \right\rangle = \left\langle b, 0 \right\rangle + \left\langle 0, x \right\rangle$ in $M = I \rtimes_{T} Q$. Note by realization in $A \rtimes_{\ast} M$, for $f \in F$ with $\ar f = n$, if we take $a \in A^{I}$ and write $m_i = \left\langle b_{i}, x_{i} \right\rangle \in I \times Q$, then 
\begin{align*}
\left\langle \ a(f,i)(m_{1},\ldots,a_{i},\ldots,m_{n}), 0 \ \right\rangle &= F_{f} \big( \left\langle 0, m_{1} \right\rangle, \ldots, \left\langle a, 0 \right\rangle, \ldots, \left\langle 0, m_{n} \right\rangle \big) \\
&= F_{f} \big( \left\langle 0, \left\langle b_{1}, 0 \right\rangle + \left\langle 0, x_{1} \right\rangle \right\rangle, \ldots, \left\langle a, 0 \right\rangle, \ldots, \left\langle 0, m_{n} \right\rangle \big) \\
&= F_{f} \big( \left\langle 0, \left\langle b_{1}, 0 \right\rangle \right\rangle + \left\langle  0, \left\langle 0, x_{1} \right\rangle \right\rangle, \ldots,\left\langle a, 0 \right\rangle, \ldots, \left\langle 0, m_{n} \right\rangle \big) \\
&= F_{f} \big( \left\langle 0, \left\langle b_{1}, 0 \right\rangle \right\rangle, \ldots,\left\langle a, 0 \right\rangle, \ldots, \left\langle 0, m_{n} \right\rangle \big) \\
&\quad \quad + \ \ F_{f} \big( \left\langle  0, \left\langle 0, x_{1} \right\rangle \right\rangle, \ldots,\left\langle a, 0 \right\rangle, \ldots, \left\langle 0, m_{n} \right\rangle \big) \\
&= F_{f} \big( \left\langle 0, \left\langle 0, x_{1} \right\rangle \right\rangle, \ldots, \left\langle a, 0 \right\rangle, \ldots, \left\langle 0, m_{n} \right\rangle \big) \\
&\vdots \\
&= F_{f} \big( \left\langle 0, \left\langle 0, x_{1} \right\rangle \right\rangle, \ldots,\left\langle a, 0 \right\rangle, \ldots, \left\langle 0, \left\langle 0, x_{1} \right\rangle \right\rangle\big)
\end{align*}
which establishes that 
\begin{align}\label{eqn:welldef}
a(f,i) \left( m_{1},\ldots,a,\ldots,m_{n} \right) = a(f,i)\left( \left\langle 0,x_{1} \right\rangle,\ldots,a,\ldots,\left\langle 0,x_{n} \right\rangle \right).
\end{align}
Inductively on the composition of action terms we see that for $t(x,\vec{y}) \in \mathcal S^{\ast}$, $a \in A^{B}$ and $m_i = \left\langle b_{i}, x_{i} \right\rangle \in I \times Q$ that 
\begin{align}\label{eqn:welldef2}
t(a,\vec{m}) &= t(a,\left\langle 0, \vec{x}\right\rangle) &\quad \quad \quad \quad \quad \quad &(t(x,\vec{y}) \in \mathcal S^{\ast}).
\end{align}

For the ideal $I \triangleleft M$, we define an action $Q \hat{\ast} A^{I}$ such that $(Q,A^{I},\hat{\ast})$ is affine datum in $\mathcal V$. For each $f \in F$ and $i \in [ \ar f]$, define
\begin{align}\label{eqn:actioninflated}
\hat{a}(f,i)(x_{1},\ldots,x_{i-1},a,x_{i+1},\ldots,x_{\ar f}) := a(f,i) \left( \left\langle 0,x_{1} \right\rangle, \ldots, \left\langle 0,x_{i-1} \right\rangle, a, \left\langle 0, x_{i+1} \right\rangle, \ldots, \left\langle 0,x_{\ar f} \right\rangle \right)
\end{align}
for any $a \in A^{I}, x_{1},\ldots,x_{i-1},x_{i+1},\ldots,x_{\ar f} \in Q$. By Eq~\eqref{eqn:welldef2} and because the definition in Eq~\eqref{eqn:induceaction} is taken over all terms in $\mathcal S^{\ast}$, we see that the action is well-defined and closed on $A^{I}$. We can succinctly write the definition relating the two actions as
\begin{align}\label{eqn:definfl}
\hat{a}(f,i) \left( \pi(\vec{m}), \vec{a} \right) = a(f,i) \left( \vec{m}, \vec{a} \right)
\end{align}
for $f \in F$, $\vec{m} \in M^{\ar f}$. If we note that for terms $s$ in the signature of the variety $\mathcal V$, the action terms $s^{\ast,T^{\ast}}$ from Lemma~\ref{lem:semident} are sums of iterated action terms from $\mathcal S^{\ast}$, then it is easy to that $Q \hat{\ast} A^{I}$ is also a $\mathcal V$-compatible action.

The \emph{inflation} maps 
\begin{align*}
\sigma &: H^{1}( Q,A^{I},\hat{\ast} ) \rightarrow H^{1}( M,A,\ast ) \quad \quad \quad \quad \sigma : H^{2}_{\mathcal V}( Q,A^{I},\hat{\ast} ) \rightarrow H^{2}_{\mathcal V}( M,A,\ast )
\end{align*}
are defined by precomposition for derivations $\sigma([d]) := [d \circ \pi]$ and for 2-cocycles $\sigma([T]) := [T \circ \pi]$ where $(T \circ \pi)_{+}(x,y) := T_{+}(\pi(x),\pi(y))$, $(T \circ \pi)_{+}(x) := T_{r}(\pi(x))$ and $(T \circ \pi)_{f}(x_{1},\ldots,x_{\ar f}) := T_{f}(\pi(x_{1}),\ldots,\pi(x_{\ar f}))$ for $r \in R$, $f \in F$ with action terms given by $M \ast A$. Let us show inflation is well-defined on cohomology; that is, precomposition maps coboundaries to coboundaries. First, suppose $G$ is a 2-coboundary for $(Q,A^{I},\hat{\ast})$ witnessed by $h: Q \rightarrow A^{I}$. Using (B1)-(B3) for trivial multilinear operations and unary actions, we see that for $z,w \in M, \vec{z} \in M^{\ar f}$ and $x = \pi(w), y = \pi(z), \vec{x} = \pi(\vec{z})$ we have
\begin{align*}
G_{+}(\pi(w),\pi(z)) = G_{+}(x,y) = h(x) + h(y) - h(x+y) &= h(\pi (w)) + h (\pi (z)) - h(\pi (x)+ \pi(y)) \\
&= h \circ \pi (w) + h \circ \pi (z) - h \circ \pi (x+y),
\end{align*}
\begin{align*}
G_{r}(\pi(w)) = G_{r}(x) = r \cdot h(x) - h(r \cdot x) =  r \cdot h(\pi(w)) - h(r \cdot \pi(w)) &= r \cdot h \circ \pi(w) - h \circ \pi (r \cdot w)
\end{align*}
and by the definition in Eq~\eqref{eqn:definfl}
\begin{align*}
G_{f}(\pi(\vec{z})) = G_{f}(\vec{x}) = \sum_{i \in [\ar f]^{\ast}} \hat{a}(f,i)(\vec{x},h(\vec{x})) - h(f^{Q}(\vec{x})) &= \sum_{i \in [\ar f]^{\ast}} \hat{a}(f,i)(\pi(\vec{z}),h(\pi(\vec{z}))) - h(f^{Q}(\pi(\vec{z}))) \\
&= \sum_{i \in [\ar f]^{\ast}} a(f,i)(\vec{z},h \circ \pi(\vec{z})) - h \circ \pi (f^{M}(\vec{z})).
\end{align*}
This shows $G \circ \pi$ is a 2-coboundary $(Q,A,\ast)$ witnessed by $h \circ \pi$. Observe that since $Q \hat{\ast} A^{I}$ is $\mathcal V$-compatible, we see that $[T] \in H^{2}_{\mathcal V}(Q,A^{I},\hat{\ast})$ if and only if $T$ is strictly $\mathcal V$-compatible. Then for the same reason as in the calculation verifying the (B3) property for the 2-coboundary above, $T \circ \pi$ is $\mathcal V$-compatible using the action $M \ast A$; that is, $[T \circ \pi] \in H^{2}_{\mathcal V}(M,A,\ast)$.

We now consider $1^{\mathrm{st}}$-cohomology. If we take a derivation $d: Q \rightarrow A^{I}$ of the action $\hat{\ast}$, then we can verify for $m_{i} = \left\langle b_i, x_{i} \right\rangle$ and $f \in F$ with $n = \ar f$ that
\begin{align*}
d \circ \pi \left( f^{M}(\vec{m}) \right) = d \left( f^{Q}(\vec{x}) \right) = \sum_{i \in [n]} \hat{a}(f,i) \left( \vec{x}, d(\vec{x}) \right) =  \sum_{i \in [n]} a(f,i) \left( \vec{m}, d \circ \pi (\vec{m}) \right),
\end{align*}
and so $d \circ \pi$ is a derivation of $(M,A,\ast)$. Now let $\gamma$ be a principal stabilizing automorphism of the datum $(Q,A^{I},\hat{\ast})$; that is, there is a term $t(x,\vec{y})$ and elements $\vec{c},\vec{d} \in (A^{I})^{n}$ such that in $A^{I} \rtimes_{\hat{\ast}} Q$
\begin{align*}
\left\langle a - d_{\gamma}(x), x \right\rangle = \gamma(a,x) = F_{t}\left( \left\langle a, x \right\rangle, \left\langle c_{1}, 0 \right\rangle, \ldots, \left\langle c_{n}, 0 \right\rangle \right) &=  \left\langle a + t^{\hat{\ast},T^{\hat{\ast}}}((0,\vec{c}),(x,\vec{0})) , x \right\rangle \\
\left\langle a, x \right\rangle = F_{t}\left( \left\langle a, x \right\rangle, \left\langle d_{1}, 0 \right\rangle, \ldots, \left\langle d_{n}, 0 \right\rangle \right) &=  \left\langle a + t^{\hat{\ast},T^{\hat{\ast}}}((0,\vec{d}),(x,\vec{0})) , x \right\rangle\\
\end{align*}
where $a=t(a,\vec{c})=t(a,\vec{d})$ and $d_{\gamma}$ is a derivation with $d_{\gamma}(x) = t^{\hat{\ast},T^{\hat{\ast}}}((0,\vec{c}),(x,\vec{0}))$ (see Remark~\ref{rem:principal}). Then in the semidirect product $A \rtimes_{\ast} M$ we have
\begin{align*}
F_{t} &\left( \left\langle a, \left\langle b,x \right\rangle \right\rangle, \left\langle c_{1}, \left\langle 0,0\right\rangle \right\rangle, \ldots, \left\langle c_{n}, \left\langle 0,0\right\rangle \right\rangle \right) - F_{t}\left( \left\langle a, \left\langle b,x \right\rangle \right\rangle, \left\langle d_{1}, \left\langle 0,0\right\rangle \right\rangle, \ldots, \left\langle d_{n}, \left\langle 0,0\right\rangle \right\rangle \right) + \left\langle a, \left\langle b,x \right\rangle \right\rangle \\
&= \left\langle a + t^{\ast,T^{\ast}} \left( (0,\vec{c}),(b,x),(0,0),\ldots,(0,0) \right) , t^{M}((b,x),(0,0),\ldots,(0,0)) \right\rangle \\
&- \left\langle a + t^{\ast,T^{\ast}} \left( (0,\vec{d}),(b,x),(0,0),\ldots,(0,0) \right) , t^{M}((b,x),(0,0),\ldots,(0,0)) \right\rangle + \left\langle a, \left\langle b,x \right\rangle \right\rangle \\
&= \left\langle a + t^{\hat{\ast},T^{\hat{\ast}}} \left( (0,\vec{c}),(x,\vec{0}) \right) , x \right\rangle - \left\langle a + t^{\hat{\ast},T^{\hat{\ast}}} \left( (0,\vec{d}),(x,\vec{0}) \right) , x \right\rangle + \left\langle a, \left\langle b,x \right\rangle \right\rangle \\
&= \left\langle a - d(x), x \right\rangle - \left\langle a, x \right\rangle + \left\langle a, \left\langle b,x \right\rangle \right\rangle \\
&= \left\langle a - d \circ \pi (b,x), \left\langle b,x \right\rangle \right\rangle .
\end{align*}
Now that we have shown inflation is well-defined on cohomology classes, it is easy to see that it is a homomorphism.

There is a \emph{restriction} map
\begin{align*}
r: H^{1}( M,A,\ast ) \rightarrow H^{1}( I,A,\ast )
\end{align*} 
defined by $r([d]):= [d|_{I}]$; that is, we take the cohomology class of the derivation $d$ restricted to $I$. Note that in the polynomial which defines a principal stabilizing automorphism in $A \rtimes_{T^{\ast}} M$, the constants come from the isomorphic copy of $A$; thus, the restriction of the automorphism to $A \rtimes_{T^{\ast}} I \leq A \rtimes_{T^{\ast}} M$ is still a principal stabilizing automorphism according to definition. It is easy to see that the restriction map is a homomorphism.

We are interested in a property of the derivations which are in the image of the restriction map. Fix a derivation $d: M \rightarrow A$, $f \in F$ with $n=\ar f$ and for simplicity, let us consider an action term for the coordinates $s = \{i, i+1,\ldots,n\}$. Then 
\begin{align*}
d|_{I} \big( b(f,s)(x_{1},\ldots,x_{i-1},b_{i},\ldots,b_{n}) \big) &= d \left( f^{M}( \left\langle 0, x_{1} \right\rangle,\ldots,\left\langle 0, x_{i-1} \right\rangle, \left\langle b_{i}, 0 \right\rangle,\ldots, \left\langle b_{n}, 0 \right\rangle ) \right) \\
&= \sum_{k=1}^{i-1} a(f,k)( \left\langle 0, x_{1} \right\rangle,\ldots, d(0,x_{k}) ,\ldots, \left\langle 0, x_{i-1} \right\rangle, \left\langle b_{i}, 0 \right\rangle,\ldots, \left\langle b_{n}, 0 \right\rangle ) \\
&+ \quad \sum_{k=i}^{n} a(f,k)( \left\langle 0, x_{1} \right\rangle, \ldots, \left\langle 0, x_{i-1} \right\rangle, \left\langle b_{i}, 0 \right\rangle,\ldots,d(b_{k},0), \ldots, \left\langle b_{n}, 0 \right\rangle ) \\
&= \sum_{k=1}^{i-1} a(f,k)( \left\langle 0, x_{1} \right\rangle,\ldots, h(x_{k}) ,\ldots, \left\langle 0, x_{i-1} \right\rangle, \left\langle b_{i}, 0 \right\rangle,\ldots, \left\langle b_{n}, 0 \right\rangle ) \\
&+ \quad \sum_{k=i}^{n} a(f,k)( \left\langle 0, x_{1} \right\rangle, \ldots, \left\langle 0, x_{i-1} \right\rangle, \left\langle b_{i}, 0 \right\rangle,\ldots,d|_{I}(b_{k}), \ldots, \left\langle b_{n}, 0 \right\rangle ) 
\end{align*}
where we have written $h(x):=d(0,x) : Q \rightarrow A$. This determines a condition on derivations which relates the given affine action $M \ast A$ and the action $Q \star I$ induced from the extension $\pi: M \rightarrow Q$. A derivation $d \in \mathrm{Der}(I,A,\ast)$ satisfies the \emph{$\square$-condition} if there exists a map $h: Q \rightarrow A$ such that 
\begin{align}\label{eqn:square}
d \left( b(f,s)(\vec{x}, \vec{b} ) \right) = \sum_{i \not\in s} a(f,i)\left( \left\langle \vec{0}, \vec{x} \right\rangle_{s} \left\langle \vec{b}, \vec{0} \right\rangle , h(\vec{x}) \right) \ + \ \sum_{i \in s} a(f,i)\left( \left\langle \vec{0}, \vec{x} \right\rangle_{s} \left\langle \vec{b}, \vec{0} \right\rangle , d(\vec{b}) \right)
\end{align}
for all $f \in F$, $s \in [\ar f]^{\ast}$. Note we are using the convention to write the tuple $\left\langle \vec{b}, \vec{x} \right\rangle = \left( \left\langle b_{1}, x_{1} \right\rangle,\ldots, \left\langle b_{n}, x_{n} \right\rangle \right)$ and $\left\langle \vec{0}, \vec{x} \right\rangle_{s} \left\langle \vec{b}, \vec{0} \right\rangle$ is the tuple formed by substituting the s-coordinates of $\left\langle \vec{b}, \vec{0} \right\rangle$ into the s-coordinates of $\left\langle \vec{0}, \vec{x} \right\rangle$. Since derivations are linear module transformation and the right-hand side of the definition in Eq~\eqref{eqn:square} is in terms of the unary actions, the set of derivations which satisfy the $\square$-condition forms a subgroup. We then define  
\[
H^{1}(I,A,\ast)^{\square} = \{ [d] \in H^{1}(I,A,\ast): d \text{ satisfies the } \square-\text{condition} \}
\]
which forms a subgroup of $1^{\mathrm{st}}$-cohomology. The above calculation can be extended to show the restriction map $r: H^{1}( M,A,\ast ) \rightarrow H^{1}( I,A,\ast )^{\square}$ is well-defined.

We now consider our analogue of the trangression map which takes $1^{\mathrm{st}}$-cohomology to $2^{\mathrm{nd}}$-cohomology. In order for the codomain of the transgression to be contained in the set of $\mathcal  V$-compatible 2-cocycles, we must enforce the $\square$-condition on derivations of the datum $(I,A^{I},\ast)$; that is, on the ``coefficients'' $A^{I}$ and not just $A$. Let us consider why this is so. Take $[d] \in H^{1}(I,A^{I},\ast)^{\square}$ where $d$ satisfies the $\square$-condition; thus, there is a map $h: Q \rightarrow A^{I}$ which satisfies Eq~\eqref{eqn:square}. Consider $|s| > 1$. Then for any $i \not\in s$, $a(f,i)\left( \left\langle \vec{0}, \vec{x} \right\rangle_{s} \left\langle \vec{b}, \vec{0} \right\rangle , h(\vec{x}) \right) = 0$ since both $\left\langle b_{j}, 0 \right\rangle$ for $j \in s$ and $h(x_{i})$ appear in the coordinates. For any $i \in s$, there is $i \neq j \in$ such that both $\left\langle b_{j}, 0 \right\rangle$ and $d(b_{i}) \in A^{I}$ appear in the coordinate; thus, we again see that $a(f,i)\left( \left\langle \vec{0}, \vec{x} \right\rangle_{s} \left\langle \vec{b}, \vec{0} \right\rangle , d(\vec{b}) \right) = 0$. Altogether, we conclude that 
\begin{align}\label{eqn:square1}
d \big( b(f,s)(\vec{x},\vec{b}) \big) = 0 \quad \quad \text{for} \quad \quad |s| > 1.
\end{align}
Now fix $s=\{k\}$. Then for the same reason we have $a(f,i)\left( \left\langle \vec{0}, \vec{x} \right\rangle_{s} \left\langle \vec{b}, \vec{0} \right\rangle , h(\vec{x}) \right) = 0$ for $i \neq k$. Then in this case we have
\begin{align}\label{eqn:square2}
d \big( b(f,i)(\vec{x},\vec{b}) \big) = a(f,i)( \left\langle 0,x_{1} \right\rangle, \ldots, d(b_{i}), \ldots, \left\langle 0,x_{n} \right\rangle ) = \hat{a}(f,i)(\vec{x},d(\vec{b})).
\end{align}
So for reference later, we conclude that a derivation $d$ of the datum $(I,A^{I},\ast)$ which satisfies the $\square$-condition implies
\begin{align}\label{eqn:squarenull}
d \big( b(f,i)(\vec{x},\vec{b}) \big) = \hat{a}(f,i)(\vec{x},d(\vec{b})) \quad \quad \quad \text{ and } \quad \quad \quad  d \big( b(f,s)(\vec{x},\vec{b}) \big) = 0 \quad \quad \text{for} \quad \quad |s| > 1.
\end{align}
We can very roughly summarize Eq~\eqref{eqn:square1} and Eq~\eqref{eqn:square2} by stating that the $\square$-condition on the datum $(I,A^{I},\ast)$ implies the relation
\begin{align}\label{eqn:dersquare}
d(Q \mathrel{\star} I) = Q \mathrel{\hat{\ast}} d(I).
\end{align}
We can also observe for the semidirect product $A^{I} \rtimes_{\ast} I = A^{I} \times I$ since the action of $I$ on $A^{I}$ is null. Then by Remark~\ref{rem:principal} we see that the principal derivations of the datum $(I,A^{I},\ast)$ must be trivial; therefore, $H^{1}(I,A^{I},\ast) = \mathrm{Der}(I,A^{I},\ast)$.

The \emph{transgression} map
\begin{align*}
\partial : H^{1}(I,A^{I},\ast)^{\square} \times H^{2}_{\mathcal V}(Q,I) \rightarrow H^{2}_{\mathcal V}(Q,A^{I},\hat{\ast})
\end{align*}
is defined by $\delta ([d],[T]):= [d \circ T]$ where 
\[
(d \circ T)_{+} = d \circ T_{+}, \quad \quad (d \circ T)_{r} = d \circ T_{r}, \quad \quad (d \circ T)_{f} = d \circ T_{f}
\]
and the action terms are exactly the action terms of $Q \mathrel{\hat{\ast}} A^{I}$. In order to show that it is well-defined on cohomology classes, let us first take a derivation $d:I \rightarrow A^{I}$. Then observe that for any $f \in F$ with $n=\ar f$,
\begin{align}\label{eqn:nullder}
d ( f^{I}(\vec{b}) ) = \sum_{i \in [n]} a(f,i) \left( \left\langle \vec{0}, \vec{b} \right\rangle, d(\vec{b}) \right) = 0
\end{align}
since $d(b_i) \in A^{I}$. This fact, together with Eq~\eqref{eqn:square1} and Eq~\eqref{eqn:square2}, implies that if we take $[d] \in H^{1}(I,A^{I},\ast)^{\square}$ and $T \sim T'$ witnessed by $h: Q \rightarrow I$, then $d \circ T \sim d \circ T'$ witnessed by $d \circ h: Q \rightarrow A^{I}$; therefore, the transgression does not depend on our choice of representative 2-cocycle for the extension $\pi: M \rightarrow Q$.

The final step is to verify that $[d \circ T] \in H^{2}_{V}(Q, A^{I},\hat{\ast})$ for any $d \in H^{1}(I,A^{I},\ast)^{\square}$. The relation Eq~\eqref{eqn:dersquare} expresses the central reason why $d \circ T$ should be a $\mathcal V$-compatible 2-cocycle for datum $(Q, A^{I},\hat{\ast})$, but a more formal argument is based on the representation of terms evaluated in the algebra $A^{I} \rtimes_{d \circ T} Q$. Let $p_{1}: A^{I} \rtimes_{d \circ T} Q \rightarrow A^{I}$ denote the first-projection set map. Then using Eq~\eqref{eqn:square1}, Eq~\eqref{eqn:square2} and Eq~\eqref{eqn:nullder}, one can show inductively on the generation of terms that for any term $t(\vec{x})$ in the language of $\mathcal V$, we have the evaluation  
\begin{align}\label{eqn:100}
F_{t} \left( \left\langle d \circ p_{1}(\vec{m}), \pi(\vec{m}) \right\rangle \right) = \left\langle d \circ p_{1} \left( F_{t}^{M}(\vec{m}) \right) , t^{Q}(\pi(\vec{m})) \right\rangle \quad \quad \quad \quad \quad \left( \vec{m} \in (A^{I} \times Q )^{\ar t} \right)
\end{align}
computed in the algebra $A^{I} \rtimes_{d \circ T} Q$. Since we have already seen that the action $Q \hat{\ast} A^{I}$ is $\mathcal V$-compatible, the 2-cocycle $d \circ T$ is $\mathcal V$-compatible if and only if it is strictly $\mathcal V$-compatible. Take $t=s \in \mathrm{Id} \, \mathcal V$. Since $M \in \mathcal V$, we have $F^{M}_{t}(\vec{m}) = F^{M}_{s}(\vec{m})$. Then using Remark~\ref{remark:affineident} for the evaluation of terms with affine datum, Eq~\eqref{eqn:100} yields
\begin{align*}
t^{\hat{\ast},d \circ T}( d \circ p_{1}(\vec{m}), \pi(\vec{m}) ) + t^{\partial,d \circ T}(\pi(\vec{m})) &= d \circ p_{1} \left( F_{t}^{M}(\vec{m}) \right) \\
&= d \circ p_{1} \left( F_{s}^{M}(\vec{m}) \right) \\
&= s^{\hat{\ast},d \circ T}( d \circ p_{1}(\vec{m}), \pi(\vec{m}) ) + s^{\partial,d \circ T}(\pi(\vec{m}))
\end{align*}
Since the action terms of $d \circ T$ are given by the $\mathcal V$-compatible action $Q \hat{\ast} A^{I}$, we conclude from the above that $t^{\partial,d \circ T}(\pi(\vec{m})) = s^{\partial,d \circ T}(\pi(\vec{m}))$; therefore, $d \circ T$ is strictly $\mathcal V$-compatible.

\begin{proof}(of Theorem~\ref{thm:HSexact})
We show exactness at each of the groups.

$H^{1}(Q,A^{I},\hat{\ast})$ : Take $[d] \in H^{1}(Q,A^{I},\hat{\ast})$ such that $\sigma([d]) = [d \circ \pi] = 0$. Then there is a term $t(x,\vec{y})$ and $\vec{c},\vec{d} \in (A^{I})^{n}$ such that
\begin{align*}
\left\langle a - d \circ \pi (b,x), \left\langle b,x \right\rangle \right\rangle &= F_{t} \left( \left\langle a, \left\langle b,x \right\rangle \right\rangle, \left\langle c_{1} , \left\langle 0, 0 \right\rangle \right\rangle, \ldots, \left\langle c_{n}, \left\langle 0, 0 \right\rangle \right\rangle \right) \\
&= \left\langle t^{A^{I}}(a,\vec{c}) + t^{\ast,T^{\ast}}((a,\vec{c}), \left\langle b,x \right\rangle) , t^{M}(\left\langle b, x \right\rangle,\left\langle 0, 0 \right\rangle, \ldots, \left\langle 0, 0 \right\rangle ) \right\rangle \\
  \left\langle a, \left\langle b,x \right\rangle \right\rangle &= F_{t} \left( \left\langle a, \left\langle b,x \right\rangle \right\rangle, \left\langle d_{1} , \left\langle 0, 0 \right\rangle \right\rangle, \ldots, \left\langle d_{n}, \left\langle 0, 0 \right\rangle \right\rangle \right)
\end{align*}
in the semidirect product $A^{I} \rtimes_{\ast} M$. Since $a \in A^{I}$ we can rewrite in terms of the induced action $Q \hat{\ast} A^{I}$
\begin{align*}
a - d(x) = a - d \circ \pi (0,x) &= t^{A^{I}}(a,\vec{c}) + t^{\ast,T^{\ast}}((a,\vec{c}), \left\langle 0,x \right\rangle) = t^{A^{I}}(a,\vec{c}) + t^{\hat{\ast},T^{\hat{\ast}}}((a,\vec{c}), x).
\end{align*}
Then 
\begin{align*}
F_{t}\left( \left\langle a, x \right\rangle,\ldots, \left\langle c_{1}, 0 \right\rangle, \ldots, \left\langle c_{n}, 0 \right\rangle \right) &= \left\langle t^{A^{I}}(a,\vec{c}) +t^{\hat{\ast},T^{\hat{\ast}}}((a,\vec{c}), x), t^{Q}(x,0,\ldots,0) \right\rangle  = \left\langle a - d(x), x  \right\rangle. 
\end{align*}
A similar calculation applies to the tuple $\vec{d} \in (A^{I})^{n}$ to show $d$ is a principal derivation of the datum $(Q,A^{I},\hat{\ast})$.

$H^{1}(M,A^{I},\ast)$ : It is clear that $r \circ \sigma = 0$. Take $[d] \in H^{1}(M,A^{I},\ast)$ such that $r(d) = d|_{I} \equiv 0$. Note in $M=I \times_{T} Q$ we can write $d(b,x) = d \big( \left\langle b, 0 \right\rangle + \left\langle 0, x \right\rangle \big) = d(b,0) + d(0,q) = d(0,q)$. Define $\gamma: Q \rightarrow A^{I}$ by $\gamma(x):= d(0,q)$. Then $d = \gamma \circ \pi$ shows $\sigma([\gamma]) = [d]$. That $\gamma$ is a derivation of the datum $(Q,A^{I},\hat{\ast})$ follows from that fact that $d$ is independent of the first-coordinate. We have shown $\im \sigma = \ker r$.

$H^{1}(I,A^{I},\ast)^{\square}$ : Take $[d] \in H^{1}(M,A^{I},\ast)$. We want to show $\delta_{T} \circ r ([d]) = [d|_{I} \circ T] = 0$ in $H^{2}_{\mathcal V}(Q,A^{I},\hat{\ast})$. Define $\psi: Q \rightarrow A^{I}$ by $\psi(x):= d(0,x)$. Then we see that for $f \in F$ with $n = ar f$,
\begin{align*}
\sum_{i \in [n]} \hat{a}(f,i)(x_{1},\ldots,\psi(x_{i}),\ldots,x_{n}) &= \sum_{i \in [n]} a(f,i)(\left\langle 0 ,x_{1} \right\rangle,\ldots,d0,(x_{i}),\ldots,\left\langle 0 ,x_{n} \right\rangle) \\
&= d\left( f^{M}\left( \left\langle 0 ,x_{1} \right\rangle,\ldots, \left\langle 0 ,x_{n} \right\rangle \right) \right) \\
&= d \left( T_{f}(\vec{x}), f^{Q}(\vec{x}) \right) \\
&= d \left( T_{f}(\vec{x}), 0 \right) + d\left(0, f^{Q}(\vec{x}) \right) = d|_{I} \circ T_{f}(\vec{x}) + \psi \left( f^{Q}(\vec{x}) \right).
\end{align*}
Similarly, we have
\begin{align*}
d|_{I} \circ T_{+}(x,y) + \psi(x+y) = d \left( T_{+}(x,y), x+y \right) = d\left( \left\langle 0, x \right\rangle + \left\langle 0, y \right\rangle \right) = d(0,x) + d(y) = \psi(x) + \psi(y)
\end{align*}
and
\begin{align*}
d|_{I} \circ T_{r}(x) + \psi(r \cdot x) = d \left( T_{r}(x), r \cdot x \right) = d \left( r \cdot \left\langle 0, x \right\rangle \right) = r \cdot d(0,x) = r \cdot \psi(x).
\end{align*}
Altogether, we have shown $\psi$ witnesses that $d|_{I} \circ T$ is a 2-coboundary of the datum $(Q,A^{I},\hat{\ast})$.

Now, take $d \in H^{1}(I,A^{I},\ast)^{\square}$ such that $[d \circ T]=0$. Then there is a map $h: Q \rightarrow A^{I}$ such that
\begin{enumerate}

	\item $d \circ T_{+}(x,y) = h(x) + h(y) - h(x+y)$;
	
	\item $d \circ T_{r}(x) = r \cdot h(x) - h(r \cdot x)$;
	
	\item $d \circ T_{f}(\vec{x}) = \sum_{i \in [n]} \hat{a}(f,i)(x_{1},\ldots,h(x_{i}),\ldots,x_{n}) - h(f^{Q}(\vec{x}))$  \quad \quad \quad \quad \quad \quad $(f \in F, n = \ar f)$.

\end{enumerate}
Define $\gamma: M \rightarrow A^{I}$ by $\gamma(b,x) = d(b) + h(x)$. Since we see that $\gamma|_{I} = d$, we need to show $\gamma$ is a derivation. This can be done by using the $\square$-condition and (3) above for $f \in F$ with $n = \ar f$: calculating, we find
\begin{align*}
\gamma \circ F_{f}^{M} \big( \left\langle b_{1}, x_{1} \right\rangle, \ldots, \left\langle b_{n}, x_{n} \right\rangle  \big) &= \gamma \left( f^{I}(\vec{b}) + \sum_{s \in [n]^{\ast}} b(f,s)(\vec{x},\vec{b}) + T_{f}(\vec{x}), f^{Q}(\vec{x}) \right) \\
&= d \left(f^{I}(\vec{b}) \right) + d \left( \sum_{s \in [n]^{\ast}} b(f,s)(\vec{x},\vec{b})  \right) + d \circ T_{f}(\vec{x}) + h \left( f^{Q}(\vec{x}) \right) \\ 
&= \sum_{i \in [n]} \hat{a}(f,i)(x_{1},\ldots,d(b_{i}),\ldots,x_{n}) + d \circ T_{f}(\vec{x}) + h \left( f^{Q}(\vec{x}) \right) \\
&= \sum_{i \in [n]} a(f,i)( \left\langle b_{1}, x_{1} \right\rangle, \ldots,d(b_{i}),\ldots, \left\langle b_{n}, x_{n} \right\rangle ) \\
&\quad + \sum_{i \in [n]} \hat{a}(f,i)(x_{1},\ldots,h(x_{i}),\ldots,x_{n}) \\
&= \sum_{i \in [n]} a(f,i)( \left\langle b_{1}, x_{1} \right\rangle, \ldots,d(b_{i}),\ldots, \left\langle b_{n}, x_{n} \right\rangle ) \\
&\quad + \sum_{i \in [n]} a(f,i)(\left\langle b_{1}, x_{1} \right\rangle, \ldots, h(x_{i}), \ldots, \left\langle b_{1}, x_{1} \right\rangle ) \\
&=  \sum_{i \in [n]} a(f,i)( \left\langle b_{1}, x_{1} \right\rangle, \ldots,\gamma(b_{i},x_{i}),\ldots, \left\langle b_{n}, x_{n} \right\rangle ).
\end{align*}
Similarly, by using (1) and (2) above we find that
\begin{align*}
\gamma \left( \left\langle b, x \right\rangle + \left\langle c, y \right\rangle \right) = \gamma \left( \left\langle b + c + T_{+}(x,y), x + y \right\rangle \right) &= d(b) + d(c) + d \circ T_{+}(x,y) + h(x + y) \\
&= d(b) + d(c) + h(x) + h(y) \\
&= \gamma(b,x) + \gamma(c,y)
\end{align*}
and 
\begin{align*}
\gamma \left( r \cdot \left\langle b, x \right\rangle \right) = \gamma \left( r \cdot b + T_{r}(x), r \cdot x \right) = d(r \cdot b) + d \circ T_{r}(x) + h \left( r \cdot x \right) = r \cdot d(b) + r \cdot h(x) = r \cdot \gamma(b,x). \\
\end{align*}
We have shown $\im r = \ker \delta_{T}$.

$H^{2}_{\mathcal V}(Q,A^{I},\hat{\ast})$ : We first show $\sigma \circ \delta_{T} = 0$. Fix $d \in H^{1}(I,A^{I},\ast)^{\square}$ and note $\sigma \circ \delta_{T}(d) = [d \circ T \circ \pi]$. Define $s: M \rightarrow I$ by $s(m): l \circ m - m$ where $l: Q \rightarrow M$ is the lifting $l(x) = \left\langle 0, x \right\rangle \in I \rtimes_{T} Q = M$. If we write $m = \left\langle b, x \right\rangle$, then $s(m) = \left\langle 0, x \right\rangle - \left\langle b, x \right\rangle = - \left\langle b, 0 \right\rangle = \left\langle -b, 0 \right\rangle$.

Then by realization of the 2-cocycle, we have
\begin{align*}
d \circ T_{+}(\pi(m_{1}),\pi(m_{2})) &= d \Big( l \circ \pi(m_{1}) + l \circ \pi(m_{1}) - l \circ \pi(m_{1} + m_{2}) \Big) \\
&= d \Big( l \circ \pi(m_{1}) + l \circ \pi(m_{1}) - (m_{1} + m_{2}) \Big) + d \Big( (m_{1} + m_{2}) - l \circ \pi(m_{1} + m_{2}) \Big)   \\
&= d \circ s(m_{1}) +  d \circ s(m_{2}) - d \circ s( m_{1} + m_{2} )    
\end{align*}
and
\begin{align*}
d \circ T_{r}(\pi(m)) = d \Big( r \cdot ( l \circ \pi(m)) - l \circ \pi(r \cdot m) \Big) &= d \Big(  r \cdot ( l \circ \pi(m)) - r \cdot m \Big) + d \Big( r \cdot m - l \circ \pi(r \cdot m) \Big) \\
&= r \cdot (d \circ s (m)) - d \circ s (r \cdot m).
\end{align*}
For a multilinear operation $f \in F$ with $n =\ar f$, we similarly observe that
\begin{align*}
d \circ T_{f}(\pi(\vec{m})) &= d \big( f^{M}(l \circ \pi(\vec{m})) - l \circ \pi (f^{Q}(\vec{m})) \big) \\
&= d \big( f^{M}(l \circ \pi(\vec{m})) - f^{M}(\vec{m}) \big) - d \circ s ( f^{M}(\vec{m})) \\
&= d \left( - f^{I}(p_{1}(\vec{m})) - \sum_{s \in [n]^{\ast}} b(f,s)(\pi(\vec{m}), p_{1}(\vec{m}) ) \right) - d \circ s ( f^{M}(\vec{m}))   \\
&= \sum_{i \in [n]} \hat{a}(f,i)( \pi(m_{1}), \ldots, d(-p_{1}(m_{i})), \ldots, \pi(m_{n}) ) - d \circ s ( f^{M}(\vec{m}))   \\
&= \sum_{i \in [n]} \hat{a}(f,i)( \pi(m_{1}), \ldots, d \circ s (m_{i}), \ldots, \pi(m_{n}) ) - d \circ s ( f^{M}(\vec{m}))
\end{align*}
Altogether, we have shown $d \circ s: M \rightarrow A^{I}$ witnesses that $d \circ T \circ \pi$ is a 2-coboundary for the datum $(M,A^{I},\ast)$.

Now fix $[S] \in H^{2}_{\mathcal V}(Q,A^{I},\hat{\ast})$ such that $\sigma([S]) = [S \circ \pi] = 0$. By definition, there exists $h:M \rightarrow A^{I}$ such that 
\begin{enumerate}

	\item $S_{+}(\pi(m_{1}),\pi(m_{2})) = h(m_{1}) + h(m_{2}) - h(m_{1} + m_{2})$;
	
	\item $S_{r}(\pi(m)) = r \cdot h(m) - h(r \cdot m)$;
	
	\item $S_{f}(\pi(\vec{m})) = \sum_{i \in [n]} a(f,i) ( m_{1},\ldots,h(m_{i}),\ldots,m_{n} ) - h(f^{M}(\vec{m}))$   \quad \quad \quad \quad $( f \in F, n = \ar f)$.

\end{enumerate}
We can also assume that $0=S_{+}(x,0) = S(0,x) = S_{r}(0) = S_{f}(\vec{x})$ whenever some $x_{i}=0$. If we evaluate on the elements $m_{i} = \left\langle b_{i}, 0 \right\rangle$, we see that (1) - (3) above implies the restriction $h|_{I}: I \rightarrow A^{I}$ is a derivation since the left-hand side are all zero. We wish to show $h|_{I}$ satisfies the $\square$-condition according to Eq~\eqref{eqn:squarenull}. This follows by (3) since for $f \in F$ with $n = \ar f$, we have
\begin{align*}
h \left( b(f,i)(x_{1},\ldots,b_{i},\ldots,x_{n}), 0 \right) &= h \left( f^{M} \left( \left\langle 0, x_{1} \right\rangle, \ldots, \left\langle b_{i}, 0 \right\rangle, \ldots, \left\langle 0, x_{n} \right\rangle \right) \right) \\
&= \sum_{k \in [n]} a(f,i) ( \left\langle 0, x_{1} \right\rangle, \ldots, h(0,x_{k}) ,\ldots , \left\langle  b_{i}, 0 \right\rangle ,\ldots, \left\langle 0, x_{n} \right\rangle ) \\
&\quad - S_{f}(x_{1},\ldots,0,\ldots,x_{n})\\
&=  a(f,i) ( \left\langle 0, x_{1} \right\rangle, \ldots, h( b_{i}, 0), \ldots, \left\langle 0, x_{n} \right\rangle ) \\
&= \hat{a}(f,i)(x_{1},\ldots,h|_{I}(b_{i}),\ldots,x_{n})
\end{align*}
since $\im h \subseteq A^{I}$. The same reason also shows $h \left( b(f,i)(x_{1},\ldots,b_{i},\ldots,x_{n}), 0 \right) = 0$.

Let us observe from (1) above that $0 = S_{+}(0,x) = S_{+}(\pi(b,0),\pi(0,x)) = h(b,0) + h(0,x) - h ( \left\langle b, 0 \right\rangle + \left\langle 0, x \right\rangle ) = h(b,0) + h(0,x) - h(a,x)$; thus, we can write $h(a,x) = h(a,0) + h(0,x)$. If we define $\gamma: Q \rightarrow A^{I}$ by $\gamma(x):=h(0,x)$, then we see that for $f \in F$ with $n = \ar f$,
\begin{align*}
h|_{I} \circ T_{f}(\vec{x}) = h(T_{f}(\vec{x}),0) &= h(T_{f}(\vec{x}),f^{Q}(\vec{x}) ) - h(0,f^{Q}(\vec{x})) \\
&= h \left( f^{M} \left( \left\langle 0, x_{1} \right\rangle, \ldots, \left\langle 0, x_{n} \right\rangle \right) \right) - h(0,f^{Q}(\vec{x}))  \\
&= \sum_{i \in [n]} \hat{a}(f,i)(x_{1},\ldots,h(x_{i}),\ldots,x_{n} ) - S_{f}(\vec{x}) - \gamma \left( f^{Q}(\vec{x}) \right).
\end{align*}
A similar calculation for $T_{+}$ and $T_{r}$ contributes to show $[S] \in \im \delta_{T}$; altogether, $\im \delta_{T} = \ker \sigma$. The demonstration of the theorem is complete.
\end{proof}


\section{Discussion}\label{section:6}
\vspace{0.3cm}

It should be possible in both the affine and nonabelian cases to give a more concrete development of higher cohomologies by focusing on the intended interpretations.

\begin{problem}
Let $\mathcal V$ be a variety of multilinear expansions of $R$-modules and $(Q,I)$ be datum in $\mathcal V$. \`{A} la Holt \cite{holt}, is it possible to construct higher cohomologies for general datum $(Q,I)$ in $\mathcal V$ which are characterized by equivalence classes of certain ``decorated'' exact sequences ? Can these higher cohomologies be realized by equivalence classes of models of multisorted signatures extending 2-cocycles ?
\end{problem}

For our varieties of interests, is there an analogue of the description for groups of nonabelian $2^{\mathrm{nd}}$-cohomology by weak 2-functors between certain 2-categories? Is this instructive for determining what higher cohomology should be characterizing ?

\begin{problem}
For varieties of multilinear expansions of $R$-modules, is it possible to describe the general (abelan and nonabelian) cohomologies in the framework of weak functors and n-categories ?
\end{problem}

In particular, the suggestion here is that for these varieties there may be three equivalent descriptions of cohomology given be models of certain multisorted expansions of 2-cocycles together with their equational theories, isomorphism classes of certain decorated exact sequences and weak natural isomorphism classes of weak functors between certain n-categories.

\begin{acknowledgments}
The research in this manuscript was supported by NSF China Grant \#12071374. I would also like to express my sincere gratitude to Professor Shiqing Zhang and the mathematics department at Sichuan University for a very rewarding and productive research term. 
\end{acknowledgments}


\end{document}